\documentclass[a4paper, 10pt, final]{amsart}

\usepackage{xcolor}

\usepackage[bookmarks]{hyperref}
\hypersetup{
	allcolors=black,
	pdftitle={Sharelatex Example},
	bookmarks=true,
	pdfpagemode=FullScreen,
}
\hypersetup{
	pdftitle={Your title here},
	pdfauthor={Your name here},
	pdfsubject={Your subject here},
	pdfkeywords={keyword1, keyword2},
	bookmarksnumbered=true,     
	bookmarksopen=true,         
	bookmarksopenlevel=1,       
	colorlinks=true,            
	pdfstartview=Fit,           
	pdfpagemode=UseOutlines,    % this is the option you were lookin for
	pdfpagelayout=TwoPageRight
}

\usepackage{amsthm}

\theoremstyle{definition}

\usepackage[latin1]{inputenc}
\usepackage[T1]{fontenc}
\usepackage[english]{babel}
\usepackage{fullpage}
\usepackage{amsmath}
\usepackage{setspace}
\usepackage{amsfonts}
\usepackage{amssymb}
\usepackage{graphicx}
\usepackage{amscd}
\usepackage{mathrsfs}
\usepackage{fancybox}

\usepackage{lmodern}
\usepackage{graphicx}
\usepackage{fancyheadings}
\usepackage{chngcntr}
\counterwithout{figure}{section}
\usepackage{tikz}
\usetikzlibrary{calc}
\usepackage{pgfplots}
\usepgfplotslibrary{groupplots}
\usepgfplotslibrary{units}

\usetikzlibrary{arrows}
\usepackage{tkz-berge}

\usetikzlibrary{graphs,shapes.geometric,arrows,fit,matrix,positioning}

\usepackage[extra]{tipa}

\usepackage[OT2,T1]{fontenc}

\usepackage{amsmath, amsthm, amsfonts}
\cornersize{2}
\usepackage{graphicx}

\DeclareMathOperator{\Lie}{Lie}
\DeclareMathOperator{\weight}{weight}
\DeclareMathOperator{\id}{id}

\DeclareMathOperator{\an}{an}

\DeclareMathOperator{\iter}{iter}

\DeclareMathOperator{\Ad}{Ad}

\DeclareMathOperator{\reg}{reg}
\DeclareMathOperator{\depth}{depth}

\DeclareMathOperator{\LAE}{LAE}
\DeclareMathOperator{\un}{un}
\DeclareMathOperator{\DR}{DR}

\DeclareMathOperator{\har}{har}
\DeclareMathOperator{\Spec}{Spec}

\DeclareMathOperator{\Map}{Map}
\DeclareMathOperator{\gp}{gp}
\DeclareMathOperator{\unr}{unr}

\DeclareMathOperator{\Li}{Li}
\DeclareMathOperator{\loc}{loc}

\DeclareMathOperator{\crys}{crys}

\DeclareMathOperator{\elim}{elim}

\DeclareMathOperator{\Eucl}{Eucl}
\DeclareMathOperator{\Hom}{Hom}
\DeclareMathOperator{\comp}{comp}

\DeclareMathOperator{\KZ}{KZ}

\DeclareMathOperator{\dec}{dec}

\theoremstyle{definition}

\newtheorem{Theorem}{Theorem}[section]
\newtheorem{Proposition}[Theorem]{Proposition}

\newtheorem{Lemma}[Theorem]{Lemma}

\newtheorem{Definition}[Theorem]{Definition}

\newtheorem{Fact}[Theorem]{Fact}

\newtheorem{Lemma-Notation}[Theorem]{Lemma-Notation}

\newtheorem{Example}[Theorem]{Example}

\newtheorem{Problem}[Theorem]{Problem}

\newtheorem{Proposition-Definition}[Theorem]{Proposition-Definition}

\newtheorem{Remark}[Theorem]{Remark}
\newtheorem{Nota Bene}[Theorem]{Nota Bene}

\newcommand{\simlra}{\buildrel \sim \over \longrightarrow}

\input cyracc.def
\DeclareFontFamily{U}{russian}{}
\DeclareFontShape{U}{russian}{m}{n}
        { <5><6> wncyr5
        <7><8><9> wncyr7
        <10><10.95><12><14.4><17.28><20.74><24.88> wncyr10 }{}
\DeclareSymbolFont{Russian}{U}{russian}{m}{n}
\DeclareSymbolFontAlphabet{\mathcyr}{Russian}
\makeatletter
\let\@math@cyr\mathcyr
\renewcommand{\mathcyr}[1]{\@math@cyr{\cyracc #1}}
\makeatother
\newcommand{\sh}{\mathcyr{sh}} % Le produit shuffle

\makeatletter
\newcounter{subsubsubsection}[subsubsection]
\renewcommand\thesubsubsubsection{\thesubsubsection .\@alph\c@subsubsubsection}
\newcommand\subsubsubsection{\@startsection{subsubsubsection}{4}{\z@}%
                                     {-3.25ex\@plus -1ex \@minus -.2ex}%
                                     {1.5ex \@plus .2ex}%
                                     {\normalfont\normalsize\bfseries}}
\newcommand*\l@subsubsubsection{\@dottedtocline{3}{10.0em}{4.1em}}
\newcommand*{\subsubsubsectionmark}[1]{}
\makeatother

\usepackage[top=3.5cm, bottom=3.5cm, left=2.5cm, right=2.5cm]{geometry}

\author{David Jarossay}

\title{}

\address{Universit\'{e} de Strasbourg, CNRS, IRMA UMR 7501, F-67000 Strasbourg, France}
\email{jarossay@math.unistra.fr}

\begin{document}

\begin{center}
\begin{Large} \textbf{AN EXPLICIT THEORY OF $\pi_{1}^{\un,\crys}(\mathbb{P}^{1} - \{0,\mu_{N},\infty\})$}
\end{Large}
\\ \text{ }
\\ \begin{large}\textbf{V : The  point of view of ${\varprojlim}_{N}\text{ }\pi_{1}^{\un}(\mathbb{P}^{1} - \{0,\mu_{N},\infty\})$}
\\ \text{ }
\\ 
\textbf{V-1 : The Frobenius extended to $\pi_{1}^{\un,\DR}(\mathbb{P}^{1} - \{0,\mu_{p^{\alpha}N},\infty\})$}
\end{large}
\end{center}

\begin{abstract} Let $p$ a prime number. For all $N \in \mathbb{N}^{\ast}$ prime to $p$, let $k_{N}$ be a finite field of characteristic $p$ containing a primitive $N$-th root of unity. Let $X_{k_{N},N}=\text{ }\mathbb{P}^{1} - (\{0,\infty\} \cup \mu_{N})\text{ }/\text{ }k_{N}$. This work is an explicit theory of the crystalline pro-unipotent fundamental groupoid $(\pi_{1}^{\un,\crys})$ of $X_{k_{N},N}$. In the parts I to IV, we have considered each possible value of $N$ separately. The purpose of part V is to study the role of the morphisms relating $\pi_{1}^{\un}(\mathbb{P}^{1} - \{0,\mu_{N_{1}},\infty\})$ and $\pi_{1}^{\un}(\mathbb{P}^{1} - \{0,\mu_{N_{2}},\infty\})$ when $N_{1}$ divides $N_{2}$. In V-1, we specify this question to the theme of part I, the computation of the Frobenius. For any $N \in \mathbb{N}^{\ast}$, let $K_{N}=\mathbb{Q}_{p}(\xi_{N})$ where $\xi_{N}\in \overline{\mathbb{Q}_{p}}$ is a primitive $N$-th root of unity, and $X_{K_{N},N} = \mathbb{P}^{1} - (\{0,\infty\} \cup \mu_{N})\text{ }/\text{ }K_{N}$.
\newline For $N$ prime to $p$, we are used to view the Frobenius of $\pi_{1}^{\un,\crys}(X_{k_{N},N})$ as a structure on $\pi_{1}^{\un,\DR}(X_{K_{N},N})$.
\newline In V-1, we show that the Frobenius of  $\pi_{1}^{\un,\DR}(X_{K_{N},N})$, iterated $\alpha \in \mathbb{N}^{\ast}$ times, can be extended canonically as a structure of $\pi_{1}^{\un,\DR}(X_{K_{p^{\alpha}N},p^{\alpha}N})$. This allows to define generalizations of adjoint $p$-adic multiple zeta values associated with roots of unity of order $p^{\alpha}N$, and several related objects. This also gives a canonical framework to relate to each other the direct method of computation of the Frobenius of I-1 and the indirect methods of computation of the Frobenius of I-2 and I-3.
\end{abstract}

\maketitle

\noindent
\newline 
\newline 

\tableofcontents

\newpage

\section{Introduction}

\subsection{}

Let $p$ be a prime number. For each $N \in \mathbb{N}^{\ast}$ prime to $p$, let $q_{N} \in p^{\mathbb{N}}$ such that the finite field $k_{N}=\mathbb{F}_{q_{N}}$ contains a primitive $N$-th root of unity. Let $X_{k_{N},N}$ be the variety $\mathbb{P}^{1} - (\{0,\infty\} \cup \mu_{N})$ over $k_{N}$. The purpose of this work is to construct an explicit theory of the crystalline pro-unipotent fundamental groupoid $\pi_{1}^{\un,\crys}$ of $X_{k_{N},N}$, in the sense of \cite{Deligne}, \cite{CLS}, \cite{Shiho 1}, \cite{Shiho 2}, with a particular focus on the $p$-adic multiple zeta values at $N$-th roots of unity ($p$MZV$\mu_{N}$'s) where "at roots of unity" is omitted when $N=1$ ($p$MZV's), which are images of $p$-adic periods by a reduction map, are defined via the Frobenius and characterize the Frobenius. The complex multiple zeta values at roots of unity (MZV$\mu_{N}$'s) are Betti-De Rham periods of $\pi_{1}^{\un}(\mathbb{P}^{1} - \{0,\mu_{N},\infty\})$, and are the following numbers
\begin{equation}
\label{eq:multizetas} \zeta \big(\begin{array}{c} \tilde{\xi}_{N}^{j_{d}},\ldots,\tilde{\xi}_{N}^{j_{1}} \\ n_{d},\ldots,n_{1} \end{array} \big) =  \sum_{0<m_{1}<\ldots<m_{d}} \frac{\big( \frac{\tilde{\xi}_{N}^{j_{2}}}{\tilde{\xi}_{N}^{j_{1}}}\big)^{m_{1}} \ldots \big( \frac{1}{\tilde{\xi}_{N}^{j_{d}}}\big)^{m_{d}}}{m_{1}^{n_{1}} \ldots m_{d}^{n_{d}}} \in \mathbb{C}
\end{equation}
\noindent with $N \in \mathbb{N}^{\ast}$, $\tilde{\xi}_{N}$ a primitive $N$-th root of unity in $\overline{\mathbb{Q}} \hookrightarrow \mathbb{C}$, $n_{1},\ldots,n_{d} \in \mathbb{N}^{\ast}$, and $j_{1},\ldots,j_{d} \in \{1,\ldots,N\}$, such that $(\tilde{\xi}_{N}^{j_{d}},n_{d}) \not= (1,1)$.
\newline 
\newline Let now $\xi_{N}$ be a primitive $N$-th root of unity in $\overline{\mathbb{Q}_{p}}$ and $K_{N}= \mathbb{Q}_{p}(\xi_{N})\subset \overline{\mathbb{Q}_{p}}$. Let $\alpha \in \pm \mathbb{N}^{\ast} \cup \{\pm \infty\}$. One has for each $\alpha$ a family of numbers called $p$-adic multiple zeta values at roots of unity, denoted by  $\zeta_{p,\alpha}\big(\begin{array}{c} \xi_{N}^{j_{d}},\ldots,\xi_{N}^{j_{1}} \\ n_{d},\ldots,n_{1} \end{array} \big) \in K_{N}$, for all $n_{d},\ldots,n_{1} \in \mathbb{N}^{\ast}$, $j_{1},\ldots,j_{d} \in \mathbb{Z}/N\mathbb{Z}$. They have been defined in \cite{Deligne Goncharov}, \cite{Furusho 1}, \cite{Furusho 2}, \cite{Yamashita}, for certain particular values of $\alpha$, then a different convention involving the inverse of the Frobenius was adopted in \cite{U1}, \cite{U2}, and the definition was finally generalized to all values of $\alpha$ in \cite{I-1} and \cite{I-3}. They characterize the Frobenius of  $\pi_{1}^{\un,\crys}(X_{k_{N},N})$ at base-points $(-\vec{1}_{1},\vec{1_{0}})$ iterated $\alpha$ times. By extending the conjecture in \cite{Deligne Goncharov} \S5.28, for each $\alpha$ and $N$, the ideal of the algebraic relations satisfied by the numbers $\zeta_{p,\alpha}$ is conjecturally generated by the algebraic relations satisfied by their complex analogues (\ref{eq:multizetas}) and the vanishing of the $p$-adic analogue of $2i\pi$, which implies in particular that, for all $s \in \mathbb{N}^{\ast}$, we have $\zeta_{p,\alpha}(2n) = 0$ (where, when $N=1$,
$\big(n_{d},\ldots,n_{1} \big) = \big(\begin{array}{c} 1,\ldots,1 \\ n_{d},\ldots,n_{1} \end{array}\big)$).
\newline
\newline For all $N \in \mathbb{N}^{\ast}$, let $X_{K_{N},N} = \mathbb{P}^{1} - \{0,\mu_{N},\infty\}\text{ }/\text{ }K_{N}$. When $N$ is prime to $p$, following \cite{Deligne}, \S13, we view $\pi_{1}^{\un,\crys}(X_{k_{N},N})$ as $\pi_{1}^{\un,\DR}(X_{K_{N},N})$ equipped with the Frobenius. We have computed the Frobenius, in particular $p$MZV$\mu_{N}$'s, in part I \cite{I-1} \cite{I-2} \cite{I-3}. We have used our explicit computation to study explicitly $p$MZV$\mu_{N}$'s in part II \cite{II-1} \cite{II-2} \cite{II-3}, part III \cite{III-1} \cite{III-2}, part IV \cite{IV-1} \cite{IV-2}. This part V is the last one of this theory.

\subsection{} The initial motivation for this part V is the following. Until now, we have considered the groupoids $\pi_{1}^{\un}(\mathbb{P}^{1} - \{0,\mu_{N},\infty\})$ for each $N$ separately.  If $N_{1},N_{2} \in \mathbb{N}^{\ast}$ are such that $N_{1}|N_{2}$, we have a morphism of groupoids 
\begin{equation} \label{eq:morphisms 0} \pi_{1}^{\un}(\mathbb{P}^{1} - \{0,\mu_{N_{2}},\infty\})|_{B_{N_{1},N_{2}}} \rightarrow \pi_{1}^{\un}(\mathbb{P}^{1} - \{0,\mu_{N_{1}},\infty\})|_{B_{N_{1},N_{2}}} 
\end{equation}
\noindent where $|_{B_{N_{1},N_{2}}}$ means the restriction to the set $B_{N_{1},N_{2}}$ of couples of base-points which are shared by $\pi_{1}^{\un}(\mathbb{P}^{1} - \{0,\mu_{N_{2}},\infty\})$ and $\pi_{1}^{\un}(\mathbb{P}^{1} - \{0,\mu_{N_{1}},\infty\})$. If $N_{1}$ and $N_{2}$ are prime to $p$, for convenient base-points $x,y$ over $\overline{\mathbb{F}_{p}}$, we have a morphism
\begin{equation}
\label{eq:morphisms}
\pi_{1}^{\un,\crys}(\mathbb{P}^{1} - \{0,\mu_{N_{2}},\infty\},y,x) \rightarrow  \pi_{1}^{\un,\crys}(\mathbb{P}^{1} - \{0,\mu_{N_{1}},\infty\},y,x)
\end{equation}
\noindent which means that, if $\tilde{x}$ and $\tilde{y}$ are convenient lifts of $x$ and $y$ over $W(\overline{\mathbb{F}_{p}})$, we have a morphism of pro-affine schemes between the De Rham realizations $\pi_{1}^{\un,\DR}(\mathbb{P}^{1} - \{0,\mu_{N_{2}},\infty\},\tilde{y},\tilde{x}) \rightarrow  \pi_{1}^{\un,\DR}(\mathbb{P}^{1} - \{0,\mu_{N_{1}},\infty\},\tilde{y},\tilde{x})$, which is compatible with the Frobenius. In particular, the notation $\zeta_{p,\alpha}$ above is consistent : two equal sequences $(\xi_{N_{1}}^{j_{d}},\ldots,\xi_{N_{1}}^{j_{1}})=(\xi_{N_{2}}^{j'_{d}},\ldots,\xi_{N_{2}}^{j'_{1}})$ give the same number. The morphisms (\ref{eq:morphisms 0}) form a projective system when $(N_{1},N_{2})$ varies. Moreover, each of the morphisms (\ref{eq:morphisms 0}) has a natural section, and these sections form an inductive system when $(N_{1},N_{2})$ varies. It seems reasonable to expect that these objects have a certain role to play in this theory. The subject of this part V is to realize this expectation.
In this V-1, we specialize this problem to the topic of part I : what we want is to bring together the computation of the Frobenius and the morphisms (\ref{eq:morphisms 0}).

\subsection{} Let us see what is at stake by the question above. First, the computation of the Frobenius and the morphisms (\ref{eq:morphisms}), as well as their sections, actually commute. Indeed, the morphisms (\ref{eq:morphisms}) actually commute with the canonical presentations of the groupoids $\pi_{1}^{\un,\DR}(X_{K_{N},N})$ : one has, for each $N$, a canonical base-point $\omega_{\DR}$ of $\pi_{1}^{\un,\DR}(X_{K_{N},N})$, such that the groupoid $\pi_{1}^{\un}(\mathbb{P}^{1} - \{0,\mu_{N},\infty\})$ can be described functorially in terms of its fiber at $\omega_{\DR}$, and the Hopf algebra of global functions on the affine group scheme $\pi_{1}^{\un,\DR}(\mathbb{P}^{1} - \{0,\mu_{N},\infty\},\omega_{\DR})$ is the shuffle Hopf algebra over the alphabet $e_{0\cup \mu_{N}}=\{e_{0},e_{\xi_{N}^{1}},\ldots,e_{\xi_{N}^{N}}\}$, generated as a $\mathbb{Q}$-vector space by words over $\{e_{0},e_{\xi_{N}^{1}},\ldots,e_{\xi_{N}^{N}}\}$ ; this alphabet represents a basis of $H^{1,\DR}(\mathbb{P}^{1} - \{0,\mu_{N},\infty\})$, and the canonical choice of basis, which is used all the time, is given by the differential forms $\omega_{0}(z) = \frac{dz}{z}$ and 
$\omega_{\xi_{N}^{j}}(z) = \frac{dz}{z-\xi_{N}^{j}}$, $j=1,\ldots,N$.
Since, when $N_{1}|N_{2}$, we have an inclusion $\{\omega_{0},\omega_{\xi_{N_{1}}^{1}},\ldots,\omega_{\xi_{N_{1}}^{N_{1}}}\} \subset \{\omega_{0},\omega_{\xi_{N_{2}}^{1}},\ldots,\omega_{\xi_{N_{2}}^{N_{2}}}\}$, and since the Frobenius is functorial, the computation of the Frobenius of $\pi_{1}^{\un,\DR}(X_{K_{N_{1}},N_{1}})$ can be viewed canonically as a computation within $\pi_{1}^{\un,\DR}(X_{K_{N_{2}},N_{2}})$ and, conversely, the computation of the Frobenius of  $\pi_{1}^{\un,\DR}(X_{K_{N_{2}},N_{2}})$ restricts to the computation of the Frobenius of  $\pi_{1}^{\un,\DR}(X_{K_{N_{1}},N_{1}})$. This shows that a part of the question to bring together the computation of the Frobenius and the morphisms (\ref{eq:morphisms 0}) is trivial.
\newline 
\newline Let us now consider $\pi_{1}^{\un,\DR}(X_{K_{N'},N'})$ with $N'$ non-prime to $p$, and divisible by $N$ which is prime to $p$, say $N=N'p^{-v_{p}(N')}$. Note that since $K_{N'}$ is ramified, it would make no sense to apply to $\pi_{1}^{\un,\DR}(X_{K_{N'},N'})$ the construction of the Frobenius structure of \cite{Deligne}, \S11, and that the reduction modulo $p$ of $X_{K_{N'},N'}$ is equal to the one of $X_{K_{N},N}$. One has a morphism 
\begin{equation}
\label{eq:morphisms 3} \pi_{1}^{\un,\DR}(X_{K_{N'},N'})|_{B_{N',N}} 
\rightarrow \pi_{1}^{\un,\DR}(X_{K_{N},N})|_{B_{N',N}}
\end{equation}
\noindent Let us consider the Frobenius of $\pi_{1}^{\un,\DR}(X_{K_{N},N})$ iterated $\alpha \in \mathbb{N}^{\ast} \cup - \mathbb{N}^{\ast}$ times, with $N'=Np^{|\alpha|}$. There is a basic reason why our study of $\pi_{1}^{\un,\crys}(X_{K_{N},N})$ is somewhat related to the morphism (\ref{eq:morphisms 3}). Let $W_{N}=W(k_{N})$ be the ring of Witt vectors of $k_{N}$, and let $X_{W_{N},N} = \mathbb{P}^{1} - \{0,\mu_{N},\infty\} / W_{N}$, which has the global lift of Frobenius $z \mapsto z^{p}$. For $\alpha \in \mathbb{N}^{\ast}$, the map $z \mapsto z^{p^{\alpha}}$ sends the groupoid $\pi_{1}^{\un,\DR}(X_{K_{N},N})$, with its connection $\nabla_{\KZ}^{\mu_{N}}$, to a variant on $X_{K_{N},N}^{(p^{\alpha})}=X_{W_{N},N}^{(p^{\alpha})} \times_{\Spec(W_{N})} \Spec(K_{N})$ where $X_{W_{N},N}^{(p^{\alpha})}$ is equal to the pullback of $X_{W_{N},N}$ by the $\alpha$-th iteration of the Frobenius automorphism $\sigma : W_{N} \rightarrow W_{N}$ of $W_{N}$. Whereas $\nabla_{\KZ}^{\mu_{N}}$ is the map $f \mapsto f^{-1}(df - \sum_{i=0}^{N} e_{z_{i,N}} \frac{dz}{z-z_{i,N}})$, with $z_{0,N}=0$ and $z_{i,N}= \xi_{N}^{i}$, $i=1,\ldots,N$, the variant $\nabla_{\KZ}^{\mu_{N}^{(p^{\alpha})}}$ is $f \mapsto f^{-1}(df - \sum_{i=0}^{N} e_{z_{i,N}^{(p^{\alpha})}} \frac{p^{\alpha}z^{p^{\alpha}-1}dz}{z^{p^{\alpha}}-z_{i}^{p^{\alpha}}})$ since, the $z_{i,N}$'s being roots of unity, we have $\sigma^{\alpha}(z_{i,N}) = z_{i,N}^{p^{\alpha}}$ for all $i \in \{1,\ldots,N\}$. The Frobenius of $\pi_{1}^{\un,\DR}(X_{K_{N},N})$ is a natural isomorphism between
$(\pi_{1}^{\un,\DR}(X_{K_{N},N}),\nabla_{\KZ}^{\mu_{N}})$ and 
$(\pi_{1}^{\un,\DR}(X^{(p^{\alpha})}_{K_{N},N}),\nabla_{\KZ}^{\mu_{N}^{(p^{\alpha})})}$. By writing :
\begin{equation} \label{eq:differential forms modified} \frac{d(z^{p^{\alpha}})}{z^{p^{\alpha}}-z_{i,N}^{p^{\alpha}}} = \sum_{\rho \in \mu_{p^{\alpha}}(\overline{\mathbb{Q}_{p}})}  \frac{z_{i,N} dz}{z - \rho z_{i,N}}
= \sum_{\{ \tilde{\rho} \in \mu_{p^{\alpha}N}(\overline{\mathbb{Q}_{p}})\text{ | } |\tilde{\rho} - z_{i,N}|_{p}<1\}} \frac{z_{i,N} dz}{z - \tilde{\rho}}
\end{equation}
\noindent we see that the Frobenius of $\pi_{1}^{\un,\DR}(X_{K_{N},N})$ can be viewed as an isomorphism between two quotients of the groupoid $\pi_{1}^{\un,\DR}(X_{K_{p^{\alpha}N},p^{\alpha}N})$ with its connection $\nabla_{\KZ}^{\mu_{p^{\alpha}N}}$.
\newline 
\newline This suggests to study, more generally, how $\pi_{1}^{\un,\DR}(X_{K_{p^{|\alpha|}N},p^{|\alpha|}N})$ is connected to $\pi_{1}^{\un,\crys}(X_{K_{N},N})$ where $\alpha \in \mathbb{N}^{\ast} \cup - \mathbb{N}^{\ast}$ is the number of iterations of the Frobenius and $N \in \mathbb{N}^{\ast}$ is prime to $p$.
\newline We are going to adopt a radical form of this question : reformulate, or enrich, the Frobenius structure of $(\pi_{1}^{\un,\DR}(\mathbb{P}^{1} - \{0,\mu_{N},\infty\}),\nabla^{\mu_{N}}_{\KZ})$, iterated $\alpha$ times, with its explicit formulas from part I, into a structure on $(\pi_{1}^{\un,\DR}(\mathbb{P}^{1} - \{0,\mu_{p^{\alpha}N},\infty\}),\nabla_{\KZ}^{\mu_{p^{\alpha}N}})$, with explicit formulas within $(\pi_{1}^{\un,\DR}(\mathbb{P}^{1} - \{0,\mu_{p^{\alpha}N},\infty\}),\nabla_{\KZ}^{\mu_{p^{\alpha}N}})$.

\subsection{} Given the question above, let us review how we have computed the Frobenius in part I \cite{I-1}, \cite{I-2}, \cite{I-3}. The results of the computation were expressed in terms of weighted multiple harmonic sums at roots of unity : for $n,d \in \mathbb{N}^{\ast}$, $s_{d},\ldots,s_{1} \in \mathbb{N}^{\ast}$, $j_{1},\ldots,j_{d+1} \in \mathbb{Z}/N\mathbb{Z}$,

\begin{equation}
\label{eq: multiple harmonic sums}
\har_{m} \big( \begin{array}{cc} \xi_{N}^{j_{d+1}},\ldots,\xi_{N}^{j_{1}} \\ n_{d},\ldots,n_{1} 
\end{array} \big) = m^{n_{d}+\ldots+n_{1}} \sum_{0<m_{1}<\ldots<m_{d}<m} \frac{\big( \frac{\xi_{N}^{j_{2}}}{\xi_{N}^{j_{1}}}\big)^{m_{1}} \ldots \big( \frac{\xi_{N}^{j_{d+1}}}{\xi_{N}^{j_{d}}}\big)^{m_{d}} \big(\frac{1}{\xi_{N}^{j_{d+1}}}\big)^{n}} {m_{1}^{n_{1}} \ldots m_{d}^{n_{d}}}
\end{equation}

\noindent These are, up to a multiplicative factor in $m^{\mathbb{Z}}$, the coefficients of the power series expansions at $0$ of $p$-adic hyperlogarithms, defined as iterated integrals of $p^{\alpha}\omega_{0}(z),p^{\alpha}\omega_{\xi_{N}^{1}}(z),\ldots,p^{\alpha}\omega_{\xi_{N}^{N}}(z)$, i.e. as flat sections of $\nabla^{\mu_{N}}_{\KZ}$, by Coleman integration in the sense of \cite{Coleman}, \cite{Besser}, \cite{Vologodsky}.
\newline 
\newline In \cite{I-1}, we have solved directly the differential equation satisfied by the Frobenius. A key intermediate object for the computation was a regularized version of iterated integrals of the following differential forms
\begin{equation} \label{eq:differential forms}
p^{\alpha}\omega_{0}(z) = \omega_{0}(z^{p^{\alpha}}),p^{\alpha}\omega_{\xi_{N}^{1}}(z),\ldots,p^{\alpha}\omega_{\xi_{N}^{N}}(z),\omega_{\xi_{N}^{1}}(z^{p^{\alpha}}),\ldots,\omega_{\xi_{N}^{N}}(z^{p^{\alpha}})
\end{equation}
\noindent The regularized iterated integrals of (\ref{eq:differential forms}) are overconvergent analytic functions on the rigid analytic space $(\mathbb{P}^{1,\text{an}} - \cup_{i=1}^{N} \{z \text{ }|\text{ }|z-\xi_{N}^{i}|_{p}<1\})/ K_{N}$.
The sequences of coefficients of the power series expansions of the iterated integrals of the forms (\ref{eq:differential forms}) are, up to a simple multiplicative factor, the "multiple harmonic sums with congruences" modulo powers of $p$, at roots of unity of order $N$ prime to $p$ : namely, for any $n \in \mathbb{N}^{\ast}$, for any index as in (\ref{eq: multiple harmonic sums}), $a \in \mathbb{N}^{\ast}$ and $I \subset \{1,\ldots,d\}$ :
\begin{equation} \label{eq:multiple harmonic sums with congruences}
\har^{\mod p^{a}}_{n,I} \big( \begin{array}{cc} \xi_{N}^{j_{d+1}},\ldots,\xi_{N}^{j_{1}} \\ s_{d},\ldots,s_{1}
\end{array} \big) = n^{s_{d}+\ldots+s_{1}} \sum_{\substack{0<n_{1}<\ldots<n_{d}<n \\ \text{ for all i in I,}\text{ }n_{i} \equiv n_{i+1} \mod p^{a}}} \frac{\big( \frac{\xi_{N}^{j_{2}}}{\xi_{N}^{j_{1}}}\big)^{n_{1}} \ldots \big( \frac{\xi_{N}^{j_{d+1}}}{\xi_{N}^{j_{d}}}\big)^{n_{d}} \big(\frac{1}{\xi_{N}^{j_{d+1}}}\big)^{n}} {n_{1}^{s_{1}} \ldots n_{d}^{s_{d}}}
\end{equation}

\noindent In \cite{I-2} and \cite{I-3}, we have made a distinction between three frameworks of computations : 
\newline $\bullet$ the "framework $\int_{0}^{1}$" which involves the scheme $\pi_{1}^{\un,\DR}(X_{K_{N},N},-\vec{1}_{1},\vec{1}_{0})$ and operations on it.
\newline $\bullet$ the "framework $\int_{0}^{z<<1}$" which involves multiple harmonic sums viewed as coefficients of power series expansions of hyperlogarithms and operations on $\pi_{1}^{\un,\DR}(X_{K_{N},N})$.
\newline $\bullet$ the "framework $\Sigma$" which involves multiple harmonic sums viewed as elementary finite iterated sums via their formula (\ref{eq: multiple harmonic sums}).
\newline In \cite{I-2} and \cite{I-3}, we have showed using \cite{I-1} that the differential equation of the Frobenius had a simplification in a certain "limit", in which it was equivalent to a relation between weighted multiple harmonic sums and adjoint $p$-adic multiple zeta values at roots of unity (abreviated $\Ad p$MZV$\mu_{N}$'s, and denoted by $\zeta_{p,\alpha}^{\Ad}(w)$), which are variants of $p$MZV$\mu_{N}$'s equivalent to them up to a polynomial base-change defined over $\mathbb{Q}$. We called  \emph{harmonic Frobenius} (standing for "incarnation of the Frobenius which is natural from the point of view of multiple harmonic sums"), the corresponding "limit" of the Frobenius. The harmonic Frobenius is sufficient to reconstruct the whole of the Frobenius.
\newline We have computed indirectly the harmonic Frobenius, i.e. we have expressed it in two different ways, one of them explicit in terms of multiple harmonic sums in the framework $\Sigma$, the other one in terms of $\Ad p$MZV$\mu_{N}$'s in the framework $\int_{0}^{z<<1}$ or $\int_{0}^{1}$ and we showed that the two expressions can be identified.
\newline The results were expressed by the following objects. We have defined "\emph{$p$-adic harmonic Ihara actions}", $\circ_{\har}^{\smallint_{0}^{1}}$, $\circ_{\har}^{\smallint_{0}^{z<<1}}$, $\circ_{\har}^{\Sigma}$. These are continuous group actions, related to the Ihara bracket on $\Lie \pi_{1}^{\un,\DR}(X_{K},\vec{1}_{0})$ , of a complete topological subgroup of $\pi_{1}^{\un,\DR}(X_{K},\vec{1}_{0})(K)$ equipped with an appropriate topology, on a complete space containing as an element the generating sequence of all multiple harmonic sums. We have defined "\emph{maps of comparison between series and integrals}" $\comp^{\Sigma \leftarrow \smallint}$ and $\comp^{\smallint \leftarrow \Sigma}$ relating the harmonic Ihara actions to each other. We also have defined in \cite{I-3} two maps $\iter_{\har}^{\smallint}$ and $\iter_{\har}^{\Sigma}$ of "iterations of the harmonic Frobenius". We have written fully explicit formulas for all these objects. The explicit formulas for $p$MZV$\mu_{N}$'s amount then to say that the generating series of $\Ad p$MZV$\mu_{N}$'s is the image of the one of prime weighted multiple harmonic sums by the map $\comp^{\smallint \leftarrow \Sigma}$.
\newline 
\newline The flat sections of the connection $\nabla_{\KZ}^{\mu_{p^{\alpha}N}}$ on $\pi_{1}^{\un,\DR}(X_{K_{p^{\alpha}N},p^{\alpha}N})$ are the iterated integrals of the differential forms $\frac{dz}{z}$ and $\frac{dz}{z - \rho\xi}$, with $\rho$ a $p^{\alpha}$-th root of unity and $\xi$ an $N$-th root of unity. They are called hyperlogarithms as in the case of $\pi_{1}^{\un,\DR}(X_{K_{N},N})$. The coefficients of their power series expansions at $0$ are up to a multiplicative factor in $m^{\mathbb{Z}}$, the weighted multiple harmonic sums at roots of unity of order $p^{\alpha}N$ :
 
\begin{equation}
\label{eq:ramified multiple harmonic sums}
\har_{m} \big( \begin{array}{cc} \xi_{N}^{j_{d+1}}\rho_{a}^{j'_{d+1}},\ldots,\xi_{N}^{j_{1}}\rho_{a}^{j'_{1}} \\ n_{d},\ldots,n_{1}
\end{array} \big) = m^{n_{d}+\ldots+n_{1}} \sum_{0<m_{1}<\ldots<m_{d}<m} \frac{\big( \frac{\xi_{N}^{j_{2}}\rho_{a}^{j'_{2}}}{\xi_{N}^{j_{1}}\rho_{a}^{j'_{1}}}\big)^{m_{1}} \ldots \big( \frac{\xi_{N}^{j_{d+1}}\rho_{a}^{j'_{d+1}}}{\xi_{N}^{j_{d}}\rho_{a}^{j'_{d}}}\big)^{m_{d}} \big(\frac{1}{\xi_{N}^{j_{d+1}}\rho_{a}^{j'_{d+1}}}\big)^{m}}{m_{1}^{n_{1}} \ldots m_{d}^{n_{d}}}
\end{equation}
 
\noindent with $\alpha \in \mathbb{N}^{\ast}$ and $N \in \mathbb{N}^{\ast}$ prime to $p$, and $\rho_{p^{\alpha}}$ a primitive $p^{\alpha}$-th root of unity in $K_{p^{\alpha}N}$, $m \in \mathbb{N}^{\ast}$, $\big( \begin{array}{cc} \xi_{N}^{j_{d+1}},\ldots,\xi_{N}^{j_{1}} \\ n_{d},\ldots,n_{1} 
\end{array} \big)$ as in (\ref{eq: multiple harmonic sums}) and $j'_{1},\ldots,j'_{d+1} \in \mathbb{Z}/p^{\alpha}\mathbb{Z}$.

\subsection{}

We are going to show that there exists a canonical relation between $\pi_{1}^{\un,\DR}(X_{K_{p^{\alpha}N},p^{\alpha}N})$ and $\pi_{1}^{\un,\DR}(X_{K_{N},N}^{(p^{\alpha})})$, defined over $K_{p^{\alpha}N}$, which extends the Frobenius of  $\pi_{1}^{\un,\DR}(X_{K_{N},N})$ (which is an isomorphism between $\pi_{1}^{\un,\DR}(X_{K_{N},N})$ and $\pi_{1}^{\un,\DR}(X_{K_{N},N}^{(p^{\alpha})})$). The first immediate application is to define a notion of adjoint $p$-adic multiple zeta values at $p^{\alpha}N$-th roots of unity ($\Ad$ $p$MZV$\mu_{p^{\alpha}N}$'s). We are also going to show that all the tools of computation of the Frobenius of part I, reviewed in \S1.4, can be generalized to this framework, giving a computation of the Frobenius extended to $\pi_{1}^{\un,\DR}(X_{K_{p^{\alpha}N},p^{\alpha}N})$, and in particular of ($\Ad$ $p$MZV$\mu_{p^{\alpha}N}$'s).
\newline 
\newline \textbf{Theorem-Definition V-1.a : the Frobenius extended to $\pi_{1}^{\un,\DR}(\mathbb{P}^{1} - \{0,\mu_{p^{\alpha}N},\infty\})$ ; adjoint $p$-adic multiple zeta values at $p^{\alpha}N$-th roots of unity}
\newline There exist a unique 
element of $\mathcal{A}^{\dagger}(U_{N}) \otimes_{K_{N}} K_{p^{\alpha}N} \langle \langle e_{0},(e_{\rho_{p^{\alpha}}^{i}\xi_{N}^{j}})_{\substack{i=1,\ldots,p^{\alpha} \\ j=1,\ldots,N}}\rangle\rangle \times K_{p^{\alpha}N} \langle \langle e_{0},(e_{\rho_{p^{\alpha}}^{i}\xi_{N}^{j}})_{\substack{i=1,\ldots,p^{\alpha} \\ j=1,\ldots,N}}\rangle\rangle^{N}$, denoted by $(\Li_{p,\alpha}^{\dagger,\mu_{p^{\alpha}N}},\psi^{\mu_{p^{\alpha}N}}_{\xi_{N}^{1}},\ldots,\psi^{\mu_{p^{\alpha}N}}_{\xi_{N}^{N}})$, which satisfies the following extension of the differential equation of the Frobenius 
$$ d\Li_{p,\alpha}^{\dagger,\mu_{p^{\alpha}N}} = \bigg( p^{\alpha}\omega_{0}(z)e_{0} + \sum_{\rho,\xi} \omega_{\rho\xi}(z) e_{\rho\xi} \bigg)\Li_{p,\alpha}^{\dagger,\mu_{p^{\alpha}N}} - \Li_{p,\alpha}^{\dagger,\mu_{p^{\alpha}N}} \bigg( p^{\alpha}\omega_{0}(z)e_{0} + \sum_{j=1}^{N} \omega_{\xi^{p^{\alpha}}}(z^{p^{\alpha}})\psi_{\xi_{N}^{j}}^{\mu_{p^{\alpha}N}} \bigg) $$
\noindent and $\Li_{p,\alpha}^{\dagger,\mu_{p^{\alpha}N}}(0)=1$ and a certain regularity property of $\Li_{p,\alpha}^{\dagger,\mu_{p^{\alpha}N}}$ (see \S4.3).
\newline Moreover, $(\Li_{p,\alpha}^{\dagger,\mu_{p^{\alpha}N}},\psi^{\mu_{p^{\alpha}N}}_{\xi_{N}^{1}},\ldots,\psi^{\mu_{p^{\alpha}N}}_{\xi_{N}^{N}})$ satisfies an extension of the bounds of valuation and the formulas of the direct computation proved in \cite{I-1}. For all words $w$, we denote by
$$ \zeta_{p,\alpha}^{\Ad}(w) = \psi_{1}^{\mu_{p^{\alpha}N}}[w] $$
\noindent and call these numbers the adjoint $p$-adic multiple zeta values at $p^{\alpha}N$-th roots unity ($\Ad$ $p$MZV$\mu_{p^{\alpha}N}$'s).
\newline This also enables to define (non-adjoint) $p$-adic multiple zeta values at $p^{\alpha}N$-th roots unity ($p$MZV$\mu_{p^{\alpha}N}$'s) denoted by $\zeta_{p,\alpha}(w)$.
\newline 
\newline 
\textbf{Theorem-Definition V-1.b : the harmonic Frobenius, harmonic Ihara actions, iterations of the harmonic Frobenius and comparison maps extended to $\pi_{1}^{\un,\DR}(\mathbb{P}^{1} - \{0,\mu_{p^{\alpha}N},\infty\})$}
\newline There exist
\newline a) harmonic Ihara actions at $p^{\alpha}N$-th roots of unity
\noindent\newline 
$$ \circ_{\har}^{\int_{0}^{1}} :
K_{p^{\alpha}N}\langle\langle e_{\{0\} \cup \mu_{p^{\alpha}N}} \rangle\rangle_{S}
\times \Map(\mathbb{N},
K_{p^{\alpha}N} \langle\langle e_{\{0\} \cup \mu_{N}} \rangle\rangle_{\har}^{\smallint_{0}^{1}} )
\rightarrow 
\Map(\mathbb{N},K_{p^{\alpha}N}\langle\langle  e_{\{0\} \cup \mu_{p^{\alpha}N}}  \rangle\rangle_{\har}^{\smallint_{0}^{1}} ) $$
$$ \circ_{\har}^{\smallint_{0}^{z<<1}} :
K_{p^{\alpha}N}\langle\langle e_{\{0\} \cup \mu_{p^{\alpha}N}} \rangle\rangle_{S}
\times \Map(\mathbb{N},
K_{p^{\alpha}N} \langle\langle e_{\{0\} \cup \mu_{N}} \rangle\rangle_{\har}^{\smallint_{0}^{z<<1}} )
\rightarrow 
\Map(\mathbb{N},K_{p^{\alpha}N}\langle\langle  e_{\{0\} \cup \mu_{p^{\alpha}N}}  \rangle\rangle_{\har}^{\smallint_{0}^{z<<1}} ) $$
$$ \circ_{\har}^{\Sigma} : 
K_{p^{\alpha}N}\langle\langle e_{\{0\} \cup \mu_{p^{\alpha}N}} \rangle\rangle_{S}
\times \Map(\mathbb{N},
K_{p^{\alpha}N} \langle\langle e_{\{0\} \cup \mu_{N}} \rangle\rangle_{\har}^{\Sigma} )
\rightarrow 
\Map(\mathbb{N},K_{p^{\alpha}N}\langle\langle  e_{\{0\} \cup \mu_{p^{\alpha}N}}  \rangle\rangle_{\har}^{\Sigma}) $$
\newline whose restrictions to $K_{p^{\alpha}N}\langle\langle e_{\{0\} \cup \mu_{p^{\alpha}N}} \rangle\rangle_{S}
\times \Map(\mathbb{N},
K_{p^{\alpha}N} \langle\langle e_{\{0\} \cup \mu_{N}} \rangle\rangle_{\har}$ are equal to the harmonic Ihara actions defined in part I, given by explicit formulas
\newline b) comparison maps between sums and integrals at $p^{\alpha}N$-th roots of unity
$$ \comp^{\Sigma \leftarrow \smallint} K_{p^{\alpha}N} \langle\langle e_{0 \cup \mu_{p^{\alpha}N}} \rangle\rangle_{\har}^{\Sigma})_{S} \rightarrow K_{p^{\alpha}N} \langle\langle e_{0 \cup \mu_{p^{\alpha}N}} \rangle\rangle_{S} $$
$$ \comp^{\Sigma \leftarrow \smallint} : K_{p^{\alpha}N} \langle\langle e_{0 \cup \mu_{p^{\alpha}N}} \rangle\rangle_{\har}^{\Sigma})_{S} \rightarrow K_{p^{\alpha}N} \langle\langle e_{0 \cup \mu_{p^{\alpha}N}} \rangle\rangle_{S} $$
\noindent whose restrictions to $K_{N}\langle \langle e_{0 \cup \mu_{N}} \rangle\rangle_{S}$ are the comparison maps defined in part I, and given by explicit formulas,
\newline c) maps of iteration of the harmonic Frobenius at $p^{\alpha}N$-th roots of unity :
$$ (\widetilde{\text{iter}}_{\har}^{\smallint_{0}^{1}})^{\textbf{a},\Lambda} : K_{p^{\alpha}N}\langle\langle e_{0 \cup \mu_{p^{\alpha}N}} \rangle\rangle_{S} \rightarrow K_{p^{\alpha}N}[[\Lambda^{\textbf{a}}]][\textbf{a}](\Lambda)\langle\langle e_{0 \cup \mu_{p^{\alpha}N}} \rangle\rangle^{\smallint_{0}^{1}}_{\har} $$
$$ (\widetilde{\text{iter}}_{\har}^{\Sigma})^{\textbf{a},\Lambda} : (K_{p^{\alpha}N} \langle\langle e_{0 \cup \mu_{p^{\alpha}N}}\rangle \rangle_{\har}^{\Sigma})_{S} \rightarrow K_{p^{\alpha}N}[[\Lambda^{\textbf{a}}]][\textbf{a}](\Lambda)\langle\langle e_{0 \cup \mu_{p^{\alpha}N}} \rangle\rangle^{\Sigma}_{\har} $$
\noindent whose restrictions to $K_{N}\langle\langle e_{0 \cup \mu_{N}}\rangle\rangle_{S}$, resp. $K_{N} \langle\langle e_{0 \cup \mu_{N}} \rangle \rangle_{\har}^{\Sigma})_{S}$ are the maps of iteration defined in part I, and given by explicit formulas,
\newline such that we have a simple isomorphism between $\circ_{\har}^{\smallint_{0}^{1}}$ and $\circ_{\har}^{\smallint_{0}^{z<<1}}$, and the equalities
$$ \har_{p^{\alpha}\mathbb{N}}^{\mu_{p^{a}N}} = \psi_{1}^{p^{\alpha}N} \text{ } \circ_{\har}^{\smallint_{0}^{z<<1}} \text{ }\har_{\mathbb{N}}^{\mu_{N}^{(p^{\alpha})}} $$
$$ \har_{p^{\alpha}\mathbb{N}}^{\mu_{p^{a}N}} = \har_{p^{\alpha}}^{\mu_{p^{\alpha}N}} \text{ } \circ_{\har}^{\Sigma} \text{ }\har_{\mathbb{N}}^{\mu_{N}^{(p^{\alpha})}} $$
$$ \comp^{\Sigma \leftarrow \smallint} \circ \comp^{\smallint \leftarrow \Sigma} = \id $$
$$ \har_{p^{\alpha}} = \comp^{\Sigma \leftarrow \smallint} \psi_{1}^{\mu_{p^{\alpha}N}} $$
$$ \psi_{1}^{\mu_{p^{\alpha}N}} = \comp^{\smallint \leftarrow \Sigma} \har_{p^{\alpha}} $$
\noindent and at words $w$ such that $\frac{\alpha}{\alpha_{0}} > \depth(w)$, with $(\text{iter}_{\har}^{\smallint_{0}^{1}})^{\frac{\tilde{\alpha}}{\tilde{\alpha}_{0}},q^{\tilde{\alpha}_{0}}} =  (\widetilde{\text{iter}}_{\har}^{\smallint_{0}^{1}})^{\textbf{a},\Lambda} \mod (\textbf{a}-\frac{\alpha}{\alpha_{0}},\Lambda-q^{\tilde{\alpha}_{0}})$ and  $(\text{iter}_{\har}^{\Sigma})^{\frac{\alpha}{\alpha_{0}},q^{\tilde{\alpha}_{0}}} = \widetilde{\text{iter}}_{\har,\Sigma}^{\textbf{a},\Lambda} \mod (\textbf{a}-\frac{\tilde{\alpha}}{\tilde{\alpha}_{0}},\Lambda-q^{\tilde{\alpha}_{0}})$ :
$$ \har^{\mu_{q^{\tilde{\alpha}}N},B}_{q^{\tilde{\alpha}}} = (\iter_{\har}^{\smallint_{0}^{1}})^{\frac{\tilde{\alpha}}{\tilde{\alpha}_{0}},q^{\tilde{\alpha}_{0}}}(\psi_{1}^{\mu_{q^{\tilde{\alpha}}N},B}) $$
$$ \har^{\mu_{q^{\tilde{\alpha}}N},B}_{q^{\tilde{\alpha}}} = \iter_{\har,\Sigma}^{\frac{\tilde{\alpha}}{\tilde{\alpha}_{0}},q^{\tilde{\alpha}_{0}}} (\har^{\mu_{q^{\tilde{\alpha}}N},B}_{q^{\tilde{\alpha}_{0}}}) $$
\noindent These results are a new illustration of the interest of considering all the iterates of the Frobenius rather than only the Frobenius : by considering the Frobenius iterated $\alpha$ times and extending it canonically, we can define generalizations of adjoint $p$-adic multiple zeta values at roots of unity of order $p^{\alpha}N$ : thus by making $N$ and $\alpha$ vary, we attain any possible order in $\mathbb{N}^{\ast}$.

\subsection{} We are actually going to prove a slightly more general result than Theorem-Definition V-1.a and Theorem-Definition V-1.b.
\newline We consider the differential forms $\frac{dz}{z}$ and $\frac{dz}{z - \xi_{N}^{j}}$, $j=1,\ldots,N$. Our main observation will be the following : the whole of part I remains true if we replace $\frac{dz}{z - \xi_{N}^{j}} = - \xi_{N}^{-j}\sum_{m \geq 1} \xi_{N}^{-jm} z^{m}dz$ by, more generally, $- \xi_{N}^{-j} \sum_{m \geq 1} f(m \mod p^{\alpha}) \xi_{N}^{-jm} z^{m}$, with $f$ is any function $\mathbb{Z}/p^{\alpha}\mathbb{Z} \rightarrow \mathbb{C}_{p}$ with values having not too big $p$-adic norm, let us say with values in $\mathcal{O}_{\mathbb{C}_{p}}$. If we choose $f : m \mapsto \rho_{p^{\alpha}}^{jm}$, with $\rho_{p^{\alpha}}$ a primitive $p^{\alpha}$-th root of unity, we obtain the canonical basis of $H^{1,\DR}(\mathbb{P}^{1} - \{0,\mu_{p^{\alpha}N},\infty\})$ which was implicit above. However, any choice of basis would work, and there is another choice which is interesting for us : $f : m \mapsto 1_{m \equiv m_{0} \mod p^{\alpha}}$, for $m_{0} \in \mathbb{Z}/p^{\alpha}\mathbb{Z}$ (where 1 denotes the characteristic function). Thus, we will also prove :
\newline 
\newline \textbf{Theorem-Definition V-1.c : compatibility with the change of basis of $H^{1,\DR}(\mathbb{P}^{1} - \{0,\mu_{p^{\alpha}N},\infty\})$ and adjoint $p$-adic multiple zeta values at $N$-th roots of unity with congruences modulo $p^{\alpha}$}
\newline i) In Theorem-Definitions V-1.a and V-1.b, let us replace the canonical basis of $H^{1,\DR}(X_{K_{p^{\alpha}N},p^{\alpha}N})$, which is implicit, by any other basis obtained in the way above from the canonical basis of $H^{1,\DR}(X_{K_{N},N})$ and a basis of $\Map(\mathbb{Z}/p^{\alpha}\mathbb{Z}, \mathcal{O}_{\mathbb{C}_{p}})$. Then, the Theorem-Definitions V-1.a and V-1.b remain valid.
\newline ii) Let two different bases $B$ and $B'$ of $H^{1,\DR}(\mathbb{P}^{1} - \{0,\mu_{p^{\alpha}N},\infty\})$ as in i) such that the matrix decomposing $B$ in $B'$ has coefficients in  $\mathcal{O}_{\mathbb{C}_{p}}$. 
\newline Then the automorphism of $H^{1,\DR}(\mathbb{P}^{1} - \{0,\mu_{p^{\alpha}N},\infty\})$ sending $B'$ to $B$ induces functorially a morphism between the generalizations of regularized iterated integrals, harmonic Ihara actions, iterations of the harmonic Frobenius associated respectively with $B$ and $B'$ by i).
\newline iii) Let us choose the basis of $H^{1,\DR}(\mathbb{P}^{1} - \{0,\mu_{p^{\alpha}N},\infty\})$ arising from the maps, $f : m \mapsto 1_{m \equiv m_{0} \mod p^{\alpha}}$, for $m_{0} \in \mathbb{Z}/p^{\alpha}\mathbb{Z}$. The analogs of the objects defined in Theorem-Definitions V-1.a and V-1.b associated with this basis are called as in Theorem-Definitions V-1.a and V-1.b except that "at $p^{\alpha}N$-th roots of unity" is replaced by "at $N$-th roots of unity with congruences modulo $p^{\alpha}$".

\subsection{}

\noindent Extending to $\pi_{1}^{\un,\DR}(X_{K_{p^{\alpha}N},p^{\alpha}N})$ the Frobenius of $\pi_{1}^{\un,\DR}(X_{K_{N},N})$ breaks a dissymetry between the indirect method and the direct method to compute it : in the direct method \cite{I-1}, we had to use iterated integrals of the set of differential forms $\{ \frac{dz}{z},\frac{dz}{z-\xi_{N}^{j}}, j=1,\ldots,N,\frac{d(z^{p^{\alpha}})}{z^{p^{\alpha}}-(\xi_{N}^{p^{\alpha}})^{j'}},j'=1,\ldots,N\}$, which involves at the same time
$(\pi_{1}^{\un,\DR}(X_{K_{N},N}),\nabla_{\KZ})$ and its-pull back $(\pi_{1}^{\un,\DR}(X_{K_{N},N})^{(p^{\alpha})},\nabla_{\KZ}^{(p^{\alpha})})$ ; whereas, in the indirect methods \cite{I-2} and \cite{I-3}, $(\pi_{1}^{\un,\DR}(X_{K_{N},N}),\nabla_{\KZ})$ and  $(\pi_{1}^{\un,\DR}(X_{K_{N},N})^{(p^{\alpha})},\nabla_{\KZ}^{(p^{\alpha})})$ were separated from each other, appearing in different terms of the equations.
\newline In the new context, $\{ \frac{dz}{z},\frac{dz}{z-\xi_{N}^{j}}, j=1,\ldots,N,\frac{d(z^{p^{\alpha}})}{z^{p^{\alpha}}-(\xi_{N}^{p^{\alpha}})^{j'}},j'=1,\ldots,N\}$ generates, both in the indirect and direct methods of computation, a subspace of the space of differential forms under consideration. Thus this context seems to be a natural framework for relating the direct and the indirect methods of computation of the Frobenius.
\newline 
\newline \textbf{Proposition V-1.d : Descent from $\pi_{1}^{\un,\DR}(\mathbb{P}^{1} - \{0,\mu_{p^{\alpha}N},\infty\})$ to $\pi_{1}^{\un,\DR}(\mathbb{P}^{1} - \{0,\mu_{N},\infty\})$}
\newline (see \S9 for a more precise statement) The regularization, the harmonic Ihara actions, the iterations of the harmonic Frobenius and the comparison maps, when restricted to the subspace of iterated integrals of $\{ p^{\alpha}\frac{dz}{z},p^{\alpha}\frac{dz}{z-\xi_{N}^{j}}, j=1,\ldots,N,\frac{d(z^{p^{\alpha}})}{z^{p^{\alpha}}-(\xi_{N}^{p^{\alpha}})^{j'}},j'=1,\ldots,N\}$, have coefficients are defined within $\pi_{1}^{\un,\DR}(X_{K_{N},N})$, i.e. convergent infinite summations of linear combinations of prime weighted multiple harmonic sums of $\pi_{1}^{\un,\DR}(X_{K_{N},N})$.
\newline 
\newline \textbf{Proposition V-1.e : the regularization and the iteration of the harmonic Frobenius}
\newline (see \S9 for a more precise statement) The regularization of iterated integrals on $X_{K_{p^{\alpha}N},p^{\alpha}N}$ defined in Theorem-Definition V.1-a can be computed by using the iteration of the harmonic Frobenius defined in Theorem-Definition V.1-b.
\newline 
\newline \textbf{Conclusion :} Via Theorem-Definition V-1.c, Proposition V-1.e and Proposition V-1.d give a connection between the formulas of the direct \cite{I-1} and indirect \cite{I-2} \cite{I-3} computations of the Frobenius, as well as a meaning of this connection in terms of a descent of the Frobenius extended to $\pi_{1}^{\un,\DR}(\mathbb{P}^{1} - \{0,\mu_{p^{\alpha}N},\infty\})$.
\newline 
\newline The present version of this paper is preliminary. In the next version, we will write explicit formulas for all the maps involved (they are not significatively different from the formulas for the maps of part I) and we will write commutative diagrams interconnecting them. This will finish to state the compatibility between the direct computation of \cite{I-1} and the indirect computations of \cite{I-2} and \cite{I-3}.
\newline Finally, these facts also have an algebraic meaning and will be related, for example, to the fact that the spaces of multiple harmonic sums with congruences of (\ref{eq:multiple harmonic sums with congruences}) are stable by the double shuffle relations. This will be explained in \cite{V-2}.

\subsection{\label{paragraph generic}}

The previous considerations also suggest to retrieve the maximal versions of the results of part I, relatively to curves $\mathbb{P}^{1}$ - punctures (one can also ask for the maximality relatively to the terms of multiple harmonic sums : in the framework $\Sigma$, this is treated, in part I and this paper, by remarking that it is possible to replace in most computations the maps $n_{i} \mapsto \frac{1}{n_{i}^{s_{i}}}$ by certain locally analytic group homomorphisms $K^{\times} \rightarrow K^{\times}$).
\newline We propose an element of answer in the Appendix A. Let us consider a more general $\pi_{1}^{\un,\DR}(\mathbb{P}^{1} - \{0=z_{0},z_{1},\ldots,z_{r},\infty\})$ over a complete normed field $K$ of characteristic $0$, with $z_{1},\ldots,z_{r}$ of norm $1$. It is still equipped with the similar connection $\nabla_{\KZ} : f \mapsto f^{-1}(df - \sum_{i=0}^{r} \frac{dz}{z-z_{i}}e_{z_{i}})$. The flat sections of $\nabla_{\KZ}$, called hyperlogarithms, have the power series expansions of their flat sections expressed in terms of multiple harmonic sums : for $d \in \mathbb{N}^{\ast}$, $j_{1},\ldots,j_{d} \in \{1,\ldots,r\}$, and $n_{d},\ldots,n_{1} \in \mathbb{N}^{\ast}$,
\begin{equation} \label{eq:generic multiple harmonic sums} 
\har_{m} \big(\begin{array}{c} z_{j_{d+1}},\ldots,z_{j_{1}} \\ n_{d},\ldots,n_{1} \end{array} \big) = m^{n_{d}+\ldots+n_{1}} \sum_{0<m_{1}<\ldots<m_{d}<m} \frac{\big( \frac{z_{j_{2}}}{z_{j_{1}}}\big)^{m_{1}} \ldots \big( \frac{z_{j_{d+1}}}{z_{j_{d}}}\big)^{m_{d}}\big( \frac{1}{z_{j_{d+1}}}\big)^{m}}{m_{1}^{n_{1}} \ldots m_{d}^{n_{d}}} \in K
\end{equation}
\noindent For the next results, we do not claim a crystalline interpretation similar to the previous ones (let us recall that $z \mapsto z^{p}$ stabilizes sets of roots of unity but not a generic subset $z_{1},\ldots,z_{r}$ of $K$).
\newline 
In Appendix A, we construct a generalization of the $\Sigma$-harmonic Ihara action and the equations that it satisfies to this setting, provided certain hypothesis on $z_{1},\ldots,z_{r}$. By viewing $z_{1},\ldots,z_{r}$ as variables, this leads to a notion of $p$-adic pseudo adjoint multiple polylogarithms.
\newline Our strategy consists in writing (if $K=\mathbb{C}_{p}$ or more generally if this is possible) $z_{i}=\omega(z_{i})+\epsilon_{i}$ with $\omega(z_{i})$ a root of unity of order prime to $p$ and $|\epsilon_{i}|_{p}<1$, and writing a power series expansions with respect to $\epsilon_{i}$'s ; the computation works when $|\epsilon_{i}|_{p}<p^{-\frac{1}{p-1}}$.
\newline We could use the strategy of Appendix A to generalize the rest of this paper to $\pi_{1}^{\un,\DR}(\mathbb{P}^{1} - \{0=z_{0},z_{1},\ldots,z_{r},\infty\})$.
\newline The qualitative difference between the results for $\pi_{1}^{\un,\DR}(\mathbb{P}^{1} - \{0,\mu_{p^{\alpha}N},\infty\})$ and  $\pi_{1}^{\un,\DR}(\mathbb{P}^{1} - \{0=z_{0},z_{1},\ldots,z_{r},\infty\})$ illustrate the particularity of the case of $\pi_{1}^{\un,\DR}(\mathbb{P}^{1} - \{0,\mu_{p^{\alpha}N},\infty\})$ and the meaning of the results for that case.

\subsection{Related work}
\noindent 
\newline $\bullet$ Aside from our papers \cite{I-1}, \cite{I-2}, \cite{I-3}, our two notes of anouncement \cite{J N1} and \cite{J N2} which announce some parts of \cite{I-1} and \cite{I-2}, and the present paper, the question of computing of the Frobenius of $\pi_{1}^{\un,\crys}(X_{k_{N},N})$ or very closely related questions appear, to our knowledge, in the work of Deligne \cite{Deligne}, \S19.6, Besser and de Jeu \cite{Besser de Jeu}, Unver \cite{U1}, \cite{U2}, \cite{U3}, \cite{U4}, Yamashita \cite{Yamashita} \S3, Dan-Cohen and Chatzistamatiou \cite{Dan-Cohen}.
\newline $\bullet$ Our results give new proofs, explicit formulas and generalizations to the Proposition 2.9 in \cite{U4}. The main result of \cite{U4} (Theorem 1.1) is proved by our paper \cite{I-1} and generalized in this paper. One can join Proposition 2.5 and Proposition 2.9 of \cite{U4} into a statement giving an explicit basis of a certain space of multiple harmonic sums with congruences (which we interpret as a space of iterated integrals over $X_{K_{p^{\alpha}N},p^{\alpha}N}$) and a decomposition of each element of the space in the basis (we interpret the elements of the basis as iterated integrals over $X_{K_{N},N}$). The results of \cite{U1}, \cite{U2}, \cite{U3} are included in those of \cite{U4}.
\newline $\bullet$ The $p$-adic iterated integrals on certain subspaces of $\mathbb{P}^{1}$ of the differential forms $\frac{dz}{z}$ and $\frac{dz}{z-\rho}$, where $\rho$ is a root of unity of order not necessarily prime to $p$, in particular values at $1$ of some $p$-adic multiple polylogarithms defined by certain such iterated integrals, and certain linear combinations of them, including the ones corresponding to multiple harmonic sums with congruences, are studied by Furusho, Komori, Matsumoto, Tsumura in \cite{FKMT}, for different purposes, concerning the generalizations with several variables of the Kubota-Leopoldt $p$-adic zeta function defined in \cite{FKMT}.

\subsection*{Acknowledgments}

This work has been achieved at Universit\'{e} Paris Diderot and Universit\'{e} de Strasbourg, supported by ERC grant 257638 and Labex IRMIA.

\section{Review on $\pi_{1}^{\un,\DR}$ and $\pi_{1}^{\un,\crys}$}

We review $\pi_{1}^{\un,\DR}(X_{K})$ where $X_{K}$ is of the form $\mathbb{P}^{1} - \{0,z_{1},\ldots,z_{r},\infty\}$ over a (complete normed) field $K$ of characteristic zero, and then $\pi_{1}^{\un,\crys}(X_{k_{N},N})$, viewed as $\pi_{1}^{\un,\DR}(X_{K_{N},N})$ equipped with the Frobenius, where $N$, $k_{N}$, $K_{N}$ are as in \S1.1. For more details, in particular about the definitions, see \cite{I-1} \S2.

\subsection{$\pi_{1}^{\un,\DR}(\mathbb{P}^{1} - \{0,z_{1},\ldots,z_{r},\infty\}\text{ }/\text{ }K)$}

Let $K$ be any field of characteristic zero, and $X_{K} = \mathbb{P}^{1} - \{0,z_{1},\ldots,z_{r},\infty\}\text{ }/\text{ }K$, where $r \in \mathbb{N}$ and $z_{1},z_{2},\ldots,z_{r} \in K$ ; let $z_{0}=0$, and $e_{\{z_{0},\ldots,z_{r}\}}$ the alphabet $\{e_{0},e_{z_{1}},\ldots,e_{z_{r}}\}$.

\begin{Proposition-Definition} Let $\mathcal{O}^{\sh,e_{\{z_{0},\ldots,z_{r}\}}}$ be the $\mathbb{Q}$-vector space $\mathbb{Q}\langle e_{\{z_{0},\ldots,z_{r}\}} \rangle= \mathbb{Q}\langle e_{0},e_{z_{1}},\ldots,e_{z_{r}}\rangle$, freely generated by words on $e_{\{z_{0},\ldots,z_{r}\}}$ including the empty word. The following operations i) to iv) make it into a graded Hopf algebra over $\mathbb{Q}$, where the grading is the number of letters (called the weight) of words :
\newline i) the shuffle product $\sh:\mathcal{O}^{\sh,e_{\{z_{0},\ldots,z_{r}\}}}\otimes \mathcal{O}^{\sh,e_{\{z_{0},\ldots,z_{r}\}}} \rightarrow \mathcal{O}^{\sh,e_{\{z_{0},\ldots,z_{r}\}}}$ defined by, for all words :
\begin{center} $(e_{z_{i_{l+l'}}}\ldots e_{z_{i_{l+1}}})\text{ }\sh\text{ }(e_{z_{i_{l}}} \ldots e_{z_{i_{1}}}) =
\sum_{\substack{\sigma \text{ permutation of }\{1,\ldots,l+l'\} \\\text{s.t. } \sigma(l)<\ldots<\sigma(1)\\ \text{and } \sigma(l+l')<\ldots<\sigma(l+1)}}
e_{z_{i_{\sigma^{-1}(l+l')}}} \ldots e_{z_{i_{\sigma^{-1}(1)}}}$ 
\end{center}
\noindent ii) the deconcatenation coproduct $\Delta_{\dec} : \mathcal{O}^{\sh,e_{\{z_{0},\ldots,z_{r}\}}}\rightarrow \mathcal{O}^{\sh,e_{\{z_{0},\ldots,z_{r}\}}} \otimes \mathcal{O}^{\sh,e_{\{z_{0},\ldots,z_{r}\}}}$, defined by, for all words :
\begin{center} 
$\Delta_{\dec}(e_{z_{i_{l}}}\ldots e_{z_{i_{1}}}) = \sum_{l'=0}^{l} e_{z_{i_{l}}}\ldots e_{z_{i_{l'+1}}} \otimes e_{z_{i_{l'}}} \ldots e_{z_{i_{1}}}$ 
\end{center}
\noindent iii) the counit $\epsilon : \mathcal{O}^{\sh,e_{\{z_{0},\ldots,z_{r}\}}} \rightarrow \mathbb{Q}$ sending all non-empty words to $0$.
\newline iv) the antipode $S : \mathcal{O}^{\sh,e_{\{z_{0},\ldots,z_{r}\}}} \rightarrow \mathcal{O}^{\sh,e_{\{z_{0},\ldots,z_{r}\}}}$, defined by, for all words :
\begin{center} 
$S(e_{z_{i_{l}}}\ldots e_{z_{i_{1}}}) = (-1)^{l} e_{z_{i_{1}}}\ldots e_{z_{i_{l}}}$ 
\end{center}
\noindent The Hopf algebra obtained in this way is called the shuffle Hopf algebra over the alphabet $\{e_{z_{0}},\ldots,e_{z_{r}}\}$.
\end{Proposition-Definition}

\begin{Proposition} \label{shuffle group scheme}
(Follows from \cite{Deligne}, \S12) The group scheme $\Spec(\mathcal{O}^{\sh,e_{\{z_{0},\ldots,z_{r}\}}})$ is pro-unipotent, and is canonically isomorphic to $\pi_{1}^{\un,\DR}(X_{K},\omega_{\DR})$, where $\omega_{\DR}$ is the canonical base-point of $\pi_{1}^{\un,\DR}(X_{K})$ in the sense of \cite{Deligne}, \S12.
\end{Proposition}

\noindent If $A$ is a ring, let $A\langle \langle e_{\{z_{0},\ldots,z_{r}\}} \rangle\rangle$ be the non-commutative $A$-algebra of power series over the alphabet $e_{\{z_{0},\ldots,z_{r}\}}$ \footnote{following a common abuse of notation, we denote in the same way $e_{z_{i}}$ in $\mathcal{O}^{\sh,e_{\{z_{0},\ldots,z_{r}\}}}$, and $e_{z_{i}}$ in $A\langle \langle e_{\{z_{0},\ldots,z_{r}\}} \rangle\rangle$, although, as we are going to see, $A\langle \langle e_{\{z_{0},\ldots,z_{r}\}} \rangle\rangle$ can be equipped in a natural way with the structure of the dual of the topological Hopf algebra of $\mathcal{O}^{\sh,e_{\{z_{0},\ldots,z_{r}\}}}$ and we will view it as containing the group of points $\Spec(\mathcal{O}^{\sh,e_{\{z_{0},\ldots,z_{r}\}}})(A)$}.
Let us denote by $\mathcal{W}(e_{\{z_{0},\ldots,z_{r}\}})$ the set of words over $e_{\{z_{0},\ldots,z_{r}\}}$ ; the coefficients of the decomposition of an element $f \in A\langle \langle e_{\{z_{0},\ldots,z_{r}\}} \rangle\rangle$ in the basis $\mathcal{W}(e_{\{z_{0},\ldots,z_{r}\}})$ are denoted by $f[\text{word}]$, as follows

\begin{equation}
\label{eq:fword} f = \sum_{w \in \mathcal{W}(e_{\{z_{0},\ldots,z_{r}\}})} f[w]w
\end{equation}

\begin{Proposition-Definition} (see for instance \cite{U1}, \S4) The dual  $(\mathcal{O}^{\sh,e_{\{z_{0},\ldots,z_{r}\}}})^{\vee}$ of the topological Hopf algebra $\mathcal{O}^{\sh,e_{\{z_{0},\ldots,z_{r}\}}}$ is $\mathbb{Q} \langle \langle e_{\{z_{0},\ldots,z_{r}\}} \rangle \rangle$, the completion of the universal enveloping algebra of the complete free Lie algebra over the variables $e_{z_{0}},e_{z_{1}},\ldots,e_{z_{r}}$, equipped with its canonical structure of topological Hopf algebra. For any $\mathbb{Q}$-algebra $A$, the group of points $\Spec(\mathcal{O}^{\sh,e_{\{z_{0},\ldots,z_{r}\}}})(A)$ consists of the grouplike elements of $(\mathcal{O}^{\sh,e_{\{z_{0},\ldots,z_{r}\}}})^{\vee}\otimes_{\mathbb{Q}} A$ and is :
\begin{equation} \label{eq:shuffle equation} \{ f \in A\langle\langle e_{\{z_{0},\ldots,z_{r}\}} \rangle\rangle \text{ }|\text{ }\forall w,w' \in \mathcal{W}(e_{\{z_{0},\ldots,z_{r}\}}), f[w\text{ }\sh\text{ }w']=f[w]f[w'],\text{ and }f[\emptyset] = 1 \}
\end{equation}
\noindent and the Lie algebra of $(\mathcal{O}^{\sh,e_{\{z_{0},\ldots,z_{r}\}}})^{\vee}\otimes_{\mathbb{Q}} A$ consists thus of the primitive elements of $(\mathcal{O}^{\sh,e_{\{z_{0},\ldots,z_{r}\}}})^{\vee}\otimes_{\mathbb{Q}} A$ and is :
\begin{equation} \label{eq:shuffle equation modulo products}
\{ f \in A \langle\langle e_{\{z_{0},\ldots,z_{r}\}} \rangle\rangle \text{ }|\text{ }\forall w,w' \in \mathcal{W}(e_{\{z_{0},\ldots,z_{r}\}}), f[w\text{ }\sh\text{ }w']=0\}
\end{equation}
\end{Proposition-Definition}

\noindent One says that the elements $f$ in (\ref{eq:shuffle equation}) satisfy the shuffle equation, and that the elements $f$ in (\ref{eq:shuffle equation modulo products}) satisfy the shuffle equation modulo products.

\begin{Proposition-Definition} (follows from \cite{Deligne}, \S7.30 and \S12) \label{prop connexion} 
\newline i) The canonical connection of $\pi_{1}^{\un,\DR}(X_{K})$, called the Knizhnik-Zamolodchikov (or KZ) connection, is given, in the sense of \cite{Deligne}, \S7.30.2, by
\begin{center} $\nabla^{X_{K}}_{\KZ} : f \mapsto f^{-1} \big( df - e_{0} f \frac{dz}{z} - \sum_{i=1}^{r} e_{z_{i}} f \frac{dz}{z-z_{i}}  \big)$ 
\end{center}
\noindent Its horizontal sections have their coefficients defined by iterated integrals of $\frac{dz}{z}$ and $\frac{dz}{z-z_{i}}$, $i=1,\ldots,r$, called hyperlogarithms.
\newline ii) Assume now that $K$ is a completed normed field and that $z_{1},\ldots,z_{r}$ have norm $1$. The hyperlogarithms are the functions given by the following power series expansion, for $|z|_{K}<1$ : for $d \in \mathbb{N}^{\ast}$, $s_{1},\ldots,s_{d} \in \mathbb{N}^{\ast}$, $j_{1},\ldots,j_{d} \in \{1,\ldots,r\}$,
$$ \Li \big(\begin{array}{c} z_{j_{d}},\ldots,z_{j_{1}} \\ s_{d},\ldots,s_{1} \end{array} \big)(z) =  \sum_{0<n_{1}<\ldots<n_{d}} \frac{\big( \frac{z_{j_{2}}}{z_{j_{1}}}\big)^{n_{1}} \ldots \big( \frac{z}{z_{j_{d}}}\big)^{n_{d}}}{n_{1}^{s_{1}} \ldots n_{d}^{s_{d}}} \in K $$
\end{Proposition-Definition}

\noindent In particular, $e_{z_{i}}$ is the residue of $\nabla^{X_{K}}_{\KZ}$ at $z_{i}$ ; for certain purposes, it will be natural to view $\mathbb{Q} \langle e_{\{z_{0},\ldots,z_{r}\}} \rangle$ as the $\mathbb{Q}$-vector space freely generated by words on the larger alphabet $\{e_{0},e_{z_{1}},\ldots,e_{z_{r}},e_{\infty}\}$, moded out by the sum of all the residues $e_{0} + e_{z_{1}} + \ldots + e_{z_{r}} + e_{\infty}$.

\begin{Definition} \label{def of tau} (Following \cite{Deligne Goncharov}, \S5) Let $\tau$ be the action of $\mathbb{G}_{m}(K)$ on $K \langle\langle e_{\{0\}\cup\mu_{N}} \rangle\rangle$, that maps $(\lambda,f) \in \mathbb{G}_{m}(K) \times K \langle\langle e_{\{0\}\cup\mu_{N}} \rangle\rangle$ to $\sum_{w\in\mathcal{W}(e_{\{0\}\cup\mu_{N}})} \lambda^{\weight(w)} f[w]w$, where $f$ is written as in equation (\ref{eq:fword}).
\end{Definition}

\subsection{$\pi_{1}^{\un,\crys}(X_{K_{N},N}$)}

The crystalline pro-unipotent fundamental groupoid of $X_{k_{N},N}$ can be defined as the fundamental groupoid associated to the Tannakian category the unipotent $F$-isocristals over $X_{k_{N},N}$, following Chiarellotto and Le Stum \cite{CLS} ; a variant using logarithmic geometry is due to Shiho \cite{Shiho 1}, \cite{Shiho 2} ; one has then a theorem of comparison relating this object to $\pi_{1}^{\un,\DR}(X_{K_{N},N})$ \cite{CLS} \cite{Shiho 1} \cite{Shiho 2} ; alternatively, following Deligne \cite{Deligne}, one can define directly a Frobenius structure on
$\pi_{1}^{\un,\DR}(X_{K_{N},N})$, and call  $\pi_{1}^{\un,\crys}(X_{k_{N},N})$ the data of 
$\pi_{1}^{\un,\DR}(X_{K_{N},N})$ plus the Frobenius. We review this more elementary point of view in this paragraph.
\newline
\newline Let us fix a prime $p$, and $N \in \mathbb{N}^{\ast}$ prime to $p$. We now go back to the notations of the introduction, and assume that $K=K_{N}$, $r=N$, and $(z_{1},\ldots,z_{r})=( \xi_{N}^{1},\ldots,\xi_{N}^{N})$ where $\xi_{N}$ is a primitive $N$-th root of unity in $K_{N}$. Following \cite{I-1} and \cite{I-3}, for each $\alpha \in \mathbb{N}^{\ast}$, we adopt the convention that the Frobenius iterated $\alpha$ times is $\tau(p^{\alpha})\phi^{\alpha}$ where $\phi$ is in the sense of \cite{Deligne}, \S13, and, for each $\alpha \in -\mathbb{N}^{\ast}$, the Frobenius iterated $\alpha$ times is $F_{\ast}^{\alpha}$, where $F_{\ast}=\phi^{-1}$ is in the sense of \cite{Deligne}, \S11.
\newline We have 
$\mathcal{O}\big(\pi_{1}^{\un,\DR}(X_{K},\omega_{\DR}(X_{K}))\big)$, resp. $\mathcal{O}\big(\pi_{1}^{\un,\DR}(X_{K}^{(p^{\alpha})},\omega_{\DR}(X_{K}^{(p^{\alpha})}))\big)$ is the shuffle Hopf algebra over the alphabet 
$\{e_{0},e_{z_{1}},\ldots,e_{z_{N}}\}$ resp. $\{e_{0},e_{z_{1}^{(p^{\alpha})}},\ldots,e_{z_{N}^{(p^{\alpha})}}\}$. We denote by $w \mapsto w^{(p^{\alpha})}$ the isomorphism between them that sends $e_{0} \mapsto e_{0}$ and $e_{z_{i}} \mapsto e_{z_{i}^{(p^{\alpha})}}$, $i \in \{1,\ldots,N\}$.
\newline For any $\alpha \in \mathbb{N}^{\ast} \cup - \mathbb{N}^{\ast}$, the Frobenius is determined by its values at the canonical paths $_{y}1_{x}$ in the sense of \cite{Deligne}, \S12 ; they are determined by the couple $(\Li_{p,\alpha}^{\dagger},\Phi_{p,\alpha})$ made of the non-commutative generating series of overconvergent $p$-adic hyperlogarithms ($p$HL$^{\dagger}$'s) and the non-commutative generating series of $p$MZV$\mu_{N}$'s, defined as follows :

\begin{Definition} i) (see \S1.1 and \cite{I-1} for references) Let $U_{N}$ be the rigid analytic space $\mathbb{P}^{1,\an} - \cup_{i=1}^{N} \{z \text{ }|\text{ }|z-\xi_{N}^{i}|_{p}<1\}\text{ }/\text{ }K_{N}$ and $\frak{A}^{\dagger}(U_{N})$ the $K_{N}$-algebra of overconvergent rigid analytic functions on it.
\newline For $\alpha \in \mathbb{N}^{\ast}$, let $\Li_{p,\alpha}^{\dagger},$ resp. $\Li_{p,-\alpha}^{\dagger} \in \pi_{1}^{\un,\DR}(X_{K},\vec{1}_{0})(\frak{A}^{\dagger}(U_{N}))$, be the map 
$z \mapsto (p^{\alpha})^{\weight}\phi_{\ast}^{-\alpha}({}_z 1_{\vec{1}_{0}})$, resp. $z \mapsto F_{\ast}^{\alpha} ({}_z 1_{\vec{1}_{0}})$ : the coefficients of $\Li_{p,\pm \alpha}^{\dagger}$, which are elements of $\frak{A}^{\dagger}(U_{N})$, are called overconvergent $p$-adic hyperlogarithms.
\newline ii) (see \S1.1 for references) We denote by $z_{N+1} = \infty$ ; let $i \in \{1,\ldots,N+1\}$. Let
$\Phi^{(z_{i})}_{p,\alpha} = \tau(p^{\alpha})\phi^{\alpha} ( _{\vec{1}_{z_{i}^{p^{\alpha}}}} 1 _{\vec{1}_{0}}) \in \pi_{1}^{\un,\DR}(X_{K_{N},N},\vec{1}_{z_{i}},\vec{1}_{0})(K_{N})$ if $\alpha>0$, and 
$\Phi^{(z_{i}^{p^{|\alpha|}})}_{p,\alpha} = F_{\ast}^{\alpha} ( _{\vec{1}_{z_{i}}} 1 _{\vec{1}_{0}}) \in \pi_{1}^{\un,\DR}(X_{K_{N},N}^{(p^{|\alpha[]})},\vec{1}_{z_{i}^{p^{|\alpha|}}},\vec{1}_{0})(K_{N})$ if $\alpha<0$.
\newline For $w = e_{0}^{s_{d}-1}e_{z_{i_{d}}} \ldots e_{0}^{s_{1}-1}e_{z_{i_{1}}}= \big( \begin{array}{c} z_{i_{d}}, \ldots, z_{i_{1}} \\ s_{d},\ldots,s_{1} \end{array} \big)$, with $d \in \mathbb{N}^{\ast}$, and $s_{1},\ldots,s_{d} \in \mathbb{N}^{\ast}$, and $i_{1},\ldots,i_{d} \in \{1,\ldots,N\}$, one calls $p$-adic multiple zeta values at roots of unity the numbers $\zeta^{(z_{i})}_{p,\alpha}(w) = \Phi^{(z)}_{p,\alpha}[w]$ if $\alpha>0$, and  $\zeta^{(z_{i}^{p^{|\alpha|}})}_{p,\alpha}(w^{(p^{\alpha})}) = \Phi^{(z_{i}^{p^{|\alpha|}})}_{p,\alpha}[w^{(p^{\alpha})}]$ if $\alpha<0$.
\newline iii) For all objects $\ast$ above, and $\alpha = \frac{\log(q_{N})}{\log(p)} \tilde{\alpha}$, let $\ast_{q_{N},\tilde{\alpha}} = \ast_{p,\alpha}$.
\newline iii) \label{def Li coleman} (Furusho \cite{Furusho 1} for $N=1$, Yamashita \cite{Yamashita} for any $N$). We fix a determination $\log_{p}$ of the $p$-adic logarithm. Let $\Li_{p,X_{K}}^{\KZ}$ resp. $\Li_{p,X_{K}^{(p^{\alpha})}}^{\KZ}$ be the unique Coleman function on $X_{K}$, resp. $X_{K}^{(p^{\alpha})}$, which is a horizontal section of $\nabla_{\KZ}$ and has the asymptotic behaviour
$\Li_{p,X_{K}}^{\KZ}(z) \underset{z \rightarrow 0}{\sim} e^{e_{0} \log_{p}(z)}$, resp. 
$\Li_{p,X_{K}^{(p^{\alpha})}}^{\KZ}(z) \underset{z \rightarrow 0}{\sim} e^{e_{0} \log_{p}(z)}$.
\end{Definition}

\noindent Below, we assume $\alpha>0$. The properties for $\alpha<0$ are similar. For all $j=1,\ldots,N$, we denote by $z_{j}=\xi_{N}^{j}$, by $\Phi^{(\xi_{N}^{j})}_{p,\alpha}$ the image of $\Phi_{p,\alpha}$ by the automorphism $(x \mapsto z_{j}x)_{\ast}$, and $\omega_{z_{j}}(x)= \frac{dx}{x-z_{j}}$. We denote by $z_{0}=0$.
The $\Li_{p}^{\KZ}$'s are related to multiple harmonic sums of equation (\ref{eq: multiple harmonic sums}) as follows :

\begin{Fact} For $z \in \mathbb{C}_{p}$ such that $|z|_{p}<1$, and for each word $w=e_{0}^{n_{d}-1}e_{\xi_{N}^{j_{d}}} \ldots e_{0}^{n_{1}-1}e_{\xi_{N}^{j_{1}}}$, with $j_{d},\ldots,j_{1} \in \{1,\ldots,N\}$, $n_{d},\ldots,n_{1}\in \mathbb{N}^{\ast}$, we have
	$$ \Li_{p,X_{K}}^{\KZ}[w](z) = (-1)^{d}\sum_{0<m_{1}<\ldots<m_{d}}
	\frac{\big(\frac{\xi_{N}^{j_{2}}}{\xi_{N}^{j_{1}}}\big)^{m_{1}} \ldots \big(\frac{\xi_{N}^{j_{d}}}{\xi_{N}^{j_{d-1}}}\big)^{m_{d-1}}\big(\frac{z}{\xi_{N}^{j_{d}}}\big)^{m_{d}}}{m_{1}^{n_{1}}\ldots m_{d}^{n_{d}}} $$
	$$ \Li_{p,X_{K}^{(p^{\alpha})}}^{\KZ}[w^{(p^{\alpha})}](z) = (-1)^{d}\sum_{0<m_{1}<\ldots<m_{d}}
	\frac{\big(\frac{\xi_{N}^{j_{2}p^{\alpha}}}{\xi_{N}^{j_{1}p^{\alpha}}}\big)^{n_{1}} \ldots \big( \frac{\xi_{N}^{j_{d}p^{\alpha}}}{\xi_{N}^{j_{d-1}p^{\alpha}}}\big)^{m_{d-1}} \big(\frac{z}{\xi_{N}^{j_{d}p^{\alpha}}}\big)^{m_{d}}}{m_{1}^{n_{1}}\ldots m_{d}^{n_{d}}} $$
\end{Fact}

\begin{Proposition} The couple $(\Li_{p,\alpha}^{\dagger},\Phi_{p,\alpha})$ is determined by :
\begin{equation} \label{eq:horizontality equation}d\Li_{p,\alpha}^{\dagger} = \bigg( \sum_{j=0}^{N} p^{\alpha} \omega_{z_{j}}(z) e_{z_{j}} \bigg) \Li_{p,\alpha}^{\dagger} - \Li_{p,\alpha}^{\dagger} \bigg( \sum_{j=0}^{N} \omega_{z_{j}^{p^{\alpha}}}(z^{p^{\alpha}}) ,{\Phi^{(z_{j})}_{p,\alpha}}^{-1}e_{z_{j}}\Phi^{(z_{j})}_{p,\alpha} \bigg)
\end{equation}
\noindent i.e.
\begin{multline} \label{eq:horizontality equation bis} \Li_{p,\alpha}^{\dagger}(z)(e_{0},e_{z_{1}},\ldots,e_{z_{N}})\times\Li_{p,X_{K}^{(p^{\alpha})}}^{\KZ}(z^{p^{\alpha}})\big(e_{0},{\Phi^{(z_{1})}_{p,\alpha}}^{-1}e_{z_{1}}\Phi^{(z_{1})}_{p,\alpha},\ldots,{\Phi^{(z_{N})}_{p,\alpha}}^{-1}e_{z_{N}}\Phi^{(z_{N})}_{p,\alpha} \big)
\\ = \Li_{p,X_{K}}^{\KZ}(z)(p^{\alpha}e_{0},p^{\alpha}e_{z_{1}},\ldots,p^{\alpha}e_{z_{N}})
\end{multline}
\noindent and
\begin{equation} \label{eq:algebraic Frobenius equation}e_{0} + \sum_{i=1}^{N} (\Phi_{p,\alpha}^{(z_{i})})^{-1} e_{z_{i}} \Phi_{p,\alpha}^{(z_{i})} + \Li_{p,\alpha}(\infty)^{-1}(e_{0}+\sum_{i=1}^{N}e_{z_{i}}) \Li_{p,\alpha}(\infty) = 0 
\end{equation}
\end{Proposition}

\section{Setting for computations}

\subsection{Bases of $H^{1,\DR}(\mathbb{P}^{1} - \{0,\mu_{p^{\alpha}N},\infty\})$ and presentations of $\pi_{1}^{\un,\DR}(X_{K_{p^{\alpha}N},p^{\alpha}N})$ ; two particular examples}

\subsubsection{Bases of $H^{1,\DR}(\mathbb{P}^{1} - \{0,\mu_{N},\infty\})$, bases of $\Map(\mathbb{Z}/p^{\alpha}\mathbb{Z},K_{p^{\alpha}N})$ and bases of $H^{1,\DR}(\mathbb{P}^{1} - \{0,\mu_{p^{\alpha}N},\infty\})$}

\noindent 
\newline We define a certain type of basis of $H^{1,\DR}(\mathbb{P}^{1} - \{0,\mu_{p^{\alpha}N},\infty\})$.

\begin{Definition} \label{def forming a basis}Let $\omega_{0}(z) = \frac{dz}{z},\omega_{\xi_{N}^{j}}(z) = \frac{dz}{z- \xi_{N}^{j}} = -\frac{1}{\xi_{N}^{j}} \sum_{m\geq 0} (\xi_{N}^{j})^{-m}z^{m}$, $j=1,\ldots,N$ be the canonical basis of $\Omega^{1}(\mathbb{P}^{1} - \{0,\mu_{N},\infty\})$. Let $b=(f_{1},\ldots,f_{r})$ be a set of maps $\Map(\mathbb{Z}/p^{\alpha}\mathbb{Z},K_{p^{\alpha}N})$.
\newline Then we consider the sequence of differential forms
\begin{equation} \label{eq: form a basis of palpha N}  \bigg( \frac{1}{\xi_{N}^{j}} \sum_{m\geq 0} f_{i}(m \mod p^{\alpha}) (\xi_{N}^{j})^{-m}z^{m} \bigg)_{\substack{i=1,\ldots,r \\ j=1,\ldots,N}} \end{equation} 
\end{Definition}

\begin{Remark} We could also generalize this definition and our computations by replacing the canonical basis of $\Omega^{1}(\mathbb{P}^{1} - \{0,\mu_{N},\infty\})$ by any basis. This would give ultimately variants of $p$MZV$\mu_{N}$'s, for example expressed by multiple harmonic sums involving congruences modulo $N$. However, this object is less natural for our current purposes. We could also choose $N$ different $b$'s.
\end{Remark}

\noindent We are interested in the cases where $r=p^{\alpha}$ and $b$ is a basis of $\Map(\mathbb{Z}/p^{\alpha}\mathbb{Z},K_{p^{\alpha}N})$, and (\ref{eq: form a basis of palpha N}) joint with $\omega_{0}$ define a basis of $H^{1,\DR}(\mathbb{P}^{1} - \{0,\mu_{p^{\alpha}N},\infty\})$. We will be interested in two examples :

\begin{Example} \label{example basis} i) With $r=p^{\alpha}$ and $f_{i} : m \mapsto \rho_{p^{\alpha}}^{im}$, $i=1,\ldots,r$, where $\rho_{p^{\alpha}}$ is a primitive $p^{\alpha}$-th root of unity, we obtain the canonical basis $B_{can}$ of $H^{1,\DR}(\mathbb{P}^{1} - \{0,\mu_{p^{\alpha}N},\infty\})$ :
$$ \bigg\{ \frac{dz}{z} , \frac{dz}{z - \rho_{p^{\alpha}}^{i}\xi_{N}^{j}}, i=1,\ldots,p^{\alpha},\text{ }j=1,\ldots,N \bigg\} $$
\noindent ii) With $r=p^{\alpha}$ and $f_{i} : m \mapsto 1_{m \equiv i \mod p^{\alpha}}$, we obtain another basis $B_{cong}$ of  $H^{1,\DR}(\mathbb{P}^{1} - \{0,\mu_{p^{\alpha}N},\infty\})$, the "canonical basis with congruences modulo $p^{\alpha}$" :
$$ \bigg\{ \frac{dz}{z} , \frac{z^{r}dz}{z^{p^{\alpha}} - \xi_{N}^{jp^{\alpha}}}, r=1,\ldots,p^{\alpha},\text{ }j=1,\ldots,N \bigg\} $$
\end{Example}

\noindent Of course, we are interested in the second example partly because we have the lift of Frobenius $z \mapsto z^{p^{\alpha}}$ on $\mathbb{P}^{1} - \{0,\mu_{N},\infty\}$.

\subsubsection{Presentation of $(\pi_{1}^{\un,\DR}(X_{K_{p^{\alpha}N},p^{\alpha}N}),\nabla_{\KZ}^{\mu_{p^{\alpha}N}})$ from the canonical basis with congruences of $H^{1,\DR}(\mathbb{P}^{1} - \{0,\mu_{p^{\alpha}N},\infty\})$}

\noindent In this paragraph, we formalize the presentation of  $(\pi_{1}^{\un,\DR}(X_{K_{p^{\alpha}N},p^{\alpha}N}),\nabla_{\KZ}^{\mu_{p^{\alpha}N}})$ defined by the canonical basis with congruences modulo $p^{\alpha}$ of $H^{1,\DR}(\mathbb{P}^{1} - \{0,\mu_{p^{\alpha}N},\infty\})$ (Example \ref{example basis}, ii)).

\begin{Definition}
Let $e_{0 \cup \mu_{N}^{\mod p^{\alpha}}}$ be the alphabet formed by $e_{0}$ and the formal variables $e_{\xi_{N}^{j}}^{r \mod p^{\alpha}}$, $j=1,\ldots,N$, $r=0,\ldots,p^{\alpha}-1$.
\end{Definition}

\begin{Definition}
	\noindent\newline i) Let $\mathcal{O}^{\sh,e_{0 \cup \mu_{N}^{\mod p^{\alpha}}}}$ be the the shuffle Hopf algebra over the alphabet $\{ e_{0} \} \cup \{ e^{r \mod p^{\alpha}}_{\xi_{N}^{j}}, \text{ } r=0,\ldots,p^{a}-1,\text{ }j=1,\ldots,N\}$.
	\newline ii) Let $\tilde{\Pi}^{\un}(Y_{K_{N},N}^{\mod p^{\alpha}}) = \Spec(\mathcal{O}^{\sh,e_{0 \cup \mu_{N}^{\mod p^{\alpha}}}})$. It is a pro-unipotent algebraic group over $\mathbb{Z}$.
\newline iii) Let also $K_{N} \langle \langle e_{0 \cup \mu_{N}^{\mod p^{\alpha}}} \rangle\rangle = K_{N} \langle\langle e_{0}, (e^{r \mod p^{\alpha}}_{\xi_{N}^{j}})_{\substack{r=0,\ldots,p^{\alpha}-1 \\ j=1,\ldots,N}} \rangle\rangle$ be the non-commutative algebra of formal power series over the variables $e_{0}$ and $e^{r \mod p^{\alpha}}_{\xi_{N}^{j}}$, $r=0,\ldots,p^{\alpha}-1$, $j=1,\ldots,N$, with coefficients in $K_{N}$. It is naturally equipped with the structure of topological Hopf algebra dual to $\mathcal{O}^{\sh,e_{0 \cup \mu_{N} \times \mathbb{Z}/p^{\alpha}\mathbb{Z}}}$.
\newline iv) Let 
$\displaystyle\nabla_{\KZ}^{\mu^{\mod  p^{\alpha}}_{N}} : K_{p^{\alpha}N}[[z]]\langle \langle e_{0 \cup \mu_{N}^{\mod p^{\alpha}}} \rangle\rangle  \rightarrow  K_{p^{\alpha}N}[[z]]\frac{dz}{z} \langle\langle e_{0 \cup \mu_{N}^{\mod p^{\alpha}}} \rangle\rangle$, defined by
\begin{equation} \label{eq: nabla mod p alpha} L \mapsto dL - \bigg( \frac{dz}{z}e_{0} + \sum_{r=0}^{p^{\alpha}-1} \sum_{j=1}^{N} e^{r \mod  p^{\alpha}}_{\xi_{N}^{j}} \frac{z^{r}dz}{z^{p^{\alpha}} - \xi_{N}^{p^{\alpha}j}} \bigg) L
\end{equation}
\end{Definition}

\begin{Definition} We have defined the canonical presentation with congruences modulo $p^{\alpha}$ of $\pi_{1}^{\un,\DR}(X_{K_{p^{\alpha}N},p^{\alpha}N})$. The presentation reviewed in \S2.1 is the canonical presentation of $\pi_{1}^{\un,\DR}(X_{K_{p^{\alpha}N},p^{\alpha}N})$.
\end{Definition}

\begin{Remark} The previous objects are defined over 
$Y_{K_{N},N}^{\mod p^{\alpha}} = \Spec\big( K_{N}[z,\frac{1}{z},\frac{1}{z^{p^{\alpha}}-\xi_{N}^{p^{\alpha}}},\ldots,\frac{1}{z^{p^{\alpha}}-\xi_{N}^{Np^{\alpha}}}] \big)$, but the singularities of $\frac{1}{z^{p^{\alpha}}-\xi_{N}^{p^{\alpha}}}$are logarithmic only over $K_{p^{\alpha}N}$ and we have  $Y_{K_{N},N}^{\mod p^{\alpha}} \times_{\Spec(K_{N})} \Spec(K_{p^{\alpha}N}) = X_{K_{p^{\alpha}N},p^{\alpha}N}$
\end{Remark}

\noindent Of course, there is an adaptation of these constructions for other choices of bases of $\Omega^{1}(X_{K_{p^{\alpha}N},p^{\alpha}N})$.

\subsubsection{Base-change relating the two presentations of $(\pi_{1}^{\un,\DR}(X_{K_{p^{\alpha}N},p^{\alpha}N}),\nabla_{\KZ}^{\mu_{p^{\alpha}N}})$}

\begin{Proposition} i) The correspondence $\tilde{e}_{\xi_{N}^{j}}^{r \mod p^{\alpha}} = \frac{1}{p^{a}} \sum_{\rho^{p^{a}}=1} \big(\frac{-1}{\rho\xi}\big)^{p^{\alpha}-r}e_{\rho\xi}$ defines a $K_{p^{\alpha}N}$-linear automorphism $K_{p^{\alpha}N}e_{0} \oplus \big(  \oplus_{\substack{j=1,\ldots,N\\ r=0,\ldots,p^{\alpha}-1}}K_{p^{\alpha}N} e_{\xi_{N}^{j}}^{\mod p^{\alpha}} \big) \simlra K_{p^{\alpha}N} e_{0} \oplus \big(\oplus_{\substack{j=1,\ldots,N\\ i=0,\ldots,p^{\alpha}-1}} K_{p^{\alpha}N}e_{\rho_{p^{\alpha}}^{i}\xi_{N}^{j}}\big)$, whose inverse is 
$e_{\rho\xi} \mapsto  \sum_{r=0}^{p^{\alpha}-1} \big( \frac{1}{\rho\xi} \big)^{1+r} \xi^{p^{\alpha}} e_{\xi^{p^{\alpha}}}^{r \mod p^{\alpha}}$.
\newline ii) The dual of this isomorphism is characterized by the commutativity of the diagram
$$ \displaystyle
\begin{array}{ccccc}
K_{p^{\alpha}N}[[z]] \langle\langle e_{0 \cup \mu_{N}^{\mod p^{\alpha}}} \rangle\rangle   &
{\longrightarrow}
&  K_{p^{\alpha}N}[[z]]\langle\langle e_{0 \cup \mu_{N}^{\mod p^{\alpha}}} \rangle\rangle  \\
\downarrow_{\nabla_{\KZ}^{\mu_{N}^{\mod p^{\alpha}}}} && \downarrow_{\nabla_{\KZ}^{\mu_{p^{\alpha}N}}} \\
K_{p^{\alpha}N}[[z]]\frac{dz}{z}\langle\langle e_{0 \cup \mu_{N}^{\mod p^{\alpha}}} \rangle\rangle  & \longrightarrow &  K_{p^{\alpha}N}[[z]]\frac{dz}{z}\langle\langle e_{0 \cup \mu_{N}^{\mod p^{\alpha}}} \rangle\rangle \\ \end{array} $$
\noindent where the horizontal arrows arise from the canonical isomorphism $\Pi_{N}^{\mod p^{\alpha}} \times_{\Spec(\mathbb{Z})} \Spec(K_{p^{\alpha}N}) \simeq \Pi_{p^{\alpha}N} \times_{\Spec(\mathbb{Z})} \Spec(K_{p^{\alpha}N})$ defined by this linear isomorphism.
\end{Proposition}

\begin{proof} Let $Z$ be a formal variable, $\rho \in \mu_{p^{\alpha}}(K_{p^{\alpha}N})$ and $\xi \in \mu_{N}(K_{p^{\alpha}N})$. By writing the power series expansion of $(1-\rho^{-1}\xi^{-1}z)^{-1}$, we get
	$\frac{1}{Z-\rho \xi} = \sum_{r=0}^{p^{\alpha}-1} \big( \frac{1}{\rho\xi} \big)^{1+r} \xi^{p^{\alpha}} \frac{z^{r}}{z^{p^{\alpha}} - \xi^{p^{\alpha}}}$. Conversely, we have
	$\frac{1}{Z^{p^{\alpha}} - \xi^{p^{\alpha}}} = \frac{1}{p^{a}} \sum_{\rho^{p^{a}}=1} \big[ \big(\frac{-1}{\rho\xi}\big)^{p^{\alpha}-r}\frac{1}{Z-\rho\xi}\big]$. Indeed, $ \frac{1}{Z^{p^{\alpha}} - \xi^{p^{\alpha}}} = \frac{1}{p^{\alpha}Z^{p^{\alpha}-1}} \sum_{\rho^{p^{\alpha}}=1} \frac{1}{Z-\rho\xi}$ and, for all $t \in \mathbb{N}^{\ast}$ and $\eta \in K - \{0\}$, 
	$ \frac{1}{Z^{t}(Z-\eta)} = \frac{(-1)^{t+1}}{\eta^{t+1}(Z-\eta)} + \sum_{l=1}^{t} \frac{(-1)^{l}}{\eta^{l}Z^{t+1-l}}$, whence
	$\frac{Z^{r}}{Z^{p^{\alpha}} - \xi^{p^{\alpha}}} = \frac{1}{p^{a}} \sum_{\rho^{p^{a}}=1} \big[ \big(\frac{-1}{\rho\xi}\big)^{p^{\alpha}-r}\frac{1}{Z-\rho\xi} + \sum_{l=1}^{p^{\alpha}-r}  \big(\frac{-1}{\rho\xi}\big)^{l} \frac{1}{Z^{p^{\alpha}+1-r-l}} \big]$ and $\sum_{\rho^{p^{\alpha}}=1} \rho^{-l} = 0$ for each $l \in \{1,\ldots,p^{\alpha}-1\}$.
\end{proof}

\subsubsection{Comment on viewing of $\pi_{1}^{\un,\DR}(X_{K_{p^{\alpha}N},p^{\alpha}N})$ with its presentation with congruences modulo $p^{\alpha}$ as an extension of $\pi_{1}^{\un,\DR}(X_{K_{N},N}^{(p^{\alpha})})$}

\begin{Remark} Whereas  $\pi_{1}^{\un,\DR}(X_{K_{p^{\alpha}N},p^{\alpha}N})$ with its canonical presentation admits $\pi_{1}^{\un,\DR}(X_{K_{N},N})$ as a natural quotient and subobject,  $\pi_{1}^{\un,\DR}(X_{K_{p^{\alpha}N},p^{\alpha}N})$ with its presentation with congruences modulo $p^{\alpha}$ (\S3.1.2) admits the pull-back by Frobenius  $\pi_{1}^{\un,\DR}(X_{K_{N},N}^{(p^{\alpha})})$ as a natural quotient and sub-object.
\newline Thus we may view $\pi_{1}^{\un,\DR}(X_{K_{p^{\alpha}N},p^{\alpha}N})$ with its presentation with congruences modulo $p^{\alpha}$ as a generalization of  $\pi_{1}^{\un,\DR}(X_{K_{N},N}^{(p^{\alpha})})$ to $X_{K_{p^{\alpha}N},p^{\alpha}N}$. 
However, the Frobenius that we are going to construct will not be an isomorphism relating two different presentations $\pi_{1}^{\un,\DR}(X_{K_{p^{\alpha}N},p^{\alpha}N})$, but rather a relation between $\pi_{1}^{\un,\DR}(X_{K_{p^{\alpha}N},p^{\alpha}N})$ and $\pi_{1}^{\un,\DR}(X_{K_{N},N}^{(p^{\alpha})})$ having coefficients in $K_{p^{\alpha}N}$. It would be possible to extend our construction as a relation between  two different presentations of $\pi_{1}^{\un,\DR}(X_{K_{p^{\alpha}N},p^{\alpha}N})$, but this would be more artificial and less interesting arithmetically.
\end{Remark}

\noindent The view of $\pi_{1}^{\un,\DR}(X_{K_{p^{\alpha}N},p^{\alpha}N})$ with its presentation with congruences modulo $p^{\alpha}$ as an extension of $\pi_{1}^{\un,\DR}(X_{K_{N},N}^{(p^{\alpha})})$  leads however to the following observations.

\begin{Remark} The Proposition 3.8 and the natural passage from the i) of Example \ref{example basis} to the ii) of Example \ref{example basis} are better expressed by the map :
	$$ \Eucl_{p^{\alpha}N} : \begin{array}{l} K_{p^{\alpha}N}[[z]] \simeq K_{p^{\alpha}N}[[y]]^{\oplus p^{a}}
	\\ S(z) \mapsto (S(y)^{(0 \mod p^{\alpha})},\ldots,S(y)^{(p^{\alpha}-1 \mod p^{\alpha})})
	\end{array} $$
\noindent characterized by the equation $S(z) = \sum_{r=0}^{p^{a}-1} z^{r}S^{(r \mod p^{\alpha})}(z^{p^{a}})$.
\newline $\Eucl_{p^{\alpha}N}$ extends in a natural way to an isomorphism $K_{p^{\alpha}N}[[z]]dz \simeq K_{p^{\alpha}N}[[y]]^{\oplus p^{a}}dy$ and further as an isomorphism $K_{p^{\alpha}N}[[z]]dz \oplus K_{p^{\alpha}N} \frac{dz}{z}  \simeq K_{p^{\alpha}N}[[y]]^{\oplus p^{a}}dy \oplus K_{p^{\alpha}N} \frac{dy}{y}$, characterized by $z^{p^{\alpha}} \mapsto y$ and the commutation with the differentials $d$, given by $\frac{dz}{z} \mapsto \frac{1}{p^{\alpha}}\frac{dy}{y}$ and the following equation : for $L \in K_{p^{\alpha}N}[[z]] \langle \langle e_{0 \cup \mu_{p^{\alpha}N}}\rangle\rangle$,
\begin{equation}
\label{eq:regression to muN 1} dS(z) = \sum_{r=0}^{p^{\alpha}-1} z^{r} d \bigg( S^{(r \mod p^{\alpha})}(z^{p^{\alpha}}) \bigg) + \sum_{r=1}^{p^{\alpha}-1} rz^{r} \frac{S^{(r \mod p^{\alpha})}(z^{p^{\alpha}})}{p^{\alpha} z^{p^{\alpha}}} d(z^{p^{\alpha}}) 
\end{equation}
\end{Remark}

\begin{Remark} Using $\Eucl_{p^{\alpha}N}$, $\nabla_{\KZ}^{\mu_{N}^{\mod p^{\alpha}}}$ as a connection on a trivial bundle over $\mathbb{P}^{1} - \{0,\mu_{N},\infty\}$ defined over $K_{N}$, but which is not pro-unipotent. We have :
	\begin{equation} \label{eq:regression to muN 2} \frac{dz}{z}e_{0} + \sum_{r=0}^{p^{\alpha}-1} \sum_{j=1}^{N} e^{r \mod  p^{\alpha}}_{\xi_{N}^{j}} \frac{z^{r}dz}{z^{p^{\alpha}} - \xi_{N}^{p^{\alpha}j}} = \frac{d(z^{p^{\alpha}})}{p^{\alpha}z^{p^{\alpha}}} e_{0} + \sum_{r=0}^{p^{\alpha}-1} \sum_{j=1}^{N} e^{r \mod  p^{\alpha}}_{\xi_{N}^{j}} z^{r+1} \frac{d(z^{p^{\alpha}})}{p^{\alpha}z^{p^{\alpha}}}\frac{1}{z^{p^{\alpha}} - \xi_{N}^{p^{\alpha}j}} 
	\end{equation}
	\noindent Thus, by substituting equations (\ref{eq:regression to muN 1}) and (\ref{eq:regression to muN 2}) inside the expression $\nabla_{\KZ}^{\mu_{N}^{\mod p^{\alpha}}}(L)=0$ where $\nabla_{\KZ}^{\mu_{N}^{\mod p^{\alpha}}}$ is according to equation (\ref{eq: nabla mod p alpha}), we see that $\Eucl_{p^{\alpha}N}$ defines an isomorphism of modules with connections 
	$$(K_{p^{\alpha}N}[[z]]\langle\langle e_{0 \cup \mu_{N}}\rangle\rangle,\nabla^{\mu_{N}^{\mod p^{\alpha}}}_{\KZ}) \simlra (K_{p^{\alpha}N}[[y]]^{\oplus p^{a}} \langle\langle e_{0 \cup \mu_{N}}\rangle\rangle,\nabla'_{\KZ}) $$ 
	\noindent where $\nabla'_{\KZ} : (L_{0},\ldots,L_{p^{\alpha}-1}) \mapsto  ((\nabla'_{\KZ}L)_{0},\ldots,(\nabla'_{\KZ}L)_{p^{\alpha}-1})$ defined by, for all $r=0,\ldots,p^{\alpha-1}$,
	$$ (\nabla'_{\KZ} L)_{r}(y) = \frac{L_{r}}{p^{\alpha}y} \frac{dy}{y} (e_{0}-r) + \sum_{\substack{0 \leq r_{1},r_{2} \leq p^{\alpha}-1 \\ r_{1}+r_{2}+1 \equiv r \mod p^{\alpha}}} \sum_{j=0}^{N-1} e_{\xi_{N}^{p^{\alpha}j}}^{r_{1} \mod p^{\alpha} }\frac{dy}{y(y-\xi_{N}^{p^{\alpha}j})}  L_{r_{2}}(y) $$
\noindent That bundle keeps a weight filtration, but the flat sections of its weight-graded quotients are built out of the functions $y \mapsto y^{\frac{r}{p^{\alpha}}}$, $r =0,\ldots,p^{\alpha}-1$. 
\end{Remark}

\subsection{Multiple harmonic sums at roots of unity attached to elements of $\Map(\mathbb{Z}/p^{\alpha}\mathbb{Z},K_{p^{\alpha}N})$}

We extend the framework established in part I to deal with multiple harmonic sums as elementary $p$-adic functions.

\subsubsection{Definition} 
We suppose chosen a basis $B=(f_{1},\ldots,f_{p^{\alpha}})$ of $\Map(\mathbb{Z}/p^{\alpha}\mathbb{Z},K_{p^{\alpha}N})$ such that the operation defined in \S3.1 defines a basis of $H^{1,\DR}(\mathbb{P}^{1} - \{0,\mu_{p^{\alpha}N},\infty\})$. We suppose that $f_{1},\ldots,f_{p^{\alpha}}$ take values in $\mathcal{O}_{K_{p^{\alpha}}}$.

\begin{Definition} \label{def mhs 1} 
For $m \in \mathbb{N}^{\ast}$, $d \in \mathbb{N}^{\ast}$, $j_{1},\ldots,j_{d} \in \{1,\ldots,r\}$, $n_{1},\ldots,n_{d} \in \mathbb{N}^{\ast}$, we call weighted multiple harmonic sums at $N$-th roots of unity with  the numbers
\begin{equation} \label{eq: multiple harmonic sums with functions of Z / p alpha Z}
\har_{m} \bigg(\begin{array}{c} f_{i_{d}},\ldots,f_{i_{1}} \\ \xi_{N}^{j_{d+1}},\ldots,xi_{N}^{j_{1}} \\ n_{d},\ldots,n_{1} \end{array} \bigg) = m^{n_{d}+\ldots+,n_{1}} \sum_{0<m_{1}<\ldots<m_{d}<m} 
\frac{f_{i_{1}}(m_{1})\big( \frac{\xi_{N}^{j_{2}}}{\xi_{N}^{j_{1}}}\big)^{m_{1}} \ldots f_{i_{d}}(m_{d})\big( \frac{\xi_{N}^{j_{d+1}}}{\xi_{N}^{j_{d}}}\big)^{m_{d}} 
\bigg( \frac{1}{\xi_{N}^{j_{d+1}}} \bigg)^{m}}{m_{1}^{n_{1}} \ldots m_{d}^{n_{d}}}
\in K_{p^{\alpha}N}
\end{equation}
\end{Definition}

\begin{Remark} The slightly more general variant 
$$
\har_{m} \bigg(\begin{array}{c} f_{i_{d+1}},\ldots,f_{i_{1}} \\ \xi_{N}^{j_{d+1}},\ldots,xi_{N}^{j_{1}} \\ n_{d},\ldots,n_{1} \end{array} \bigg) =
\\ m^{n_{d}+\ldots+n_{1}} \sum_{0<m_{1}<\ldots<m_{d}<m} 
\frac{f_{i_{1}}(m_{1})\big( \frac{\xi_{N}^{j_{2}}}{\xi_{N}^{j_{1}}}\big)^{m_{1}} \ldots f_{i_{d}}(m_{d}) \big( \frac{\xi_{N}^{j_{d+1}}}{\xi_{N}^{j_{d}}}\big)^{m_{d}} 
f_{i_{d+1}}(m) \bigg( \frac{1}{z_{i_{d+1}}} \bigg)^{m}}{m_{1}^{n_{1}} \ldots m_{d}^{n_{d}}} $$
\noindent incorporates implicitly the isomorphism $\Eucl_{p^{\alpha}N}$ appearing in \S3.1.4 but we will not use it.
\end{Remark}

\begin{Remark} \label{character chi}Let $G_{p,\alpha,N}$ be the subgroup of $\Hom_{\gp}(K_{p^{\alpha}N}^{\ast},K_{p^{\alpha}N}^{\ast})$ made of elements $\chi$ such that, for all $\epsilon \in K$ such that $|\epsilon|_{p}\leq \frac{1}{p^{\alpha}}$, we have an absolutely convergent expansion $\chi(1+\epsilon) = \sum_{l\geq 0} \chi^{((l))}(1)\epsilon^{l}$ with $(\chi^{((l))}(1))_{l \in \mathbb{N}} \in K^{\mathbb{N}}$. In particular, the elements of $G_{p,\alpha,N}$ are locally analytic maps. Moreover, $G_{p,\alpha,N}$ contains the elements $n_{i} \mapsto \frac{1}{n_{i}^{s_{i}}}$, $s_{i} \in \mathbb{N}^{\ast}$. 
\newline Then, as in part I, most of the computations of the next paragraphs remain true if we replace the maps $m_{i} \mapsto \frac{1}{m_{i}^{n_{i}}}$, $n_{i} \in \mathbb{N}^{\ast}$ by $\chi_{i}(m_{i})$, where $\chi_{1},\ldots,\chi_{d} \in G_{p,\alpha,N}$.
\newline This generalization seems to attain the maximal version of our statements concerning multiple harmonic sums in the $\Sigma$ framework.
\end{Remark}

\noindent The notations concerning generating sequences of multiple harmonic sums in this sense are delayed to the next paragraphs, because they depend on the choice of framework : $\int_{0}^{1}$, $\int_{0}^{z<<1}$ or $\Sigma$.

\subsubsection{Bounds of valuation of prime weighted multiple harmonic sums}

These extend a definition and a lemma of part I.

\begin{Definition} We call prime weighted multiple harmonic sums the weighted multiple harmonic sums as in of Definition \ref{def mhs 1} with $m$ equal to a power of $p$.
\end{Definition}

\begin{Lemma} For any word $w$, for any $\alpha \in \mathbb{N}^{\ast}$, we have $v_{p}(\har_{p^{\alpha}}(w)) \geq \weight(w)$ and, with $w=\bigg(\begin{array}{c} f_{i_{d+1}},\ldots,f_{i_{1}} \\ \xi_{N}^{j_{d+1}},\ldots,xi_{N}^{j_{1}} \\ n_{d},\ldots,n_{1} \end{array} \bigg)$, $w'=\bigg(\begin{array}{c} f'_{i_{d+1}},\ldots,f'_{i_{1}} \\ \xi_{N}^{j_{d+1}},\ldots,xi_{N}^{j_{1}} \\ n_{d},\ldots,n_{1} \end{array} \bigg)$ where $f'_{i} : m \mapsto f_{i}(p^{\alpha-1}m)$,
$$ p^{-\weight(w)}\har_{p^{\alpha}}(w) \equiv \har_{p}(w') \mod p $$
\end{Lemma}

\begin{proof} Same proof with the case of $\mathbb{P}^{1} - \{0,\mu_{N},\infty\}$ : $\har_{p^{\alpha}}(w)$ viewed as an iterated sum indexed by $0<n_{1}<\ldots<n_{d}<p^{\alpha}$ can be splitted into its subsum indexed by $p^{\alpha-1}|n_{1},\ldots,p^{\alpha-1}|n_{d}$, which is equal to $\har_{p}(w)$ and the subsum indexed by the complement of  $\{p^{\alpha-1}|n_{1},\ldots,p^{\alpha-1}|n_{d}\}$, whose valuation is $\geq \weight(w)+1$, since $|z_{1}|_{K} = \ldots = |z_{r+1}|_{K} = 1$.
\end{proof}

\subsection{Algebraic and topological structures on $\pi_{1}^{\un,\DR}(X_{K_{p^{\alpha}N},p^{\alpha}N})$ with a different presentation}

\subsubsection{Notations}

We suppose chosen a basis $b=(f_{1},\ldots,f_{p^{\alpha}})$ of the space of maps $\mathbb{Z}/p^{\alpha}Z \rightarrow K_{p^{\alpha}N}$ which defines by Definition \ref{def forming a basis} a basis $B$ of $H^{1,\DR}(\mathbb{P}^{1} - \{0,\mu_{p^{\alpha}N},\infty\})$.

\begin{Definition}
\noindent \newline i) Let $e_{0 \cup \mu_{p^{\alpha}N}}^{B}$ be the alphabet $\{e_{0},(\tilde{e}_{\xi_{N}^{j},f_{i}}),j=1,\ldots,N,i=1,\ldots,p^{\alpha}\}$.
\newline ii) The weight of a word over $e_{0 \cup \mu_{p^{\alpha}N}}^{B}$ is its number of letters. The depth of such a word is its number of letters different from $e_{0}$.
\newline Let $\mathcal{W}(e_{0 \cup \mu_{p^{\alpha}N}}^{B})$ be the set of words over $e_{0 \cup \mu_{p^{\alpha}N}}^{B}$.
\newline iii) Let $\mathcal{O}^{\sh,e_{0 \cup \mu_{p^{\alpha}N}}^{B}}$ be the shuffle Hopf algebra over the alphabet $e_{0 \cup \mu_{p^{\alpha}N}}^{B}$.
\newline iv) Let $K_{p^{\alpha}N}\langle\langle e_{0 \cup \mu_{p^{\alpha}N}}^{B}\rangle\rangle$ be the non-commutative $K_{p^{\alpha}N}$-algebra of power series over the variables $\{e_{0},(\tilde{e}_{\xi_{N}^{j},f_{i}}),j=1,\ldots,N,i=1,\ldots,p^{\alpha}\}$.
\end{Definition}

\noindent The changes of basis above preserve the weight and the depth.
\noindent We maintain in this case the notation concerning the coefficients of formal power series, namely $f= \sum_{w\text{ word}}f[w]w$ for $f \in  K_{p^{\alpha}N}\langle\langle e_{0 \cup \mu_{p^{\alpha}N}}^{B}\rangle\rangle$. 

\subsubsection{Extension of the adjoint Ihara product}

What follows is an extension of the notion of "adjoint Ihara product" defined in \cite{I-2}. 
We assume that the alphabet $e_{0 \cup \mu_{N}}$ is identified to a set of $K_{p^{\alpha}N}$-linear combinations of the letters of the alphabet $e_{0 \cup \mu_{p^{\alpha}N}}^{B}$.

\begin{Definition} For $g_{\xi_{N}^{1}},\ldots,g_{\xi_{N}^{N}} \in K_{p^{\alpha}N}\langle\langle e_{0 \cup \mu_{p^{\alpha}N}}^{B}\rangle\rangle$ and $f \in K_{N}\langle\langle e_{0\cup \mu_{N}}\rangle\rangle$, let 
$$ (g_{\xi_{N}^{1}},\ldots,g_{\xi_{N}^{N}}) \circ^{\smallint_{0}^{1}}_{U\Lie} f = f(e_{0},g_{\xi_{N}^{1}},\ldots,g_{\xi_{N}^{N}}) $$ 
\end{Definition}

\subsubsection{Topologies on $K_{p^{\alpha}N}\langle\langle e_{0 \cup \mu_{p^{\alpha}N}}^{B}\rangle\rangle$}

We view an element of $K_{p^{\alpha}N}\langle\langle e_{0 \cup \mu_{p^{\alpha}N}}^{B}\rangle\rangle$ as a function $\mathcal{W}(e_{0 \cup \mu_{p^{\alpha}N}}^{B}) \rightarrow K_{p^{\alpha}N}$.

\begin{Definition} i) Let $\mathcal{N}_{\Lambda,D} : K_{p^{\alpha}N}\langle\langle e_{0 \cup \mu_{p^{\alpha}N}}^{B}\rangle\rangle \rightarrow \mathbb{R}_{+}[[\Lambda,D]]$ be the map
	$$ f \mapsto 
	\sum_{(n,d) \in \mathbb{N}^{2}} 
	\max_{w \text{ of weight n and depth d}} \big|f[w]\big|_{p} \Lambda^{n}D^{d} $$
\noindent ii) The $\mathcal{N}_{\Lambda,D}$-topology on $K_{p^{\alpha}N}\langle\langle e_{0 \cup \mu_{p^{\alpha}N}}^{B}\rangle\rangle$ is the topology of simple convergence of functions $\mathcal{W}(e_{0 \cup \mu_{p^{\alpha}N}}) \rightarrow K_{p^{\alpha}N}$. It makes $K_{p^{\alpha}N}\langle\langle e_{0 \cup \mu_{p^{\alpha}N}}^{B}\rangle\rangle$ into a topological $K_{p^{\alpha}N}$-algebra.
\end{Definition}

\begin{Definition} 
i) Let $K_{p^{\alpha}N}\langle\langle e_{0 \cup \mu_{p^{\alpha}N}}^{B}\rangle\rangle_{b} \subset K_{p^{\alpha}N}\langle\langle e_{0 \cup \mu_{p^{\alpha}N}}^{B}\rangle\rangle$ be the subset of elements $f$ such that, for all $d$, the supremum $\sup|f[w]|$ over words $w$ of depth $d$ is finite.
\newline ii) Let $\mathcal{N}_{D} : K_{p^{\alpha}N}\langle\langle e_{0 \cup \mu_{p^{\alpha}N}}^{B}\rangle\rangle_{b} \rightarrow \mathbb{R}_{+}[[D]]$ be the map
$$ f \mapsto \sum_{d \in \mathbb{N}} \max_{w \text{ of depth d}} \big|f[w]\big|_{p} D^{d}$$
\noindent iii) The $\mathcal{N}_{D}$-topology on $K_{p^{\alpha}N}\langle\langle e_{0 \cup \mu_{p^{\alpha}N}}^{B}\rangle\rangle_{b}$ is the topology of convergence of functions $\mathcal{W}(e_{0 \cup \mu_{p^{\alpha}N}}) \rightarrow K_{p^{\alpha}N}$ which is uniform on $\{\text{w word of depth d}\}$ for each $d$. It makes $K_{p^{\alpha}N}\langle\langle e_{0 \cup \mu_{p^{\alpha}N}}^{B}\rangle\rangle_{b}$ into a topological $K_{p^{\alpha}N}$-algebra.
\end{Definition}

\begin{Definition} Let $K_{p^{\alpha}N}\langle\langle e_{0 \cup \mu_{p^{\alpha}N}}^{B}\rangle\rangle_{S} \subset K_{p^{\alpha}N}\langle\langle e_{0 \cup \mu_{p^{\alpha}N}}^{B}\rangle\rangle$ (S stands for summable) be the subspace of elements $f$ such that we have, for all $d$, 
$$ 	\max_{w \text{ of weight n and depth d}} \big|f[w]\big|_{p} \rightarrow_{n \rightarrow \infty} 0 $$ 
\noindent One can equip it with the $\mathcal{N}_{D}$-topology.
\end{Definition}

\begin{Lemma} If we have two basis $B$ and $B'$ such that the decomposition of $B$ in $B$' and the decomposition of $B$ in $B'$ have coefficients of norm $\leq 1$, then one has an isomorphism of topological algebras  $K_{p^{\alpha}N}\langle\langle e_{0 \cup \mu_{p^{\alpha}N}}^{B} \simeq K_{p^{\alpha}N}\langle\langle e_{0 \cup \mu_{p^{\alpha}N}}^{B'}$, and 
$K_{p^{\alpha}N}\langle\langle e_{0 \cup \mu_{p^{\alpha}N}}^{B}\rangle\rangle_{b} \simeq K_{p^{\alpha}N}\langle\langle e_{0 \cup \mu_{p^{\alpha}N}}^{B'}\rangle\rangle_{b}$.
\end{Lemma}

\section{The differential equation of the Frobenius extended to $\pi_{1}^{\un,\DR}(\mathbb{P}^{1} - \{0,\mu_{p^{\alpha}N},\infty\})$}

We suppose chosen $B=(f_{1},\ldots,f_{p^{\alpha}})$ a basis of $H^{1,\DR}(\mathbb{P}^{1} - \{0,\mu_{p^{\alpha}N},\infty\})$ of the type of \S3.1 and we work in the framework of \S3. We show that the differential equation characterizing the Frobenius $\pi_{1}^{\un,\DR}(\mathbb{P}^{1} - \{0,\mu_{N},\infty\})$ can be canonically extended as a structure on $\pi_{1}^{\un,\DR}(\mathbb{P}^{1} - \{0,\mu_{p^{\alpha}N},\infty\})$.

\subsection{Regularization of algebraic iterated integrals over $X_{K_{p^{\alpha}N},p^{\alpha}N}$}

Let again $\rho_{p^{\alpha}}$ be a primitive $p^{\alpha}$-th root of unity. We are going to define regularized the iterated integrals of $\omega_{0}= \frac{dz}{z}, \text{ }\omega_{\rho_{a}^{i}\xi_{N}^{j}}= \frac{dz}{z - \rho_{p^{\alpha}}^{i}\xi_{N}^{j}},\text{ } i=0,\ldots,p^{\alpha}-1,\text{ }j=1,\ldots,N$ as the image of the regularization of part I and the formulas for it by the isomorphisms of \S3.2.
\newline We are going to write the definition and formulas in a more direct way, by adapting a part of \cite{I-1}, \S4.
\newline 
\newline Below, $\mathbb{Z}_{p}^{(N)}= \varprojlim_{l \rightarrow \infty} \mathbb{Z}/ Np^{l}\mathbb{Z} = \amalg_{0\leq m \leq N-1} (m+N\mathbb{Z}_{p})$. The next definition is the  variant of Definition 4.2.1 and Definition 4.2.2 in \cite{I-1} with coefficients in $K_{p^{\alpha}N}$ :

\begin{Definition}
\label{defdel} i) Let $\mathcal{S}_{\alpha}(K_{p^{\alpha}N}) \subset K_{p^{\alpha}N}^{\mathbb{N}}$ be the subspace consisting of the sequences 
$b=(b_{l})_{l \in \mathbb{N}}$ such that there exist $\kappa_{b},\kappa_{b}',\kappa_{b}'' \in \mathbb{R}_{+}^{\ast}$ such that, for all $l \in \mathbb{N}$, we have $|b_{l}|_{p}\leq p^{\kappa_{b} + \frac{\kappa_{b}'}{\log(p)}\log(l+\kappa_{b}'')+(\alpha-1)l}$.
\newline ii) \label{thesubspace} Let $\LAE_{\mathcal{S}_{\alpha},\xi}(\mathbb{Z}_{p}^{(N)},K_{p^{\alpha}N})$ be the space of functions $c : \mathbb{N} \rightarrow L$, such that, for all $n_{0} \in \mathbb{Z}_{p}$ and $i \in \{1,\ldots,N\}$, there exist sequences $(\kappa^{(l,i)}_{n_{0}}(c))_{l\in\mathbb{N}} \in \mathcal{S}_{\alpha}$, such that, for $n \in \mathbb{Z}_{p}$ such that $|n - n_{0}|_{p} \leq p^{-\alpha}$ we have the absolutely convergent series expansion :$c(n) = \sum_{l \in \mathbb{N}} \sum_{i=1}^{N} c^{(l,i)}(n_{0}) \xi^{-in}(n-n_{0})^{l}$. We also denote it by 
$\LAE^{\reg}_{\mathcal{S}_{\alpha},\xi}(\mathbb{Z}_{p}^{(N)} - \{0\},K_{p^{\alpha}N})$.
\end{Definition}

\noindent Below, we identify an iterated integral in $K_{p^{\alpha}N}[[z]]$, and the map $m \in \mathbb{N}^{\ast} \mapsto (\text{coefficient of }z^{m}) \in K_{p^{\alpha}N}$ which it defines. The regularization of certain iterated integrals on $\mathbb{P}^{1} - \{0,\mu_{N},\infty\}$ is defined in \cite{I-1}, \S4 by removing at each step of the process of iterated integration a certain linear combination of the differential forms $\frac{d(z^{p^{\alpha}})}{z^{p^{\alpha}} - \xi_{N}^{jp^{\alpha}}}$, $j=1,\ldots,p^{\alpha}$.

\begin{Proposition-Definition}
\noindent 
The regularization of \cite{I-1}, \S4 can be applied to the iterated integrals of $\omega_{0}= \frac{dz}{z}, \text{ }\omega_{\xi_{N}^{j},f_{i}}= - \xi_{N}^{-j}\sum_{m\geq 0} f_{i}(m \mod p^{\alpha})\xi_{N}^{-jm} z^{m}$.
\newline In other terms, for any word 
$\omega_{0}^{n_{d}-1}\omega_{\xi_{N}^{j_{d}},f_{i_{d}}} \ldots \omega_{0}^{n_{1}-1}\omega_{\xi_{N}^{j_{d}},f_{i_{1}}}$, one can define a regularized iterated integral, which is a $K_{p^{\alpha}N}$-linear combination of iterated integrals and is in the space $\LAE^{\reg}_{\mathcal{S}_{\alpha},\xi}(\mathbb{Z}_{p}^{(N)} - \{0\},L)$.
\end{Proposition-Definition}

\begin{proof} The proof is by induction on the weight using the formulas of \cite{I-1}, \S4. One simply sees that at each step of the induction, because $f_{i}$ depends only on $m \mod p^{\alpha}$ and because it has values of norm $\leq 1$, the process can be applied and gives similar results.
\end{proof}

\subsection{The differential equation of the Frobenius of $\pi_{1}^{\un,\DR}(\mathbb{P}^{1} - \{0,\mu_{N},\infty\})$ extended to $\pi_{1}^{\un,\DR}(\mathbb{P}^{1} - \{0,\mu_{p^{a}N},\infty\})$}

We imitate the differential equation of the Frobenius with the regularization defined above.

\begin{Problem} \label{problem} Do there exist, $(L,\Theta_{\xi_{N}^{1}},\ldots,\Theta_{\xi_{N}^{N}})$, elements of $\mathcal{A}^{\dagger}(U_{N}) \otimes_{K_{N}} K_{p^{\alpha}N} \langle \langle e_{0},(\tilde{e}_{\xi_{N}^{j},f_{i}})_{\substack{i=1,\ldots,p^{\alpha} \\ j=1,\ldots,N}}\rangle\rangle \times K_{p^{\alpha}N} \langle \langle e_{0},(\tilde{e}_{\xi_{N}^{j},f_{i}})_{\substack{i=1,\ldots,p^{\alpha} \\ j=1,\ldots,N}}\rangle\rangle^{N}$ which satisfies
\newline i) the following extension of the differential equation of the Frobenius 
\begin{equation} \label{eq: diff eq Frobenius p alpha N} dL = \bigg( p^{\alpha}\omega_{0}(z)e_{0} + \sum_{\substack{j=1,\ldots,N \\ i=1,\ldots,p^{\alpha}}} \omega_{\xi_{N}^{j},f_{i}}(z) \tilde{e}_{\xi_{N}^{j},f_{i}} \bigg)L - L \bigg( p^{\alpha}\omega_{0}(z)e_{0} + \sum_{j=1}^{N} \omega_{\xi^{p^{\alpha}}}(z^{p^{\alpha}})\Theta_{\xi_{N}^{j}} \bigg)
\end{equation}
\noindent ii) $L(0)=1$
\newline iii) For any word $w$ over $e_{0 \cup \mu_{p^{\alpha}N}}^{B}$, $L[w]$ is in the space $\LAE^{\reg}_{\mathcal{S}_{\alpha},\xi}(\mathbb{Z}_{p}^{(N)} - \{0\},K_{p^{\alpha}N})$ defined in \S3.2.
\end{Problem}

\begin{Remark} For any solution of the problem above, 
the image of $L$ by the canonical projection $A^{\dagger}(U_{N})\langle \langle e_{0 \cup \mu_{p^{\alpha}N}}^{B}\rangle\rangle \rightarrow A^{\dagger}(U_{N})\langle \langle e_{0 \cup \mu_{N}}\rangle\rangle$ is $\Li_{p,\alpha}^{\dagger}(\mu_{N})$ reviewed in \S2.2 : this is because, by \cite{Deligne}, \S11, the Frobenius of $\pi_{1}^{\un,\DR}(X_{K_{N},N})$ is uniquely characterized as the unique isomorphism $\pi_{1}^{\un,\DR}(X_{K_{N},N}) \simlra \pi_{1}^{\un,\DR}(X_{K_{N},N}^{(p^{\alpha})})$ which is horizontal with respect to $(\nabla_{\KZ}^{\mu_{N}},\nabla_{\KZ}^{\mu_{N}^{(p^{\alpha})}})$.
\end{Remark}

\begin{Proposition} The Problem \ref{problem} has a unique solution.
\end{Proposition}

\begin{proof} Let $\int$ be the canonical integration operator from $0$ to a formal variable $z$, which send sequences of differential forms to an element in $K_{p^{\alpha}N}[[z]]$. By induction on the weight, we have to show that, for any weight $n$, the problem has a unique solution up to weight $n-1$ and, for any word $w$ of weight $n$, there is a unique way to define $\Theta_{\xi_{N}^{1}},\ldots,\Theta_{\xi_{N}^{N}}$ in weight $n$ such that the following iterated integral, necessarily equal to $L[w]$, is regular :
\begin{multline} \label{eq:regularity equa} \int \omega_{0} L[\partial_{e_{0}}(w)] + \sum_{\substack{j=1,\ldots,N \\ i=1,\ldots,p^{\alpha}}} \omega_{\xi_{N}^{j},f_{i}}(z) \int L[\partial_{\tilde{e}_{\xi_{N}^{j},f_{i}}  }(w)] 
\\ -  \int p^{\alpha} \omega_{0} L[\tilde{\partial}_{e_{0}}(w)] 
 - \sum_{j=1}^{N} \omega_{\xi_{N}^{jp^{\alpha}}}(z^{p^{\alpha}}) \Theta_{\xi_{N}^{j}}[w] 
- \sum_{\substack{w'w''=w \\ \text{weight}(w')\geq 1}} L[w']\Theta_{\xi_{N}^{j}}[w''] 
\end{multline}
\noindent The singularity at $0$ vanishes by the condition $L(0)=1$. The regularity of (\ref{eq:regularity equa}) amounts to say that the right-hand side of (\ref{eq:regularity equa}) is equal to its image by the regularization of \S4.1. 
\newline In (\ref{eq:regularity equa}), only the term $- \sum \omega_{\xi^{p^{\alpha}}}(z^{p^{\alpha}}) \Theta_{\xi_{N}^{j}}[w]$ has coefficients of $\Theta_{\xi_{N}^{j}}$'s of weight $n$ : all the other terms involve coefficients of $\Theta_{\xi_{N}^{j}}$'s of strictly lower weights only. Since the regularization of \S4.1 is obtained by adding a linear combination $-\sum \omega_{\xi^{p^{\alpha}}}(z^{p^{\alpha}}) C_{\xi}$ characterized in terms of the polar coefficients of the iterated integral, we see that the regularity of  (\ref{eq:regularity equa}) amounts to a unique choice of $(\psi_{\xi_{N}^{1}}[w] ,\ldots,\psi_{\xi_{N}^{N}}[w])$, and thus a unique choice of $L[w]$.
\end{proof}

\subsection{Definitions\label{reformulation}}

\begin{Definition} Let $(\Li_{p,\alpha}^{\dagger,\mu_{p^{\alpha}N}},\psi^{\mu_{p^{\alpha}N}}_{\xi_{N}^{1}},\ldots,\psi^{\mu_{p^{\alpha}N}}_{\xi_{N}^{N}})$ be the unique solution to Problem \ref{problem}.
\end{Definition}

\begin{Definition} Let adjoint $p$-adic multiple zeta values at $p^{\alpha}N$-th roots of unity be the coefficients of $\psi_{1}$, namely :
$$ \zeta_{p,\alpha}^{\Ad,\mu_{p^{\alpha}N}} (w) = \psi_{1}[w] $$
\end{Definition}
 
\begin{Definition} The Frobenius of $\pi_{1}^{\un,\DR}(\mathbb{P}^{1} - \{0,\mu_{N},\infty\})$ extended to $\pi_{1}^{\un,\DR}(\mathbb{P}^{1} - \{0,\mu_{p^{\alpha}N},\infty\})$ is the map  $K_{p^{\alpha}N}[[z]] \langle\langle e_{0 \cup \mu_{N}}\rangle\rangle \rightarrow K_{p^{\alpha}N}[[z]] \langle\langle e_{0 \cup \mu_{p^{\alpha}N}}^{B}\rangle\rangle$ which sends 
$$ f \mapsto \Li_{p,\alpha}^{\dagger}(\mu_{p^{\alpha}N})f(e_{0},\psi_{\xi_{N}^{1}},\ldots,\psi_{\xi_{N}^{N}}) $$
\end{Definition}

\begin{Definition} The Frobenius of $\pi_{1}^{\un,\DR}(\mathbb{P}^{1} - \{0,\mu_{p^{\alpha}N},\infty\},-\vec{1}_{0})$ is the map $f \mapsto f(e_{0},\psi_{\xi_{N}^{1},p^{\alpha}},\ldots, \psi_{\xi_{N}^{N},p^{\alpha}})$, $K_{p^{\alpha}N} \langle\langle e_{0 \cup \mu_{N}}\rangle\rangle \rightarrow K_{p^{\alpha}N} \langle\langle e_{0 \cup \mu_{p^{\alpha}N}}^{B}\rangle\rangle$
\end{Definition}

\noindent We reformulate the definition in a more meaningful way. Let us choose an infinite sequence $\rho_{D} = (f_{i_{\delta}})_{\delta \in \mathbb{N}}$ of elements of $B=(f_{1},\ldots,f_{p^{\alpha}})$.
\newline We associate to such a sequence a linear injection
$$ i_{\rho_{D}} : \begin{array}{c} \mathcal{O}^{\sh,e_{0 \cup \mu_{N}}} \rightarrow \mathcal{O}^{\sh,e_{0 \cup \mu_{p^{\alpha}N}}}
\\ e_{0}^{n_{d}-1}e_{\xi_{N}^{j_{d}}} \ldots e_{0}^{n_{1}-1}e_{\xi_{N}^{j_{1}}} e_{0}^{n_{0}-1} \mapsto e_{0}^{n_{d}-1}e_{\xi_{N}^{j_{d}},f_{i_{d}}} \ldots e_{0}^{n_{1}-1}e_{\xi_{N}^{j_{d}},f_{i_{1}}} \rho_{\delta_{1}} e_{0}^{n_{0}-1} 
\end{array} $$
\noindent We have $\cup_{\rho_{D}}i_{\rho_{D}}(\mathcal{O}^{\sh,e_{0 \cup \mu_{N}}}) = \mathcal{O}^{\sh,e_{0 \cup \mu_{p^{\alpha}N}}^{B}}$. However, $i_{\rho_{D}}(\mathcal{O}^{\sh,e_{0 \cup \mu_{N}}})$ is not stabilized by the shuffle product. The definition of the Frobenius of $X_{K_{p^{\alpha}N},p^{\alpha}N}$ can be interpreted as follows : in the differential equation of the Frobenius of $X_{K_{N},N}$, let us replace the factor $\Li_{p}^{\KZ}$, viewed as a function on $\mathcal{O}^{\sh,e_{0 \cup \mu_{N}}}$, by $\Li \circ i_{\rho_{D}}$ : there is a unique way to transform  $(\Li_{p,\alpha}^{\dagger},\Ad_{\Phi_{p,\alpha}}(e_{1}))$ into variants in order to obtain an equation having regular solutions. This regular equation is the restriction to $i_{\rho_{D}} (\mathcal{O}^{\sh,e_{0 \cup \mu_{N}}})$ differential equation (\ref{eq: diff eq Frobenius p alpha N}). In a summary :
\newline i) for each sequence $\rho_{D}$, we have a "transformed copy" of the Frobenius of $\pi_{1}^{\un,\DR}(X_{K_{N},N})$ inside $\pi_{1}^{\un,\DR}(X_{K_{p^{\alpha}N},p^{\alpha}N})$, which acts on iterated integrals of $X_{K_{N},N}$ transformed by $i_{\rho_{D}}$.
\newline ii) the Frobenius extend to $\pi_{1}^{\un,\DR}(X_{K_{p^{\alpha}N},p^{\alpha}N})$ in the previous sense is just an infinite sequence of "transformed copies" of the Frobenius of $\pi_{1}^{\un,\DR}(X_{K_{p^{\alpha}N},p^{\alpha}N})$, with transformed coeffficients.
\newline 
\newline We can now sketch the definition of (non-adjoint) $p$-adic multiple zeta values at $p^{\alpha}N$-th roots of unity :
\newline - by formalizing the machinery of "transformed copies" described above and applying it to multiple harmonic sums and the explicit formulas for $\zeta_{p,\alpha}$ found in \cite{I-1}, we get their definition
\newline - Alternatively, consider the expression of the coefficients of $\Phi_{p,\alpha}$ in terms of those of $\Phi^{-1}_{p,\alpha}e_{1}\Phi_{p,\alpha}$ obtained in \cite{I-1}. There is a natural way to extend it to power series over the alphabet $e_{0 \cup \mu_{p^{\alpha}N}}^{B}$. This leads to a definition $\zeta_{p,\alpha}$ ; however, one has to check that it does not depends on choices.
\newline - In principle, the equality between all the possible definitions follows from the uniqueness of the solution to Problem \ref{problem}.
\newline 
\newline The $p$-adic multiple zeta values at $p^{\alpha}N$-th roots of unity should be closely related to the "twisted $p$-adic multiple $L$-values" defined in Theorem 3.38 of \cite{FKMT} ; note that in \cite{FKMT}, $p$-adic multiple $L$-value means $p$-adic multiple zeta values at root of unity of order prime to $p$, following a terminiology of Yamashita \cite{Yamashita}, and not values of the $p$-adic multiple $L$-functions introduced in \cite{FKMT}.

\subsection{Combinatorics and bounds of valuation for $\Li_{p,\alpha}^{\dagger_{p,\alpha},\mu_{N}}$}

\noindent We sketch a generalization to $\mathbb{P}^{1} - \{0,\mu_{p^{\alpha}N},\infty\}$ of the solution of the differential equation of the Frobenius explained in \cite{I-1}. The differential equation (\ref{eq: diff eq Frobenius p alpha N}) amounts to
\begin{equation} \label{eq: diff eq Frobenius p alpha N functional} \Li_{p,\alpha}^{\dagger,\mu_{p^{\alpha}N},B} = \Li_{p}^{\KZ(\mu_{p^{\alpha}N})}(e_{0},e_{\xi_{N}^{j},f_{i}})(z) \times \Li_{p}^{\KZ,\mu_{N}}(e_{0},\psi_{\xi_{N}^{1}},\ldots,\psi_{\xi^{N}_{N}})(z^{p^{\alpha}})^{-1} 
\end{equation}
\noindent Let us view $L$ as a linear map $\mathcal{O}^{\sh,e_{0 \cup \mu_{p^{\alpha}N}}} \rightarrow K_{p^{\alpha}N}$. 
\newline 
\newline - In \cite{I-1}, we defined a "regularization by the Frobenius" $\reg_{Frob,\alpha}$ and wrote a formula for it : 
\newline This can be extended immediately into a map  $\reg_{\psi_{\xi_{N}^{1}},\ldots,\psi_{\xi^{N}_{N}}}$ ; the formula for it involves a convolution product of functions over two different shuffle Hopf algebras, and a certain coproduct related to the Goncharov motivic coproduct.
\newline - Then, the regularity condition of $\Li_{p,\alpha}^{\dagger,\mu_{p^{\alpha}N},B}$ amounts to the equation $\int \reg_{\psi_{\xi_{N}^{1}},\ldots,\psi_{\xi^{N}_{N}}} = \int \reg_{X_{K}} \reg_{\psi_{\xi_{N}^{1}},\ldots,\psi_{\xi^{N}_{N}}}$.
\newline - We have explicit formulas for regularized iterated integrals which are very similar to those of \cite{I-1}.
\newline - The bounds of valuations of the coefficients of $\Li_{p,\alpha}^{\dagger,\mu_{N}}$ and $\zeta_{p,\alpha}^{\mu_{N}}$ of the Appendix to Theorem I-1 remain true, with the notions of depth and weight of words over $e_{0 \cup \mu_{p^{\alpha}N}}^{B}$ defined in \S3.1.

\section{The $\int_{0}^{z<<1}$-harmonic Ihara action extended to $\pi_{1}^{\un,\DR}(\mathbb{P}^{1} - \{0,\mu_{p^{\alpha}N},\infty\})$}

Again we suppose given a basis of $H^{1,\DR}(\mathbb{P}^{1} -\{0,\mu_{p^{\alpha}N},\infty\})$ of the type of \S3.1, and we work in the framework of \S3. Using \S4, we extend the $\int_{0}^{z<<1}$-harmonic Ihara action built in \cite{I-2} into a structure on  $\pi_{1}^{\un,\DR}(\mathbb{P}^{1} - \{0,\mu_{p^{\alpha}N},\infty\})$. 

\subsection{The main technical lemma}

We extend to $\mathbb{P}^{1} - \{0,\mu_{p^{\alpha}N},\infty\}$ the main preliminary lemma of this theory :

\begin{Lemma} \label{key lemma}
For any word $w$ over $e_{0 \cup \mu_{p^{\alpha}N}}^{B}$, we have $\Li_{p,\alpha}^{\dagger,\mu_{p^{\alpha}N},B}[e_{0}^{l}w] \rightarrow_{l \rightarrow \infty} 0$.
\end{Lemma}

\begin{proof} Similar to the proof of Lemma 3.2.1 in \cite{I-2}, as a direct consequence of the generalization, explained in \S4.4, of the bounds of valuation of $\Li_{p,\alpha}^{\dagger}$ in \cite{I-1}. More generally, the result is true if we replace the sequence $(e_{0}^{l}w)_{l \in \mathbb{N}}$ by any sequence of words $(w_{l})_{l \in \mathbb{N}}$ such that $\displaystyle \weight(w_{l}) \rightarrow_{l \rightarrow \infty} \infty$ and 
$\displaystyle \limsup_{l \rightarrow \infty}\depth(w_{l}) < \infty$.
\end{proof}

\subsection{Notations for generating series of multiple harmonic sums in the framework $\int_{0}^{z<<1}$}

We take notations for the non-commutative generating series of multiple harmonic sums in the framework $\int_{0}^{z<<1}$.

\begin{Definition} Let $K_{p^{\alpha}N} \langle\langle e_{0 \cup \mu_{p^{\alpha}N}}^{B} \rangle\rangle_{\har}^{\smallint_{0}^{z<<1}} \subset K_{p^{\alpha}N} \langle\langle e_{0 \cup \mu_{p^{\alpha}N}}^{B} \rangle\rangle$ be the vector subspace consisting of the elements $f\in K_{p^{\alpha}N} \langle\langle e_{0 \cup \mu_{p^{\alpha}N}}^{B} \rangle\rangle$ such that, for all words $w$ on $e_{0 \cup \mu_{p^{\alpha}N}}^{B}$, the sequence $(f[e_{0}^{l}w])_{l\in \mathbb{N}}$ is constant.
\end{Definition}

\begin{Definition} Let $(K_{p^{\alpha}N} \langle\langle e_{0 \cup \mu_{p^{\alpha}N}}^{B} \rangle\rangle_{\har}^{\smallint_{0}^{z<<1}})_{S}$ be the subset of elements defined by the condition $\sup_{\substack{w\text{ of weight n} \\ \text{ of depth d}}}  |f[w]|_{K} \rightarrow_{n \rightarrow \infty} 0$.
\end{Definition}

\begin{Definition} i) For all $n \in \mathbb{N}^{\ast}$, let $\har_{n}^{\mu_{p^{\alpha}N},B} = (\har_{n}(w))_{w \in \mathcal{W}_{\har}^{\Sigma}} \in K_{p^{\alpha}N} \langle\langle e_{0 \cup \mu_{p^{\alpha}N}}^{B}\rangle\rangle_{\har}^{\Sigma}$
\newline ii) For all $I \subset \mathbb{N}$, let $\har_{I}^{\mu_{p^{\alpha}N},B} = (\har_{n}^{\mu_{p^{\alpha}N},B})_{n \in I} \in \Map(I, K_{p^{\alpha}N} \langle\langle e_{0 \cup \mu_{p^{\alpha}N}}^{B} \rangle\rangle_{\har}^{\Sigma})$.
\end{Definition}

\subsection{The $\int_{0}^{z<<1}$-harmonic Ihara action and of the $\int_{0}^{z<<1}$-harmonic Frobenius generalized to $\mathbb{P}^{1} - \{0,\mu_{p^{\alpha}N},\infty\}$}

\begin{Definition}
\noindent \newline i) Let $K_{p^{\alpha}N} \langle\langle e_{0 \cup \mu_{p^{\alpha}N}}^{B} \rangle\rangle^{\lim} \subset K_{p^{\alpha}N}\langle\langle e_{0 \cup \mu_{p^{\alpha}N}}^{B} \rangle\rangle$ be the vector subspace consisting of the elements $f\in K_{p^{\alpha}N}\langle\langle e_{0 \cup \mu_{p^{\alpha}N}}^{B} \rangle\rangle$ such that, for all words $w$ on $e_{0 \cup \mu_{p^{\alpha}N}}^{B}$, the sequence $(f[e_{0}^{l}w])_{l\in \mathbb{N}}$ has a limit in $K_{p^{\alpha}N}$ when $l \rightarrow \infty$.
\newline ii) \label{def limit map}We have a map that we call "limit" :
\begin{center} $\lim : K_{p^{\alpha}N} \langle \langle e_{0 \cup \mu_{p^{\alpha}N}}^{B} \rangle\rangle^{\lim} \rightarrow K_{p^{\alpha}N} \langle \langle e_{0 \cup \mu_{p^{\alpha}N}}^{B} \rangle\rangle_{\har}^{\smallint_{0}^{z<<1}} $ \end{center}
\noindent defined by, for all words $w$,
\begin{center} $\displaystyle (\lim f) = \sum_{w\text{ word}} \bigg( \lim_{l \rightarrow \infty} f[e_{0}^{l}w] \bigg) w$
\end{center}
\end{Definition}

\begin{Lemma} \label{lemma continuity limit} The map $\lim$ of Definition \ref{def limit map} is continuous for restriction of the $\mathcal{N}_{D}$-topology on the source and target.
\end{Lemma}

\begin{proof} Clear. \end{proof}

\begin{Proposition-Definition} \label{dR-rt harmonic Ihara action} The map :
$$ \circ^{\smallint_{0}^{z<<1}}_{\har} :
K_{p^{\alpha}N} \langle\langle e_{0 \cup \mu_{p^{\alpha}N}}^{B} \rangle\rangle_{S} \times
\Map(\mathbb{N}, K_{p^{\alpha}N} \langle\langle e_{0 \cup \mu_{N}} \rangle\rangle_{\har}^{\smallint_{0}^{z<<1}})
\rightarrow 
\Map(\mathbb{N},K_{p^{\alpha}N} \langle\langle e_{0 \cup \mu_{p^{\alpha}N}}^{B} \rangle\rangle_{\har}^{\smallint_{0}^{z<<1}})
$$
\noindent defined by the equation
\begin{center} $(g_{\xi_{N}^{1}},\ldots,g_{\xi_{N}^{N}}) \circ_{\har}^{\smallint_{0}^{z<<1}} (n \mapsto h_{n}) = \big(n \mapsto \lim \big( \tau(n)(g_{\xi_{N}^{1}},\ldots,g_{\xi_{N}^{N}}) \circ^{\int_{0}^{1}}_{U\Lie} h_{n} \big)\big)$ \end{center}
\noindent is well-defined. We call it the $\int_{0}^{z<<1}$-harmonic Ihara action of $X_{K_{p^{\alpha}N},p^{\alpha}N}$.
\end{Proposition-Definition}

\begin{Proposition} We have 
$$ \har_{p^{\alpha}\mathbb{N}}^{\mu_{p^{a}N}} = \psi_{1}^{p^{\alpha}N} \text{ } \circ_{\har}^{\smallint_{0}^{z<<1}} \text{ }\har_{\mathbb{N}}^{\mu_{N}^{(p^{\alpha})}} $$
\end{Proposition} 

\begin{proof} In the extended differential equation of the Frobenius rewritten as (\ref{eq: diff eq Frobenius p alpha N functional}), we take the coefficient of degree $p^{\alpha}n$ in the power series expansion at $0$ $(n \in \mathbb{N}^{\ast})$, apply $\tau(n)$ (where $\tau$ is defined in Definition \ref{def of tau}), we specify the equality for a word $e_{0}^{l}\tilde{e}_{b_{j_{d+1}}} e_{0}^{n_{d}-1}\tilde{e}_{b_{j_{d}}} \ldots  e_{0}^{n_{1}-1}\tilde{e}_{b_{j_{1}}}$, and compute the limit $l \rightarrow \infty$ by applying the Lemma \ref{key lemma}.
\end{proof}

\section{The $\Sigma$-harmonic Ihara action extended to $\pi_{1}^{\un,\DR}(\mathbb{P}^{1} - \{0,\mu_{p^{\alpha}N},\infty\})$ \label{Sigma harmonic action}}

\subsection{Basic equations of multiple harmonic sums}

We adapt a few (simple) basic properties of multiple harmonic sums written in \cite{I-2}, \S4 to this more general case. The first property is the "formula of shifting".

\begin{Lemma} \label{shifting} Let $m \in \mathbb{N}$ and $\delta \in \mathbb{N}$ ; we have by definition 
$$\har_{\delta,\delta+m}\bigg( \begin{array}{c} f_{i_{d}},\ldots,f_{i_{1}} \\ \xi_{N}^{j_{d+1}},\ldots,\xi_{N}^{j_{1}} \\ n_{d},\ldots,n_{1} \end{array} \bigg) = \delta^{n_{d}+\ldots+n_{1}}\frac{(\xi_{N}^{j_{1}})^{\delta}}{(\xi_{N}^{j_{d+1}})^{\delta+m}} \sum_{0<m_{1}<\ldots<m_{d}<m} \prod_{d'=1}^{d}  \frac{\bigg(\frac{\xi_{N}^{j_{d'+1}}}{\xi_{N}^{j_{d'}}} \bigg)^{m_{d'}}  f_{i_{d'}}(m_{d'}+\delta)}{
(m_{d'}+\delta)^{-n_{i}}} $$ 
\noindent Assume moreover that $m=p^{\alpha}$, and $\delta= up^{\alpha}$, $u\in \mathbb{N}$ (in particular for all $m' \in [1,m-1]$, we have $v_{p}(m') < v_{p}(\delta)$. Then, we have :

\begin{multline} \har_{\delta,\delta+m}\bigg( \begin{array}{c} f_{i_{d}},\ldots,f_{i_{1}} \\ \xi_{N}^{j_{d+1}},\ldots,\xi_{N}^{j_{1}} \\ n_{d},\ldots,n_{1} \end{array} \bigg) 
= \sum_{l_{1},\ldots,l_{d} \geq 0} \bigg( \prod_{j=1}^{d} \delta^{l_{j}} {-n_{i} \choose s_{i}} \bigg) \text{ }\har_{m} \bigg( \begin{array}{c} f_{i_{d}},\ldots,f_{i_{1}} \\ \xi_{N}^{j_{d+1}},\ldots,\xi_{N}^{j_{1}} \\ n_{d}+l_{d},\ldots,n_{1}+l_{1} \end{array} \bigg) 
\end{multline}
\end{Lemma}

\begin{proof} Clear.
\end{proof}

\noindent The second property is the formula of splitting. $\frak{h}_{m}$ refers to the non-weighted multiple harmonic sums : $\har_{m}(w) = m^{\weight(w)}\frak{h}_{m}(w)$. See \cite{I-2}, \S4 and \S5 for more details on the notations.

\begin{Lemma} \label{splitting} Let $N \in \mathbb{N}^{\ast}$, $\{t_{1},\ldots,t_{r}\} \subset [1,N-1]$ with $t_{1} < \ldots < t_{r}$. We also denote by $t_{0} = 0$ and $t_{r+1} = N$.
\begin{multline} \frak{h}_{m} \bigg( \begin{array}{c} f_{i_{d}},\ldots,f_{i_{1}} \\ z_{i_{d+1}},\ldots,z_{i_{1}} \\ n_{d},\ldots,n_{1} \end{array} \bigg) =
\\  \sum_{\substack{S = \{t_{i_{1}},\ldots,t_{i_{T}}\} \subset \{t_{1},\ldots,t_{r}\}
\\ \imath : S \rightarrow \{1,\ldots,d\}\text{ injective }
	\\ m_{\imath(i_{j})}
\\ S_{0} \amalg \ldots \amalg S_{r} = \{1,\ldots,d\} - f(S) \text{, such that}(\ast)}} 
\prod_{j=1}^{T} \frac{f_{\imath(i_{j})}(m_{\imath(i_{j})}) \mod p^{\alpha})}{m_{\imath(i_{j})}^{n_{\imath(i_{j})}}} \prod_{i=0}^{r} \frak{h}_{t_{i},t_{i+1}}\bigg(\bigg( \begin{array}{c} f_{i_{d}},\ldots,f_{i_{1}} \\ z_{i_{d+1}},\ldots,z_{i_{1}} \\ n_{d},\ldots,n_{1} \end{array} \bigg) \text{restricted to }S_{i} \bigg) 
\end{multline}
\noindent where the condition $(\ast)$ is that $S_{0} \amalg \ldots \amalg S_{r}$ is an increasing connected partition of $\{1,\ldots,d\} - f(S)$, such that 
$S_{i_{j}} \amalg \ldots \amalg S_{i_{j+1}-1}$ is an increasing connected partition of $]f(t_{i_{j}}),f(t_{i_{j+1}}))[$.
\end{Lemma}

\begin{proof} Clear. \end{proof}

\subsection{Notations for generating series of multiple harmonic sums in the framework $\Sigma$ \label{notation mhs sigma}}

In the framework $\Sigma$, we need to take into account both the generating series of multiple harmonic sums and the generating series of localized multiple harmonic sums.

\begin{Definition} i) Let $\mathcal{W}_{\har}^{\Sigma,\mu_{p^{\alpha}N},B}$ be the set of words of the form $\bigg( \begin{array}{c} f_{i_{d}},\ldots,f_{i_{1}} \\ \xi_{N}^{j_{d+1}} ,\ldots,\xi_{N}^{j_{1}} \\ n_{d},\ldots,n_{1} \end{array} \bigg)$, as in equation (\ref{eq:generic multiple harmonic sums}) called $\Sigma$-harmonic words. Let $\mathcal{W}_{\har,\loc}^{\Sigma,\mu_{p^{\alpha}N},B}$ be the larger set of similar words with $n_{d},\ldots,n_{1} \in \mathbb{Z}$, called localized $\Sigma$-harmonic words.
\newline ii) Let $K_{p^{\alpha}N} \langle\langle e_{0 \cup \mu_{p^{\alpha}N}}^{B} \rangle\rangle_{\har}^{\Sigma} = \prod_{w \in \mathcal{W}_{\har}^{\Sigma,\mu_{p^{\alpha}N},B}} K_{p^{\alpha}N}.w$, and 
$K_{p^{\alpha}N} \langle\langle e_{0}^{\pm 1} \cup e_{\mu_{p^{\alpha}N}}^{B}  \rangle\rangle_{\har}^{\Sigma} = \prod_{w \in \mathcal{W}_{\har,\loc}^{\Sigma,\mu_{p^{\alpha}N},B}} K_{p^{\alpha}N}.w$
\newline iii) The coefficient of an element $f \in K_{p^{\alpha}N} \langle\langle e_{0 \cup \mu_{p^{\alpha}N}}^{B} \rangle\rangle_{\har}^{\Sigma}$ in front of a word $w \in \mathcal{W}_{\har}^{\Sigma,\mu_{p^{\alpha}N}}$ is denoted by $f[w]$ (as for the elements of $K \langle\langle e_{z_{0}},\ldots,e_{z_{r}} \rangle\rangle$ in \S2.1). Same notation for the localized variant.
\newline iv) Let $(K_{p^{\alpha}N} \langle\langle e_{0 \cup \mu_{p^{\alpha}N}}^{B} \rangle\rangle_{\har}^{\Sigma})_{S} \subset K \langle\langle e_{0 \cup \mu_{p^{\alpha}N}}^{B} \rangle\rangle_{\har}^{\Sigma}$, resp. $(K \langle\langle e_{0}^{\pm 1} \cup e_{\mu_{p^{\alpha}N}}^{B} \rangle\rangle_{\har}^{\Sigma})_{S} \subset K \langle\langle  e_{0}^{\pm 1} \cup e_{\mu_{p^{\alpha}N}}^{B} \rangle\rangle_{\har}^{\Sigma}$ be the subset of elements $f$ defined by the condition $\sup_{\substack{w\text{ of weight n} \\ \text{ of depth d}}}  |f[w]|_{p}
\rightarrow_{n \rightarrow \infty} 0$.
\end{Definition}

\begin{Definition} i) For all $n \in \mathbb{N}^{\ast}$, let $\har_{n}^{\mu_{p^{\alpha}N},B}= (\har_{n}(w))_{w \in \mathcal{W}_{\har}^{\Sigma,\mu_{p^{\alpha}N},B}} \in K_{p^{\alpha}N} \langle\langle e_{0 \cup \mu_{p^{\alpha}N}}^{B} \rangle\rangle_{\har}^{\Sigma}$ and 
	 $\har_{n}^{\mu_{p^{\alpha}N},B,\loc}= (\har_{n}(w))_{w \in \mathcal{W}_{\har,\loc}^{\Sigma,\mu_{p^{\alpha}N},B}} \in K_{p^{\alpha}N} \langle\langle e_{0}^{\pm 1} \cup e_{\mu_{p^{\alpha}N}}^{B} \rangle\rangle_{\har}^{\Sigma}$ 
\newline ii) For all $I \subset \mathbb{N}$, let $\har_{I}^{\Sigma,\mu_{p^{\alpha}N},B} = (\har_{n}^{\Sigma,\mu_{p^{\alpha}N},B})_{n \in I} \in \Map(I, K_{p^{\alpha}N} \langle\langle e_{0 \cup \mu_{p^{\alpha}N}}^{B} \rangle\rangle_{\har}^{\Sigma})$ and 
$\har_{I,\loc}^{\Sigma,\mu_{p^{\alpha}N},B} = (\har_{n}^{\Sigma,\mu_{p^{\alpha}N},B})_{n \in I} \in \Map(I, K_{p^{\alpha}N} \langle\langle e_{0 \cup \mu_{p^{\alpha}N}}^{B} \rangle\rangle_{\har,\loc}^{\Sigma})$.
\end{Definition}

\begin{Definition} \label{modification of parameters}
	For any map $\varphi : \{\xi_{N}^{1},\ldots,\xi_{N}^{N}\} \rightarrow \{\xi_{N}^{1},\ldots,\xi_{N}^{N}\}$, we define  $\varphi_{\ast}(\har_{n}^{\mu_{N}})$ as the element of 
	$K_{p^{\alpha}N} \langle\langle e_{0 \cup \mu_{N}} \rangle\rangle_{\har}^{\Sigma}$
	\noindent whose coefficient at a word $\big(\begin{array}{c} \xi_{N}^{j_{d+1}},\ldots,\xi_{N}^{j_{1}} \\ n_{d},\ldots,n_{1} \end{array} \big)$ is $\har_{n}\big(\begin{array}{c} \varphi(\xi_{N}^{j_{d+1}}),\ldots,\varphi(\xi_{N}^{j_{1}}) \\ n_{d},\ldots,n_{1} \end{array}\big)$.
\end{Definition}

\subsection{The localized $\Sigma$-harmonic Ihara action at $p^{\alpha}N$-th roots of unity}

\begin{Proposition-Definition} There exists a natural explicit map
$$ (\circ_{\har}^{\Sigma,\mu_{p^{\alpha}N},B})_{\loc}  : (K_{p^{\alpha}N}\langle \langle e_{0 \cup \mu_{p^{\alpha}N}}^{B} \rangle\rangle_{\har}^{\Sigma})_{S} \times \Map(\mathbb{N},K_{p^{\alpha}N}\langle\langle e_{0 \cup \mu_{N}} \rangle\rangle_{\har}^{\Sigma}) \rightarrow \Map(\mathbb{N},K_{p^{\alpha}N}\langle \langle e_{0 \cup \mu_{p^{\alpha}N}}^{B} \rangle\rangle) $$
	\noindent which satisfies the equation 
$$ \har_{p^{\alpha}\mathbb{N}}^{\mu_{p^{\alpha}N},B}  = \har_{p^{\alpha}}^{\mu_{p^{\alpha}N},B} \circ  (z \mapsto z^{p^{\alpha}})_{\ast}\har_{n}^{\mu_{N},\loc} $$
\end{Proposition-Definition}

\begin{proof} The proof is the same with the proposition-definition of $ (\circ_{\har}^{\Sigma})_{\loc}$ for $X_{K_{N},N}$ in \cite{I-2}, \S5. We start with $\har_{p^{\alpha}m}(w)$, for a word $w$, written as an iterated sum indexed by $0<m_{1}< \ldots < m_{d}<m$, and we introduce the Euclidean division of $m_{i}$ by $p^{\alpha}$ : $m_{i} = p^{\alpha} u_{i} + r_{i}$ ; we write, for $n_{i}\in \mathbb{N}^{\ast}$, $m_{i}^{-n_{i}} = r_{i}^{-n_{i}}(\frac{ p^{\alpha} u_{i}}{r_{i}} +1)^{-n_{i}} = r_{i}^{-n_{i}} \sum_{l_{i}\geq 0} {-n_{i} \choose l_{i}} \big( \frac{ p^{\alpha} u_{i}}{r_{i}} \big)^{l_{i}}$. This gives the result, via the formula of shifting (Lemma \ref{shifting}) and the formula of splitting (Lemma \ref{splitting}).
\newline If $n \in \mathbb{N}^{\ast}$ has the Euclidean division by $p^{\alpha}$ defined by $n= p^{\alpha}u + r$, we have, for all $i=1,\ldots,p^{\alpha}$ and $j=1,\ldots,N$, $f_{i}(n \mod p^{\alpha}) \xi_{N}^{jn}= (f_{i}(r) \xi_{N}^{jr}) (\xi_{N}^{jp^{\alpha}})^{u}$ : we see that only $\xi_{N}^{jp^{\alpha}}$ apperas in the part of the computation which depends on $u$ ; this is why the computation works in the same way with $X_{K_{N},N}$.
\end{proof}

\subsection{Conclusion}

In \cite{I-2}, \S5, we have constructed a map of "elimination of the localization" ; this is a natural map
$$ \elim : \Map(\mathbb{N},K_{p^{\alpha}N}\langle \langle e_{0 \cup \mu_{p^{\alpha}N}}^{B} \rangle\rangle) \rightarrow \Map(\mathbb{N},K_{p^{\alpha}N} \langle\langle e_{0}^{\pm 1} \cup e_{ \mu_{p^{\alpha}N}}^{B} \rangle\rangle) $$
\noindent which satisfies
$$ \elim (\har_{\mathbb{N}}) = \har_{\mathbb{N}}^{\loc} $$
\noindent (The terminology refers for the dual of $\elim$, which indeed suppresses the localization.)

\begin{Proposition-Definition} Let 
$$  \circ_{\har}^{\Sigma,\mu_{p^{\alpha}N},B} = \circ_{\har,\loc}^{\Sigma,\mu_{p^{\alpha}N},B} \circ (\id \times \elim) $$
\noindent We call it the $\Sigma$-harmonic Ihara action of $X_{K_{p^{\alpha}N},p^{\alpha}N}$. We have : 
$$ \har_{p^{\alpha}\mathbb{N}}^{\mu_{p^{a}N},B} = \har_{p^{\alpha}}^{\mu_{p^{\alpha}N},B} \text{ } \circ_{\har}^{\Sigma} \text{ }(z \mapsto z^{p^{\alpha}})_{\ast}\har_{\mathbb{N}}^{\mu_{N}}$$
\end{Proposition-Definition}

\begin{proof} Consequence of \S6.3 and of the property of $\elim$ reviewed above.
\end{proof}

\begin{Remark} It can be sometimes more natural to consider the dual of $\circ_{\har}^{\Sigma}$. We call it the $\Sigma$-harmonic Ihara coaction of $X_{K_{p^{\alpha}N},p^{\alpha}N}$ ; same terminology for the duals of the other harmonic Ihara actions.
\end{Remark}

\begin{Definition} i) Let $\comp^{\smallint \leftarrow \Sigma,\mu_{p^{\alpha}N},B} : (K_{p^{\alpha}N} \langle\langle e_{0 \cup \mu_{p^{\alpha}N}}^{B} \rangle\rangle_{\har}^{\Sigma})_{S} \rightarrow K_{p^{\alpha}N} \langle\langle e_{0 \cup \mu_{p^{\alpha}N}}^{B} \rangle\rangle_{S}$ be defined by : 
$\comp^{\smallint \leftarrow \Sigma,\mu_{p^{\alpha}N},B} [e_{0}^{l}\tilde{e}_{f_{i_{d+1}}\xi_{N}^{j_{d+1}}} e_{0}^{n_{d}-1}\tilde{e}_{f_{i_{d}}\xi_{N}^{j_{d}}}\ldots e_{0}^{n_{1}-1}\tilde{e}_{f_{i_{1}}\xi_{N}^{j_{1}}}]$ is the coefficient of $n^{l}\har_{n}(\emptyset)$ in the coefficient of $\bigg(\begin{array}{c} f_{i_{d}},\ldots,f_{i_{1}} 
\\ \xi_{N}^{j_{d}},\ldots,\xi_{N}^{j_{1}} \\ n_{d},\ldots,n_{1} \end{array} \bigg)$ of the equation (\ref{eq:generic harmonic Ihara action}).
\newline ii) Let $\comp^{\Sigma \leftarrow \smallint,\mu_{p^{\alpha}N},B} : K_{p^{\alpha}N} \langle\langle e_{0 \cup \mu_{p^{\alpha}N}}^{B} \rangle\rangle_{S} \rightarrow (K_{p^{\alpha}N}\langle\langle e_{0 \cup \mu_{p^{\alpha}N}}^{B} \rangle\rangle_{\har}^{\Sigma})_{S}$ be the map defined by
\newline $\comp^{\Sigma \leftarrow \smallint,\mu_{p^{\alpha}N},B}(f) \bigg[ \bigg(\begin{array}{c} f_{i_{d}},\ldots,f_{i_{1}} 
\\ \xi_{N}^{j_{d}},\ldots,\xi_{N}^{j_{1}} \\ n_{d},\ldots,n_{1} \end{array} \bigg) \bigg] = \sum_{l\geq 0} f[e_{0}^{l}\tilde{e}_{f_{i_{d+1}}\xi_{N}^{j_{d+1}}} e_{0}^{n_{d}-1}\tilde{e}_{f_{i_{d}}\xi_{N}^{j_{d}}}\ldots e_{0}^{n_{1}-1}\tilde{e}_{f_{i_{1}}\xi_{N}^{j_{1}}}]$.
\newline We call them the comparison maps between sums and integrals.
\end{Definition}

\begin{Proposition} We have $\comp^{\Sigma \leftarrow \smallint} \circ \comp^{\smallint \leftarrow \Sigma} = \id$.
\end{Proposition}

\begin{proof} This follows from the proof of the analog statement in \cite{I-2}, \S5 : it is indeed a consequence of a property of the coefficients $\mathcal{B}$ of the elimination of the localization.
\end{proof}

\subsection{Reindexation and extension to $\Map(\mathbb{N},K_{N}\langle\langle e_{\{0\}\cup\mu_{N}} \rangle\rangle_{\har})$}

We have defined $\circ_{\har}^{\Sigma}$ as a map 
$$ 
K_{p^{\alpha}N}\langle\langle e^{B}_{\{0\} \cup \mu_{p^{\alpha}N}} \rangle\rangle_{S}
\times \Map(\mathbb{N},
K_{p^{\alpha}N} \langle\langle e_{\{0\} \cup \mu_{N}} \rangle\rangle_{\har}^{\Sigma} )
\rightarrow 
\Map(\mathbb{N},K_{p^{\alpha}N}\langle\langle  e^{B}_{\{0\} \cup \mu_{p^{\alpha}N}}  \rangle\rangle_{\har}^{\Sigma}) $$
\noindent However, because of the equation involving multiple harmonic sums above, we can also view it as a map
$$ 
K_{p^{\alpha}N}\langle\langle e^{B}_{\{0\} \cup \mu_{p^{\alpha}N}} \rangle\rangle_{S}
\times \Map(\mathbb{N},
K_{p^{\alpha}N} \langle\langle e_{\{0\} \cup \mu_{N}} \rangle\rangle_{\har}^{\Sigma} )
\rightarrow 
\Map(p^{\alpha}\mathbb{N},K_{p^{\alpha}N}\langle\langle  e^{B}_{\{0\} \cup \mu_{p^{\alpha}N}}  \rangle\rangle_{\har}^{\Sigma}) $$
\noindent Let us adopt momentaneously this point of view.

\begin{Proposition-Definition} Under this point of view, the $\Sigma$-harmonic Ihara action of $X_{K_{p^{\alpha}N},p^{\alpha}N}$ extends to an explicit map : 
	$$ 
	K_{p^{\alpha}N}\langle\langle e_{\{0\} \cup \mu_{p^{\alpha}N}} \rangle\rangle_{S}
	\times \Map(\mathbb{N},
	K_{p^{\alpha}N} \langle\langle e_{\{0\} \cup \mu_{N}} \rangle\rangle_{\har}^{\Sigma} )
	\rightarrow 
	\Map(\mathbb{N},K_{p^{\alpha}N}\langle\langle  e_{\{0\} \cup \mu_{p^{\alpha}N}}  \rangle\rangle_{\har}^{\Sigma}) $$
\end{Proposition-Definition}

\begin{proof} We write the Euclidean division of $n$ by $p^{\alpha}$ : $n=p^{\alpha}u+r$. We write the formula of splitting of multiple harmonic sums at $p^{\alpha}u$ (Proposition \ref{shifting}). Then, we apply the $\Sigma$-harmonic Ihara action of $X_{K_{p^{\alpha}N},p^{\alpha}N}$ to the factors equal to multiple harmonic sums having bounds $(0,p^{\alpha}u)$ ; we apply the $p$-adic formula of shifting (Proposition \ref{splitting}) to the factors having bounds $p^{\alpha}u$ and $p^{\alpha}u+r$.
\end{proof}

\section{The $\int_{0}^{1}$-harmonic Ihara action extended to $\pi_{1}^{\un,\DR}(\mathbb{P}^{1} - \{0,\mu_{p^{\alpha}N},\infty\})$}

This paragraph is technically independent from the other ones, although it is motivated by \S\ref{Sigma harmonic action}. We take again the setting of \S3.
\newline 
\newline In \cite{I-3}, the definition and basic properties of of $\circ_{\har}^{\smallint_{0}^{1}}$ are established purely in the De Rham setting ; they require nothing crystalline. They only require the Ihara product, which is defined on $\Pi_{1,0}  =\pi_{1}^{\un,\DR}(\mathbb{P}^{1} -\{0,\mu_{p^{\alpha}N},\infty\},-\vec{1}_{1},\vec{1}_{0})$ by \cite{Deligne Goncharov}, \S5.
\newline 
\newline Let $K_{p^{\alpha}N}\langle \langle e_{0 \cup \mu_{p^{\alpha}N}}^{B} \rangle\rangle_{S'} \subset K_{p^{\alpha}N}\langle \langle  e_{0 \cup \mu_{p^{\alpha}N}}^{B} \rangle\rangle$ be the submodule of elements $h$ such that, for all words $w$, the series $\sum_{l\in \mathbb{N}} h[e_{0}^{l}w]$ is convergent ($S'$ stands for summable). It contains in particular $\tilde{\Pi}_{1,0}(K_{p^{\alpha}N})_{S'}$ and its image by $\Ad(e_{1}) : Z \mapsto Z^{-1}e_{1}Z$. We extend a map of comparison defined in \S6.

\begin{Definition} Let $\comp^{\Sigma \leftarrow \smallint,\mu_{p^{\alpha}N},B} :
\begin{array}{c} K_{p^{\alpha}N}\langle \langle e_{0 \cup \mu_{p^{\alpha}N}}^{B} \rangle\rangle_{S'} \rightarrow K_{p^{\alpha}N} \langle\langle e^{B}_{0 \cup \mu_{p^{\alpha}N}} \rangle\rangle_{\Sigma'}
\\ h \mapsto \sum_{\tilde{w}} h \big[ \frac{1}{1-e_{0}}
\tilde{w} \big] \tilde{w} \end{array}$ where the sum $\sum_{\tilde{w}}$ is over the words $\tilde{w}$ of the form $\tilde{e}_{f_{i_{d+1}}\xi_{N}^{j_{d+1}}}e_{0}^{n_{d}-1}\tilde{e}_{f_{i_{d}}\xi_{N}^{j_{d}}} \ldots e_{0}^{n_{1}-1}\tilde{e}_{f_{i_{1}}\xi_{N}^{j_{1}}}e_{0}^{n_{0}-1}$, $j_{1},\ldots,j_{d+1} \in \{1,\ldots,N\}$, $i_{1},\ldots,i_{d+1} \in \{1,\ldots,p^{\alpha}\}$, $n_{0},\ldots,n_{d} \in \mathbb{N}^{\ast}$.
\end{Definition}

\begin{Lemma} The image of $\comp^{\Sigma \leftarrow \smallint,\mu_{p^{\alpha}N},B}$ is equal to 
$K_{p^{\alpha}N} \langle \langle e_{0 \cup \mu_{p^{\alpha}N}}^{B} \rangle \rangle^{\smallint_{0}^{1}}_{\har} \subset K_{p^{\alpha}N} \langle\langle e_{0 \cup \mu_{p^{\alpha}N}}^{B} \rangle\rangle$ defined as the subset of the elements $f$ such that we have $f[e_{0}^{l}w] = 0$ for all $l \in \mathbb{N}^{\ast}$ and for all words $w$.
\end{Lemma}

\begin{proof} Indeed, each element of $K_{p^{\alpha}N} \langle \langle e_{0 \cup \mu_{p^{\alpha}N}}^{B} \rangle \rangle^{\smallint_{0}^{1}}_{\har}$ is in $K_{p^{\alpha}N} \langle\langle e_{0 \cup \mu_{p^{\alpha}N}}^{B} \rangle\rangle_{\Sigma'}$ and equal to its own image by $\comp^{\Sigma \leftarrow \smallint,\mu_{p^{\alpha}N},B}$.
\end{proof}

\begin{Proposition-Definition} \label{def De Rham Ihara}There exists a unique map 
\begin{center}
$\tilde{\Pi}_{1,0}(K_{p^{\alpha}N})_{\Sigma} \times K_{p^{\alpha}N} \langle\langle e_{0 \cup \mu_{p^{\alpha}N}}^{B}  \rangle\rangle^{\smallint_{0}^{1}}_{\har} \rightarrow K_{p^{\alpha}N} \langle\langle  e_{0 \cup \mu_{p^{\alpha}N}}^{B}  \rangle\rangle^{\smallint_{0}^{1}}_{\har}$.
\end{center}
\noindent making the following diagram commutative :
$$
\begin{array}{ccccc}
\tilde{\Pi}_{1,0}(K_{p^{\alpha}N})_{\Sigma} \times \Ad_{\tilde{\Pi}_{1,0}(K_{p^{\alpha}N})_{\Sigma}}(e_{1}) && 
\overset{\circ_{\Ad}^{\smallint_{0}^{1}} \circ (\Ad(e_{1}) \times \id)}{\longrightarrow}
&& \Ad_{\tilde{\Pi}_{1,0}(K_{p^{\alpha}N})_{S'}}(e_{1}) \\
\downarrow{\id \times \comp^{\Sigma \leftarrow \smallint,\mu_{p^{\alpha}N},B}} &&&& \downarrow{ \comp^{\Sigma \leftarrow \smallint,\mu_{p^{\alpha}N},B}} \\
\tilde{\Pi}_{1,0}(K_{p^{\alpha}N})_{\Sigma}(e_{1}) \times K_{p^{\alpha}N} \langle\langle  e_{0 \cup \mu_{p^{\alpha}N}}^{B}  \rangle\rangle^{\smallint_{0}^{1}}_{\har} && \overset{}{\longrightarrow} && K_{p^{\alpha}N} \langle\langle  e_{0 \cup \mu_{p^{\alpha}N}}^{B}  \rangle\rangle^{\smallint_{0}^{1}}_{\har}
\end{array}
$$
\noindent We call it the $\int_{0}^{1}$-harmonic Ihara action of $X_{K_{p^{\alpha}N},p^{\alpha}N}$ and denote it by $\circ_{\har}^{\smallint_{0}^{1}}$. 
\end{Proposition-Definition}

\begin{proof} Let us consider the coefficient $\comp^{\Sigma \leftarrow \smallint,\mu_{p^{\alpha}N},B} ( \Ad_{g}(e_{1}) \circ_{\Ad}^{\smallint_{0}^{1}} \Ad_{f}(e_{1}))[w]$ at any word $w$, viewed as a function of $g$ and $f$. By the formula for the dual of the adjoint Ihara action, one can see that its dependence on $f$ factorizes in a natural way through $\comp^{\Sigma\leftarrow \smallint,\mu_{p^{\alpha}N},B}(f)$. This defines $\circ_{\har}^{\smallint_{0}^{1}}$ and one can check that this map is unique.
\end{proof}

\noindent The $\int_{0}^{1}$-harmonic Ihara action is thus $\comp^{\Sigma\leftarrow \smallint,\mu_{p^{\alpha}N},B}_{\ast} \bigg(\Ad(e_{1})_{\ast}\bigg($ Ihara product$\bigg)\bigg)$ (where the push-forward refers to the set which is acted upon), characterized by the equation :
\begin{equation} \label{eq:char harmonic Ihara action} \comp^{\Sigma\leftarrow \smallint}( \Ad_{g}(e_{1}) \circ_{\Ad}^{\smallint_{0}^{1}} \Ad_{f}(e_{1})) = g \circ_{\har}^{\smallint_{0}^{1}} \comp^{\Sigma\leftarrow \smallint}(\Ad_{f}(e_{1}))
\end{equation}
\noindent One can also write an extension of where $\circ_{\Ad}^{\smallint_{0}^{1}}$ is replaced by $\circ_{U \Lie}^{\smallint_{0}^{1}}$ defined in \S3.3.2.

\section{The iteration of the harmonic Frobenius extended to $\pi_{1}^{\un,\DR}(\mathbb{P}^{1} - \{0,\mu_{p^{\alpha}N},\infty\})$}

We take again the setting of \S3, a basis $B$ of $H^{1,\DR}(\mathbb{P}^{1} - \{0,\mu_{p^{\alpha}N},\infty\})$ being chosen.

\subsection{Introduction}

We have studied in \cite{I-3} how the Frobenius iterated $\alpha$ times depends on $\alpha$ viewed as a $p$-adic integer. Unlike for $X_{K_{N},N}$, in the present case, the coefficients of the multiple harmonic sums which we consider depend on $p$ and $\alpha$, via the functions $f_{1},\ldots,f_{p^{\alpha}}$. It is thus less easy than in the usual case to study the limit $\alpha \rightarrow \infty$.
\newline In \cite{I-3}, given formal variables $\textbf{a}$, $\Lambda$, we have defined a map :

$$ (\widetilde{\text{iter}}_{\har}^{\Sigma})^{\textbf{a},\Lambda} : (K_{N}\langle\langle e_{\{0\}\cup\mu_{N}} \rangle \rangle_{\har}^{\Sigma})_{S} \rightarrow K_{N}[[\Lambda^{\textbf{a}}]][\textbf{a}](\Lambda)\langle\langle e_{\{0\}\cup\mu_{N}} \rangle\rangle^{\Sigma}_{\har} $$
\noindent such that, the map $(\text{iter}_{\har}^{\Sigma})^{\frac{\alpha}{\alpha_{0}},q^{\tilde{\alpha}_{0}}} : (K_{N} \langle\langle e_{\{0\}\cup\mu_{N}} \rangle\rangle_{\har}^{\Sigma})_{S} \rightarrow K_{N}\langle\langle e_{\{0\}\cup\mu_{N}}\rangle\rangle^{\smallint_{0}^{1}}_{\har}$ defined as the composition of $(\widetilde{\text{iter}}_{\har}^{\Sigma})^{\textbf{a},\Lambda}$ by the reduction modulo  $(\textbf{a}-\frac{\tilde{\alpha}}{\tilde{\alpha}_{0}},\Lambda-q^{\tilde{\alpha}_{0}})$, satisfies,
\begin{equation} \label{eq:third of I-3}\har_{q^{\tilde{\alpha}}} = (\iter_{\har}^{\Sigma})^{\frac{\tilde{\alpha}}{\tilde{\alpha}_{0}},q^{\tilde{\alpha}_{0}}} (\har_{q^{\tilde{\alpha}_{0}}})
\end{equation}
\noindent It was obtained by writing by hand, using elementary operations on multiple harmonic sums viewed as $p$-adic functions, an expression of $\har_{q^{\tilde{\alpha}}}$ in terms of $\har_{q^{\tilde{\alpha}_{0}}}$ for $(\tilde{\alpha}_{0},\tilde{\alpha}) \in (\mathbb{N}^{\ast})^{2}$ such that $\tilde{\alpha}_{0}|\tilde{\alpha}$. 

\subsection{Generalizations to $X_{X_{p^{a}N},p^{a}N}$ \label{5 point 2}}

The proof of the part of Theorem-Definition V-1.a concerning the maps of iteration of the harmonic Frobenius follows from going back to the proofs in \cite{I-3}, and making again the same observations : the computations can be adapted to our more general setting because the maps $f_{1},\ldots,f_{p^{\alpha}}$, applied to elements $m_{1}<\ldots<m_{d}\in \mathbb{N}$, depend only on the remainders of $m_{1},\ldots,m_{d}$ modulo $q^{\tilde{\alpha}} = p^{\alpha}$ and have coefficients of norm $\leq 1$.

\section{Application to the relation between direct and indirect methods to compute the Frobenius}

We sketch why the regularization of iterated integrals defined in \S4 can be computed in terms of the iteration of the harmonic Frobenius defined in \S8, and that the restriction of these maps to a certain subspaces has coefficients expressed in terms of $\pi_{1}^{\un,\DR}(X_{K_{N},N})$. This gives a connection between the formulas of the direct computation of the Frobenius \cite{I-1} and the ones of the indirect computation of the Frobenius \cite{I-2} \cite{I-3} and an interpretation of this connection in terms of a descent of the Frobenius extended to $\pi_{1}^{\un,\DR}(X_{K_{p^{\alpha}N},p^{\alpha}N})$.

\subsection{A subspace and descent to $\pi_{1}^{\un,\DR}(\mathbb{P}^{1} - \{0,\mu_{N},\infty\})$}

We sketch the proof of Proposition V-1.d stated in \S1.7.

\begin{Definition}
Let $e_{0 \cup \mu_{N}\cup \mu_{N}^{(p^{\alpha})}}$ be the set of letters formed by $e_{0}$,
$e_{\xi_{N}^{j}}$, $j=1,\ldots,N$, and 
$e_{(\xi_{N}^{j})^{(p^{\alpha})}}$, $j\in \{1,\ldots,N\}$, $r \in \{0,\ldots,p^{\alpha}-1\}$.
\end{Definition}

\noindent By associating $e_{z_{i}}$ to $\frac{dz}{z-z_{i}}$, and $e_{z_{i}}^{r \mod p^{\alpha}}$ to $\frac{p^{\alpha} z^{p^{\alpha}-1}dz}{z^{p^{\alpha}}-z_{i}^{p^{\alpha}}}$, we get a specific type of multiple harmonic sums.

\begin{Definition} \label{def mhs 2} We now assume $X_{K} = X_{K_{N},N}$ and $K=K_{N}$, with $N \in \mathbb{N}^{\ast}$ prime to $p$. The weighted multiple harmonic sums at roots of unity with congruences the numbers, for $d \in \mathbb{N}^{\ast}$, $n_{d},\ldots,n_{1} \in \mathbb{N}^{\ast}$, $j_{1},\ldots,j_{d+1} \in \{1,\ldots,N\}$, and $I,I' \subset \{1,\ldots,d\}$, $n \in \mathbb{N}^{\ast}$, 
\begin{equation}
\har_{m,I}^{\mod p^{\alpha}} \big( \begin{array}{c}  z_{i_{d+1}},\ldots,z_{i_{1}} \\ n_{d},\ldots,n_{1} \end{array} \big) = 
m^{n_{1}+\ldots+n_{d}}
\sum_{\substack{0=m_{0}<m_{1}<\ldots<m_{d}<m
\\ \text{ for }i \in I , n_{i-1} \equiv n_{i} \mod p^{\alpha}}} 
\bigg( \frac{1}{z_{j_{d+1}}} \bigg)^{m}
\prod_{j=1}^{d} \frac{ \bigg( \frac{z_{j_{i+1}}}{z_{j_{i}}} \bigg)^{m_{j}}}{m_{j}^{n_{j}}}
\end{equation}
\end{Definition}

\noindent We can now sketch the proof of Proposition V-1.d.

\begin{proof} (sketch) The result is obtained by redoing the constructions of the previous paragraphs with this particular type of multiple harmonic sums. It comes from how the conditions of the type $m_{i} \equiv m_{i-1} \mod p^{\alpha}$ are expressed in terms of the reduction of $n_{i}$ and $m_{i-1}$ modulo $p^{\alpha}$ : $m_{i} \equiv m_{i-1} \mod p^{\alpha}$ is equivalent to $r_{i} = r_{i-1}$ where $r_{i}$ and $r_{i-1}$ are the remainders of, respectively, $m_{i}$ and $m_{i-1}$ modulo $p^{\alpha}$.
\end{proof}

\subsection{Iteration of the harmonic Frobenius at $p^{\alpha}N$-th roots of unity and regularization at $p^{\alpha}N$-th roots of unity}

We sketch the proof of Proposition V-1.e stated in \S1.7.
\newline 
\newline The idea behind it is to play on the ambiguity of the parameter $q^{\tilde{\alpha}}$ in $\har_{q^{\tilde{\alpha}}}$ :
\newline - $q^{\tilde{\alpha}}$ is the upper bound of the domain of summation of  $\har_{q^{\tilde{\alpha}}}$ viewed as an iterated sum ; the question of studying the maps $m \in \mathbb{N}^{\ast} \mapsto \har_{m}(w) \in K$ with $m$ and $\har_{m}(w)$ viewed as $p$-adic numbers has appeared a lot in part I, via the study of the regularized version $m \in \mathbb{N}^{\ast} \mapsto \har^{\dagger_{p,\alpha}}_{m}(w) \in K$, generalized here in \S4.
\newline - $q^{\tilde{\alpha}}$ also expresses via $\tilde{\alpha}$ the number of iterations of the Frobenius ; the question of studying the iterated Frobenius as a function of the number of iterations has appeared in \cite{I-3}, and this has been generalized here in \S8.

\begin{proof} (sketch) The Proposition follows from the uniqueness of some power series expansions.
\end{proof}

\subsection{Interpretation in terms of the computation of the Frobenius}

The strategy of computation of the Frobenius in \cite{I-1} has three steps :
\newline 
\newline i) \emph{Formal algebraic solution to the differential equation of the Frobenius (\ref{eq:horizontality equation}).}
\newline By an immediate induction on the weight, equation (\ref{eq:horizontality equation}) amounts to say that each function $\Li_{p,\alpha}^{\dagger}[w]$ is a $K_{N}$-linear combination of iterated integrals of the differential forms (\ref{eq:differential forms}), where the coefficients of the linear combination are numbers $\Phi_{p,\alpha}[w']$ with $\weight(w')<\weight(w)$.
\newline We write a close formula for this property ; it involves the convolution product on the shuffle Hopf algebra as well as some combinatorics of another Hopf algebra related to a motivic Hopf algebra. We show that, in this close formula, each iterated integral of (\ref{eq:differential forms}) can be replaced by a regularized version. We also show that the close formula, which is initially inductive with respect to the weight, can actually be made inductive with respect to the depth, which is more efficient algorithmically.
\newline 
\newline ii) \emph{Explicit computation and bounds of valuation for each regularized iterated integrals of (\ref{eq:differential forms}).} 
\newline We find a formula for these functions which is inductive with respect to the depth. We show that they lie in a very small subspace of $\frak{A}^{\dagger}(U_{N})$, defined in terms of $p$-adically and weight-adically convergent infinite sums of prime weighted multiple harmonic sums (the prime weighted multiple harmonic sums are the weighted multiple harmonic sums $\har_{n}(w)$ such that $n=p^{\alpha}$ where $\alpha$ is the number of iterations of the Frobenius). We prove at the same time some bounds of valuations of the values of these functions.
\newline 
\newline iii) \emph{Reformulation of equation (\ref{eq:algebraic Frobenius equation}).}
\newline We rewrite equation (\ref{eq:algebraic Frobenius equation}) as a way to obtain the numbers $\zeta_{p,\alpha}(w)$ as $\mathbb{Q}$-linear combinations of the numbers  $\Li^{\dagger}_{p,\alpha}[w']$ which is, first, compatible with the depth filtration, and secondly, which involves rational coefficients having not too big $p$-adic norms.
\newline 
\newline Bringing together steps i), ii), iii), we deduce, first, an explicit formula for $(\Li_{p,\alpha}^{\dagger},\zeta_{p,\alpha})$ written through an induction on the depth, and, secondly, bounds of valuations for the functions $\Li_{p,\alpha}^{\dagger}[w]$ (Here we are refering to the canonical structure of the algebra on the rigid analytic functions on $U_{N}$ as a complete normed $K_{N}$-algebra) and for the numbers $\zeta_{p,\alpha}(w)$.
\newline 
\newline We see that the step of regularization can be done purely in terms of prime weighted multiple harmonic sums ; conversely, certain coefficients of multiple harmonic sums are actually equal to certain coefficients of iterated integrals.

\begin{Remark} This point of view has also the advantage of providing a canonical regularization. In I-1, \S4, the regularization of $p$-adic iterated integrals depended on choices.
\end{Remark}

\begin{Remark} This point of view also gives an interpretation of the formulas with parameters sketched in \cite{I-1}, \S6.3, c).
\end{Remark}

\appendix 

\section{$\pi_{1}^{\un,\DR}(\mathbb{P}^{1} - \{z_{0},z_{1},\ldots,z_{r},\infty\})$ and $p$-adic pseudo adjoint multiple polylogarithms}

Let $X_{K}$ be $\mathbb{P}^{1} - \{0,z_{1},\ldots,z_{r},\infty\}$ over a complete ultrametric normed field $K \supset \mathbb{Q}_{p}$ with $z_{1},\ldots,z_{r} \in K$ of norm $1$. We partially generalize to $\pi_{1}^{\un,\DR}(X_{K})$ the operations on multiple harmonic sums defined in part I and the previous paragraphs.

\subsection{Setting}

The weighted multiple harmonic sums under consideration are those of equation (\ref{eq:generic multiple harmonic sums}).
\newline Their localized variants are obtained by allowing the parameters $n_{1},\ldots,n_{d}$ in (\ref{eq:generic multiple harmonic sums}) to be any elements of $\mathbb{Z}$. They can be generalized as functions depending on locally analytic group homomorphisms $K^{\ast} \rightarrow K^{\ast}$, as in Remark \ref{character chi}.
\newline Let us formalize the indices of multiple harmonic sums :

\begin{Definition} Let $\mathcal{W}_{\har}^{\Sigma,X_{K}}$ be the set of words of the form $\big( \begin{array}{c} z_{j_{d+1}} ,\ldots,z_{j_{1}} \\ n_{d},\ldots,n_{1} \end{array} \big)$, as in equation (\ref{eq:generic multiple harmonic sums}) called $\Sigma$-harmonic words.
\newline Let $\mathcal{W}_{\har,\loc}^{\Sigma,X_{K}}$ be the larger set of sequences $\big( \begin{array}{c} z_{j_{d+1}},\ldots,z_{j_{1}} \\ n_{d},\ldots,n_{1} \end{array} \big)$, with $n_{d},\ldots,n_{1} \in \mathbb{Z}$, called localized $\Sigma$-harmonic words.
\newline We say that $\big( \begin{array}{c} z_{j_{d+1}},\ldots,z_{j_{1}} \\ n_{d},\ldots,n_{1} \end{array} \big)$ has weight $n_{d}+\ldots+n_{1}$ and depth $d$.
\end{Definition}

\noindent Let us formalize the spaces containing the generating sequences of multiple harmonic sums, and subspaces of "summable" elements :

\begin{Definition} i) Let $K \langle\langle e_{z_{0}},\ldots,e_{z_{r}} \rangle\rangle_{\har}^{\Sigma,X_{K}} = \prod_{w \in \mathcal{W}_{\har}^{\Sigma,X_{K}}} K.w$
\newline Let $K \langle\langle e_{0}^{\pm 1},e_{z_{1}},\ldots,e_{z_{r}} \rangle\rangle_{\har,\loc}^{\Sigma,X_{K}} = \prod_{w \in \mathcal{W}_{\har,\loc}^{\Sigma,X_{K}}} K.w$
\newline ii) The coefficient of an element $f \in K \langle\langle e_{z_{0}},\ldots,e_{z_{r}} \rangle\rangle_{\har}^{\Sigma}$ in front of a word $w \in \mathcal{W}_{\har}^{\Sigma}$ is denoted by $f[w]$ (as for the elements of $K \langle\langle e_{z_{0}},\ldots,e_{z_{r}} \rangle\rangle$ in \S2.1). Same notation for $f \in (K \langle\langle e_{0}^{\pm 1},e_{z_{1}},\ldots,e_{z_{r}} \rangle\rangle_{\har}^{\Sigma}$.
\newline iii) Let $(K \langle\langle e_{z_{0}},\ldots,e_{z_{r}} \rangle\rangle_{\har}^{\Sigma})_{S} \subset K \langle\langle e_{z_{0}},\ldots,e_{z_{r}} \rangle\rangle_{\har}^{\Sigma}$, resp. $(K \langle\langle e_{0}^{\pm 1},e_{z_{1}},\ldots,e_{z_{r}} \rangle\rangle_{\har}^{\Sigma})_{S} \subset K \langle\langle e_{0}^{\pm 1},e_{z_{1}},\ldots,e_{z_{r}} \rangle\rangle_{\har}^{\Sigma}$ be the subset of elements defined by the condition $\displaystyle\sup_{\substack{w\text{ of weight s} \\ \text{ of depth d}}}  |f[w]|_{K}
\rightarrow_{s \rightarrow \infty} 0$.
\end{Definition}

\noindent Let us formalize the generating sequences of multiple harmonic sums :

\begin{Definition} i) For all $n \in \mathbb{N}^{\ast}$, let $\har_{n}^{X_{K}} = (\har_{n}(w))_{w \in \mathcal{W}_{\har}^{\Sigma,X_{K}}} \in K \langle\langle e_{z_{0}},\ldots,e_{z_{r}} \rangle\rangle_{\har}^{\Sigma,X_{K}}$ and
\newline $\har_{n}^{X_{K},\loc} = (\har_{n}(w))_{w \in \mathcal{W}_{\har,\loc}^{\Sigma}} \in K \langle\langle e_{0}^{\pm 1},e_{z_{1}},\ldots,e_{z_{r}} \rangle\rangle_{\har}^{\Sigma,X_{K}}$ 
\newline ii) For all $I \subset \mathbb{N}$, let $\har_{I}^{X_{K}} = (\har_{n}(X))_{n \in I}$ viewed as a map $I \rightarrow K \langle\langle e_{z_{0}},\ldots,e_{z_{r}} \rangle\rangle_{\har}^{\Sigma,X_{K}}$, 
and
\newline $\har_{I}^{X_{K},\loc} = (\har_{n}^{X_{K},\loc})_{n \in I}$ viewed as a map $I \rightarrow K \langle\langle e_{0}^{\pm 1},e_{z_{1}},\ldots,e_{z_{r}} \rangle\rangle_{\har}^{\Sigma,X_{K}}$.
\end{Definition}

\subsection{The localized $\Sigma$-harmonic Ihara action}

Generalizing the localized $\Sigma$-harmonic Ihara action to this setting is simple. Below, the notation $(z \mapsto z^{p^{\alpha}})_{\ast}$ refers to Definition \ref{modification of parameters}.

\begin{Proposition-Definition} There exists a natural explicit map, generalizing the localized $\Sigma$-harmonic Ihara action of \cite{I-2} and \S6, 
$$ (\circ_{\har}^{\Sigma,X_{K}})_{\loc} : (K\langle \langle e_{z_{0}},\ldots,e_{z_{r}} \rangle\rangle_{\har}^{\Sigma})_{S} \times \Map(\mathbb{N},K\langle \langle e_{0}^{\pm 1},e_{z_{1}},\ldots,e_{z_{r}} \rangle\rangle_{\har}^{\Sigma}) \rightarrow \Map(\mathbb{N},K\langle\langle e_{z_{0}},\ldots,e_{z_{r}} \rangle\rangle) $$
\noindent such that we have
$$ \har_{p^{\alpha}\mathbb{N}}^{X_{K}} = \har_{p^{\alpha}}^{X_{K}} \text{ } (\circ_{\har}^{\Sigma,X_{K}})_{\loc}\text{ }(z \mapsto z^{p^{\alpha}})_{\ast}(\har_{\mathbb{N}}^{X_{K},\loc}) $$
\end{Proposition-Definition}

\begin{proof} This is the same computation with in \S6 and \cite{I-2} \S5.
\end{proof}

\subsection{The elimination of the localization}

Unlike the previous step, the step of elimination of the localization is more subtle in the generic case than in the case of roots of unity. There is, first, a naive way to write it :

\begin{Proposition-Definition} \label{prop def elim generic} There exists a natural explicit map, generalizing the elimination of the localization in \cite{I-2} and \S6,
$$ \elim^{X_{K}} : \Map(\mathbb{N},K\langle \langle e_{0},e_{z_{1}},\ldots,e_{z_{r}} \rangle\rangle) \rightarrow \Map(\mathbb{N},K \langle\langle e_{0}^{\pm 1},e_{z_{1}},\ldots,e_{z_{r}} \rangle\rangle) $$
\noindent such that we have
$$ \elim^{X_{K}} (\har_{\mathbb{N}}^{X_{K}}) =  \har^{X_{K},\loc}_{\mathbb{N}} $$
\noindent Moreover, $\elim$ commutes with the map $(z \mapsto z^{p^{\alpha}})_{\ast}$ and more generally with all the maps of the type of Definition \ref{modification of parameters}. One has an explicit formula for $\elim^{X_{K}}$ involving the analogs of the coefficients $\mathcal{B}$ defined in \cite{I-2}, but depending on $z_{1},\ldots,z_{r}$, instead of roots of unity.
\end{Proposition-Definition}

\begin{proof} This is the same computation with in \S6 and \cite{I-2} \S5.
\end{proof}

\noindent In \S6, in order to define $\circ_{\har}^{\Sigma}$ for $\mathbb{P}^{1} - \{0,\mu_{p^{\alpha}N},\infty\}$, we had to apply $elim$ only in the case of roots of unity of order prime to $p$. This is not the case anymore here ; and we do not have satisfactory bounds of valuations of the coefficients $\mathcal{B}$ in our more general case. In the case of roots of unity of order prime to $p$ we had the following fact leading to satisfactory bounds of valuations of the coefficients $\mathcal{B}$ : for each product $z$ of elements of $\{z_{1},\ldots,z_{r}\}$, either $z=1$ or $|\frac{1}{z}|_{p}=|\frac{1}{z-1}|_{p}=1$. This does not remain true if $\{z_{1},\ldots,z_{r}\}$ is not included in the set of roots of unity of order prime to $p$ : at least one $|\frac{1}{z-1}|_{p}$ is bigger. This prevents us from defining  $\circ_{\har}^{\Sigma,X_{K}}$ as $(\circ_{\har}^{\Sigma,X_{K}})_{\loc} \circ (\id \times \elim^{X_{K}})$ as we did in \cite{I-2} for $\mathbb{P}^{1} - \{0,\mu_{N},\infty\}$ and in \S6 for $\mathbb{P}^{1} - \{0,\mu_{p^{\alpha}N},\infty\}$ : it would involve divergent series. Thus, we will express differently the elimination of the localization.
\newline 
\newline For any $x \in \mathbb{C}_{p}$ such that $|x|_{p}=1$, let $\omega(x)$ be the unique root of unity of order prime to $p$ in $\mathbb{C}_{p}$ such that $|x-\omega(x)|_{p}<1$. This defines a morphism of multiplicative groups from $\mathcal{O}_{\mathbb{C}_{p}}^{\times}$ to $\underset{p\nmid N}{\cup} \mu_{N}(\mathbb{C}_{p})$. Let us assume that either $K$ is a sub-topological field of $\mathbb{C}_{p}$, or $K$ contains $\mathbb{C}_{p}$ as a sub-topological field and that $z_{1},\ldots,z_{r}$ are chosen such that for all $j \in \{1,\ldots,r\}$, there exists a (necessarily unique) root of unity in $\mathbb{C}_{p}$ of order prime to $p$, which we denote again by $\omega(z_{j})$, such that $|z_{j} - \omega(z_{j})|_{p} < 1$. We have a morphism of multiplicative groups from the subgroup of $K^{\times}$ generated by the $z_{j}$'s to $\underset{p\nmid N}{\cup} \mu_{N}(\mathbb{C}_{p})$. We will then write 
$$ z_{j} = \omega(z_{j}) + \epsilon_{j} $$
\noindent and expand the result in terms of powers of $\epsilon_{j}$'s. We will first replace $\epsilon_{1},\ldots,\epsilon_{r}$ by formal variables $E_{1},\ldots,E_{r}$ in order to neutralize the possible problems of convergence. We let for convenience $y_{i} = \frac{1}{z_{i}}$ for $i=1,\ldots,d$. The multiple harmonic sums under consideration become :

\begin{Definition} Let, for $n_{d},\ldots,n_{1} \in \mathbb{Z}$, and $j_{d+1},\ldots,j_{1} \in \{1,\ldots,r\}$,
\begin{multline} \label{eq:formal E multiple harmonic sums}
\har^{E_{d},\ldots,E_{1}}_{m} \bigg(\begin{array}{c} z_{j_{d+1}},\ldots,z_{j_{1}} \\ n_{d},\ldots,n_{1} \end{array} \bigg) 
\\ = m^{n_{d}+\ldots+n_{1}} \sum_{0<m_{1}<\ldots<m_{d}<m} \frac{
\big( \frac{\omega(y_{j_{1}})+E_{j_{1}}}{\omega(y_{j_{2}})+E_{j_{2}}}\big)^{m_{1}} \ldots \big(\frac{\omega(z_{j_{d}})+E_{j_{d}}}{\omega(z_{j_{d+1}})+E_{j_{d+1}}}\big)^{m_{d}} 
\big( \omega(y_{j_{d+1}})+E_{j_{d+1}} \big)^{m}}{m_{1}^{n_{1}} \ldots m_{d}^{n_{d}}} \in \mathbb{Q}_{p}^{\unr}[E_{1},\ldots,E_{r}]
\end{multline}
\noindent Let 
$\har^{E_{d},\ldots,E_{1}}_{m,\loc}$ be their generating sequence and let 
$\har^{E_{d},\ldots,E_{1}}_{\loc}$ be the map 
$m \mapsto \har^{E_{d},\ldots,E_{1}}_{m,\loc}$. Same notations without $\loc$ for the restrictions to $n_{d},\ldots,n_{1}\geq 1$.
\end{Definition}

\noindent (\ref{eq:formal E multiple harmonic sums}) has an expansion as a polynomial of the $E_{j}$'s (the sum below is finite) :

\begin{Lemma} 
We have an expression of the form :
$$ 
\har^{E_{d},\ldots,E_{1}}_{m} \bigg(\begin{array}{c} z_{j_{d+1}},\ldots,z_{j_{1}} \\ n_{d},\ldots,n_{1} \end{array} \bigg)  = \sum_{l_{1},\ldots,l_{d}\geq 0}  \prod_{r=1}^{d}\frac{1}{l_{r}!}
\bigg( \frac{E_{j_{r}}}{\omega(y_{i_{r}})} \bigg)^{l_{r}} \times \bigg(
\begin{array}{l} \mathbb{Z}\text{-linear combination of localized weighted} \\ \text{multiple harmonic sums at roots of unity}
\\ \text{of order prime to p, } (\omega(y_{j_{1}}),\ldots,\omega(y_{j_{d+1}}))
\\ \text{multiplied by a positive power of m} 
\end{array} \bigg) $$	
\end{Lemma}

\begin{proof} By expanding the powers of $\omega(y_{j_{r}})+E_{j_{r}}$, we get first that $\har^{E_{d},\ldots,E_{1}}_{m} \bigg(\begin{array}{c} z_{j_{d+1}},\ldots,z_{j_{1}} \\ n_{d},\ldots,n_{1} \end{array} \bigg)$ is equal to $m^{n_{1}+\ldots+n_{d}}$ multiplied by 
	$$  
	\sum_{l_{1},\ldots,l_{d}\geq 0} E_{j_{1}}^{l_{1}}\ldots E_{j_{d}}^{l_{d}} \sum_{0<m_{1}<\ldots<m_{d}<m}  {m_{1}-0 \choose l_{1}} \ldots {m-m_{d} \choose l_{d+1}} 
	\frac{\omega(y_{i_{1}})^{m_{1}-l_{1}} \omega(y_{i_{2}})^{m_{2}-m_{1}-l_{2}}  \ldots \omega(y_{i_{d+1}})^{m-m_{d}-l_{d}}   }{m_{1}^{n_{1}} \ldots m_{d}^{n_{d}}} $$
\noindent which, by rewriting the binomial coefficients, is of the following form, with $P_{l_{1},\ldots,l_{d}} \in \mathbb{Z}[M_{1},\ldots,M_{d}]$ :
$$
\sum_{l_{1},\ldots,l_{d}\geq 0}  \bigg( \prod_{r=1}^{d}\frac{E_{j_{r}}^{l_{r}}}{\omega(y_{i_{r}})^{l_{r}}   l_{r}!}\bigg) 
\sum_{0<m_{1}<\ldots<m_{d}<m}  P_{l_{1},\ldots,l_{d}}(m_{1},\ldots,m_{d}) 
\frac{\omega(y_{i_{1}})^{m_{1}} \omega(y_{i_{2}})^{m_{2}-m_{1}}  \ldots \omega(y_{i_{d+1}})^{m-m_{d}}}{m_{1}^{n_{1}} \ldots m_{d}^{n_{d}}}   $$
\noindent We obtain the result by writing each $P_{l_{1},\ldots,l_{d}}$ as a sum of monomials.
\end{proof}

\noindent We deduce another way to write the elimination of the localization :

\begin{Proposition-Definition} 
There exists an explicit map (the source and target of $\elim^{E_{1},\ldots,E_{r}}$ are defined as in A.1)
$$ \elim^{E_{1},\ldots,E_{r}} : \mathbb{Q}_{p}^{\unr}[E_{1},\ldots,E_{r}] \langle\langle e_{0 \cup (\cup_{p\nmid N} \mu_{N})} \rangle\rangle_{\har}^{\Sigma}  \rightarrow \mathbb{Q}_{p}^{\unr}[E_{1},\ldots,E_{r}] \langle \langle e_{0}^{\pm 1} \cup e_{\cup_{p\nmid N} \mu_{N}} \rangle\rangle_{\har,\loc}^{\Sigma} $$
\noindent such that we have 
$$ \elim^{E_{1},\ldots,E_{r}} (\har^{E_{1},\ldots,E_{r}}) =  \har^{E_{1},\ldots,E_{r}}_{\loc} $$
\end{Proposition-Definition} 

\begin{proof} This follows from the result of elimination of the localization for multiple harmonic sums at roots of unity of order prime to $p$ (\cite{I-2}, \S5), and the previous lemma, noting that the coefficients in the previous lemma are in $\mathbb{Z}$, and thus do not affect the convergence of the series. 
\end{proof}

\noindent It is also possible to define this map in a more universal way which is not relative to $X_{K}$ but involves one variable $E_{\xi}$ per root of unity $\xi$ of order prime to $p$.

\subsection{The $\Sigma$-harmonic Ihara action and related objects}

In this paragraph, we assume that $|z_{j} - \omega(z_{j})|_{p} < p^{-\frac{1}{p-1}}$ for all $j$. We can now generalize the $\Sigma$-harmonic Ihara action. 

\begin{Proposition-Definition} If $|\epsilon_{j}|_{p} < p^{-\frac{1}{p-1}}$ for all $j$, \label{regularized harmonic Ihara action} then
$$ \circ_{\har}^{\Sigma,X_{K}} = (\circ_{\har}^{\Sigma,X_{K}})_{\loc} \text{ }\circ\text{ } (\id \times \elim^{\epsilon_{1},\ldots,\epsilon_{r}}) $$
\noindent is well-defined and we have
\begin{equation} \label{eq:generic harmonic Ihara action}\har_{p^{\alpha}\mathbb{N}}^{X_{K}} = \har_{p^{\alpha}}^{X_{K}} \text{ } \circ_{\har}^{\Sigma,X_{K}}\text{ }(z \mapsto z^{p^{\alpha}})_{\ast}(\har_{\mathbb{N}}^{\mu_{N}}) \end{equation}
\noindent where $N$, prime to $p$, is the lcm of the orders of $\omega(z_{j})$'s as roots of unity.
\end{Proposition-Definition}

\begin{proof} The well-definedness is a consequence of A.3,  and the equation (\ref{eq:generic harmonic Ihara action}) is a direct consequence of A.2, A.3 and the fact that power series $\sum_{l \geq 0} \frac{x^{l}}{l!}$ is convergent in $\mathbb{Q}_{p}$ for $x \in \mathbb{Q}_{p}$ such that  $v_{p}(x)>\frac{1}{p-1}$.
\end{proof}

\begin{Remark} 
\noindent \newline i) In this point of view, the coefficient of $(\epsilon_{1} \ldots \epsilon_{r})^{0}$ looks like a "regularization" (in a quite unusual sense : regularization "by roots of unity of order prime to $p$") of the $\Sigma$-harmonic Ihara action.
\newline ii) The results of \S6, concerning $X_{K_{p^{\alpha}N},p^{\alpha}N}$, give the particular case where $\epsilon_{1}=\ldots= \epsilon_{d}=0$.
\newline iii) Thus in fine, this result is a sort of analytic interpretation, where "analytic" refers to functions of $\epsilon_{1},\ldots,\epsilon_{d}$, of the definition of the $\Sigma$-harmonic Ihara action for $X_{K_{p^{\alpha}N},p^{\alpha}N}$ in \S6.
\end{Remark}

\noindent The coefficients of $\circ_{\har}^{\Sigma,X_{K}}$ are not iterated integrals a priori : but, by analogy with the case of roots of unity, we will call them pseudo - iterated integrals, and define : 

\begin{Definition} i) Let $\comp^{\smallint \leftarrow \Sigma}(X_{K}) : (K\langle\langle e_{z_{0}},\ldots,e_{z_{r}} \rangle\rangle_{\har}^{\Sigma})_{S} \rightarrow K \langle\langle e_{z_{0}},\ldots,e_{z_{r}} \rangle\rangle_{S}$ be defined by : 
$\comp^{\smallint \leftarrow \Sigma}(X_{K})[e_{0}^{l}e_{z_{j_{d+1}}}e_{0}^{n_{d}-1}e_{z_{j_{d}}}\ldots e_{0}^{n_{1}-1}e_{z_{j_{1}}}]$ is the coefficient of $n^{l}\har_{n}(\emptyset)$ in the coefficient of 
\newline $\big(\begin{array}{c} z_{j_{d+1}},\ldots,z_{j_{1}} \\ n_{d},\ldots,n_{1} \end{array} \big)$ of the equation (\ref{eq:generic harmonic Ihara action}).
\newline ii) Let $\comp^{\Sigma \leftarrow \smallint}(X_{K}) : K \langle\langle e_{z_{0}},\ldots,e_{z_{r}} \rangle\rangle_{S} \rightarrow (K\langle\langle e_{z_{0}},\ldots,e_{z_{r}} \rangle\rangle_{\har}^{\Sigma})_{S}$ be the map defined by
\newline $\comp^{\Sigma \leftarrow \smallint}(X_{K})(f)[\big(\begin{array}{c} z_{j_{d+1}},\ldots,z_{j_{1}} \\ n_{d},\ldots,n_{1} \end{array} \big)] = \sum_{l\geq 0} f[e_{0}^{l}e_{z_{j_{d+1}}} e_{0}^{n_{d}-1}e_{z_{j_{d}}}\ldots e_{0}^{n_{1}-1}e_{z_{j_{1}}}]$.
\newline We call them the maps of comparison between sums and pseudo-integrals.
\end{Definition}

\begin{Proposition} \label{inversibility of formulas} We have $\comp^{\Sigma \leftarrow \smallint} \circ \comp^{\smallint \leftarrow \Sigma} = \id$.
\end{Proposition}

\begin{proof} Follows from the Proposition-Definition \ref{regularized harmonic Ihara action} and the proof for roots of unity in \S6 and in \cite{I-2} \S5.
\end{proof}

\noindent The following definition generalizes the definition of adjoint $p$-adic multiple zeta values at roots of unity and $\Sigma$-harmonic Frobenius \cite{I-2}, and their generalizations to $X_{K_{p^{\alpha}N},p^{\alpha}N}$ of \S5.

\begin{Definition} i) Let $\text{PseudoLi}^{\Ad}_{p,\alpha}(X_{K}) = (\comp^{\Sigma \leftarrow \smallint})(X_{K})( \har_{p^{\alpha}}^{X_{K}} \in K\langle\langle e_{z_{0}},e_{z_{1}},\ldots,e_{z_{r}} \rangle\rangle$. We call $p$-adic pseudo adjoint hyperlogarithms the coefficients of $\text{PseudoLi}^{\Ad}_{p,\alpha}(X_{K})$.
\newline ii) We call $\Sigma$-harmonic pseudo-Frobenius the map 
$$ \begin{array}{ccc} K\langle\langle e_{z_{0}},\ldots,e_{z_{r}} \rangle\rangle_{\har}^{\Sigma} \rightarrow  K\langle\langle e_{z_{0}},\ldots,e_{z_{r}} \rangle\rangle_{\har}^{\Sigma}
\\ h \mapsto \text{PseudoLi}^{\Ad}_{p,\alpha}(X_{K}) (\circ^{\Sigma,X_{K}}_{\har}) \text{ }h 
\end{array} $$
\end{Definition}

\noindent The next statement generalizes the expansion of multiple harmonic sums at roots of unity in terms of $p$-adic multiple zeta values at roots of unity (Corollary I-2.a  in \cite{I-2}) and its generalization to $X_{K_{p^{\alpha}N},p^{\alpha}N}$ in \S5.

\begin{Proposition} We have $\har_{p^{\alpha}}^{X_{K}} = \comp^{\Sigma \leftarrow \smallint} \text{PseudoLi}^{\Ad}_{p,\alpha}(X_{K})$, i.e., for all words,
$$ \har_{p^{\alpha}} \big(\begin{array}{c} z_{j_{d+1}},\ldots,z_{j_{1}} \\ n_{d},\ldots,n_{1} \end{array} \big) = \sum_{l\geq 0} \text{PseudoLi}^{\Ad}_{p,\alpha} \big(\begin{array}{c} z_{j_{d+1}},\ldots,z_{j_{1}} \\ l,n_{d},\ldots,n_{1} \end{array} \big)(X_{K}) $$
\end{Proposition}

\begin{proof} By Proposition \ref{inversibility of formulas}. \end{proof}

\subsection{$z_{1},\ldots,z_{r}$ viewed as variables : $p$-adic pseudo adjoint multiple polylogarithms}
 
Our point of view of replacing $\epsilon_{1},\ldots,\epsilon_{r}$ by formal parameters and viewing them as perturbations makes better sense if we allow $\epsilon_{1},\ldots,\epsilon_{r}$ to vary. According to Goncharov \cite{Goncharov}, the functions obtained from hyperlogarithms (in the sense of Proposition-Definition \ref{prop connexion}) by making $z_{1},\ldots,z_{r}$ into variables are called multiple polylogarithms. We are going to express how the $p$-adic pseudo adjoint hyperlogarithms of \S A.2 depend on $z_{1},\ldots,z_{r}$. Let $U_{K}$ be the set of $z \in K$ such that there exists a (ncessarily unique) root of unity $\eta$ of order prime to $p$ such that $|z-\eta|_{K}<p^{\frac{-1}{p-1}}$. For $r \in \mathbb{N}^{\ast}$, $U_{K}^{r}$ is equipped with the product topology.

\begin{Proposition-Definition} The coefficients of the following map are locally analytic functions (of $(y_{1},\ldots,y_{r}) = (\frac{1}{z_{1}},\ldots,\frac{1}{z_{r}})$) :
$$ \text{PseudoLi}_{p,\alpha}^{\Ad} : \begin{array}{ccc} U_{K}^{r} & \rightarrow & K \langle\langle e_{z_{0}},\ldots,e_{z_{r}} \rangle\rangle_{\har}^{\Sigma} \\ 
(z_{1},\ldots,z_{r}) & \mapsto & \text{PseudoLi}_{p,\alpha}^{\Ad}(\mathbb{P}^{1} - \{0,z_{1},\ldots,z_{r},\infty\}/K) \end{array} $$
\noindent We call its coefficients the \emph{$p$-adic pseudo adjoint multiple polylogarithms}. (Pseudo $\Ad p$MPL's)
\end{Proposition-Definition}

\begin{proof} They are convergent infinite sums of products of locally constant functions by rational functions having poles only at $z_{i}=0$.
\end{proof}

\noindent The values of Pseudo $\Ad p$MPL's at tuples of roots of unity are the $\Ad p$MZV$\mu_{p^{\alpha}N}$'s defined in this paper ; their values at tuples of roots of unity of order prime to $p$ are the $\Ad p$MZV$\mu_{N}$'s.

\end{document}